\documentclass[a4paper,11pt,twoside]{article}
\usepackage{amsmath,amssymb,amscd}
\usepackage{ascmac}
\usepackage{enumerate}
\usepackage{theorem}

\newtheorem{thm}{Theorem}[section]
\newtheorem{prop}[thm]{Proposition}
\newtheorem{cor}[thm]{Corollary}
\newtheorem{lem}[thm]{Lemma}
\newtheorem{conj}[thm]{Conjecture}
{\theorembodyfont{\normalfont}
\newtheorem{rem}[thm]{Remark}
\newtheorem{rems}[thm]{Remarks}
\newtheorem{dfn}[thm]{Definition}

}

\makeatletter
\newenvironment{proof}[1][\proofname]{\par
  \normalfont
  \topsep6\p@\@plus6\p@ \trivlist
  \item[\hskip\labelsep{\textit{\mdseries #1}\@addpunct{\mdseries.}}]\ignorespaces
}{%
  \QED \endtrivlist
}
\newcommand{\proofname}{\normalfont{\textit{Proof.}}}
\makeatother
\makeatletter
\def\BOXSYMBOL{\RIfM@\bgroup\else$\bgroup\aftergroup$\fi
  \vcenter{\hrule\hbox{\vrule height.85em\kern.6em\vrule}\hrule}\egroup}
\makeatother
\newcommand{\BOX}{%
  \ifmmode\else\leavevmode\unskip\penalty9999\hbox{}\nobreak\hfill\fi
  \quad\hbox{\BOXSYMBOL}}
\newcommand\QED{\BOX}

\numberwithin{equation}{section}

\newtheorem{thm-dfn}[thm]{Theorem-Definition}

\newtheorem{thmalph}{Theorem}[section]
\newtheorem{coralph}[thmalph]{Corollary}

\def\Ker{\mathop{\mathrm{Ker}}\nolimits}

\def\Hom{\mathop{\mathrm{Hom}}\nolimits}

\def\End{\mathop{\mathrm{End}}\nolimits}

\def\supp{\mathop{\mathrm{supp}}\nolimits}
\def\ind{\mathop{\mathrm{ind}}\nolimits}

\def\tr{\mathop{\mathrm{tr}}\nolimits}
\def\id{\mathop{\mathrm{id}}\nolimits}

\newcommand{\bb}[1]{\mathbb{#1}}

\def\tr{\mathop{\mathrm{tr}}\nolimits}
\newcommand{\C}{\mathbb{C}}
\newcommand{\R}{\mathbb{R}}
\newcommand{\Z}{\mathbb{Z}}
\newcommand{\N}{\mathbb{N}}
\newcommand{\Q}{\mathbb{Q}}
\newcommand{\ep}{\varepsilon}
\newcommand{\dif}{\mathrm{d}}

\newcommand{\map}[3]{{#1}\colon{#2}\to{#3}}
\newcommand{\brackets}[1]{ \left[ {#1} \right]}
\newcommand{\braces}[1]{ \left\{ {#1} \right\}}
\newcommand{\parens}[1]{ \left( {#1} \right)}
\newcommand{\angles}[1]{ \left \langle {#1} \right \rangle}
\newcommand{\norm}[1]{ \left \Vert {#1} \right \Vert}

\newcommand{\gm}{\gamma}
\newcommand{\gminv}{\gamma^{-1}}
\newcommand{\EG}{\mathcal{E}G}
\newcommand{\Gaa}{\widetilde{G_{\alpha}}}
\newcommand{\Ga}{\raisebox{-0.3ex}{$\widetilde{G_{\alpha}}$}}
\newcommand{\Galp}{G_{\alpha}}
\newcommand{\E}{\mathcal{E}}
\newcommand{\Max}{\mathrm{Max}}
\newcommand{\red}{\mathrm{red}}
\newcommand{\qqand}{\qquad \text{and}\qquad}
\newcommand{\qqfor}{\qquad \text{for}\quad}
\newcommand{\hattensor}{\widehat{\otimes}}
\newcommand{\turnV}[1]{\bigwedge\nolimits^{\! #1}}

\newcommand{\pardif}[3]{\frac{\partial^{#1}{#3}}{\partial{#2}}\>}
\newcommand{\ordif}[3]{\frac{\dif^{#1}{#3}}{\dif{#2}}\>}
\newcommand{\Sym}{\mathrm{S
ymb
}}
\newcommand{\Pseudo}{\Psi
{\scriptstyle \mathrm{\!D\!O}}
}
\newcommand{\PPseudo}{\mathrm{P}\mathchar`-\Psi
{\scriptstyle \mathrm{\!D\!O}}
}
\newcommand{\VVVVVV}{\bb{V}^0,\bb{V}^1}

\newcommand{\settingsGXa}
{Let $G$ be a second countable locally compact Hausdorff group.
Let $X$ be a $G$-compact 
proper complete $G$-Riemannian manifold. }

\newcommand{\dirsum}[1]{\bigoplus #1}
\newcommand{\hatsum}{\;\widehat{\oplus}\;}
\newcommand{\matrixtwo}[4]{
\left( \!\! \begin{array}{cc} #1 & #2 \\ #3 & #4 \end{array}\!\! \right)}

\title{On the Strong Novikov Conjecture of Locally Compact Groups
\\for Low Degree Cohomology Classes }
\author{Yoshiyasu Fukumoto \footnote
{
Department of Mathematics, Faculty of Science, Kyoto University
\newline \, \, \,
Sakyo-ku, Kyoto, 606-8502, JAPAN
\newline \, \, \,
fukumoto@math.kyoto-u.ac.jp}
}
\date{Kyoto University}
\begin{document}
\maketitle
\begin{abstract}
The main result of this paper is non-vanishing of the image of the index map 
from the $G$-equivariant $K$-homology of 
a proper $G$-compact $G$-manifold $X$ 
to the $K$-theory of the $C^{*}$-algebra of the group $G$. 
Under the assumption that the Kronecker pairing of a $K$-homology class 
with a low-dimensional cohomology class is non-zero, 
we prove that the image of this class under the index map is non-zero. 
Neither discreteness of the locally compact group $G$ nor freeness of the action 
of $G$ on $X$ are required. 
The case of free actions of discrete groups was considered earlier 
by B. Hanke and T. Schick.
\end{abstract}
\tableofcontents
\section{Introduction}
The main result of this paper is the
non-vanishing of the image of the index map,
assuming that the non-vanishing of the
pairing with low dimensional cohomology.
In this paper, actions of a locally compact group $G$
on $X$ will be assumed proper and co-compact,
but neither discreteness of $G$ nor the freeness of  the action
will be needed.
First of all, let us recall the Strong Novikov Conjecture
and a result in the case of discrete groups.
\begin{conj}[Strong Novikov Conjecture]
Let $\Gamma$ be a discrete group.
The Baum-Connes assembly map
$\map {\mu} {K_{0}(B\Gamma)} {K_{0}(C^{*}_{\Max}(\Gamma))}$
is injective after tensoring $\Q$.
\end{conj}
\begin{thm}
{\upshape{\cite[Theorem 1.2]{Ha-Sc08}}}
Let $\Gamma$ be a discrete group.
If a $K$-homology class $h\in K_{0}(B\Gamma)$ satisfies
the following condition;
there exists a cohomology class $c\in H^{*}(B\Gamma;\Q)$
generated by classes of degree at most $2$
such that 
$\angles{c, \mathrm{ch}(h)} \neq 0$.
Then $\mu(h) \neq 0$.
\end{thm}
Next, the statement of Baum-Connes conjecture
for the proper actions is the following;
\begin{conj}
{\upshape{\cite[(3.15)]{Ba-Co-Hi76}}}
Let $G$ be a locally compact Hausdorff second countable group
and let $\EG$ be the universal example for proper actions of $G$
defined in definition \ref{UnivExample}.
Then
$$
\map{\mu} {\mathcal{R}K_{0}^{G}(\EG)} K_{0}(C^{*}(G))
$$
is an isomorphism.
\end{conj}
This conjecture is true for such groups as, 
for instance, hyperbolic groups and 
Lie groups with finitely many connected components.

In this paper, few such assumptions on groups will be required.
For example, we allow the semi-direct product of a compact group
with a countable discrete group.
Instead, we will set some assumptions on the domain
of the Baum Connes map like \cite{Ma03} and \cite{Ha-Sc08},
and we focus on the rational injectivity of this map.

Now let us fix our settings.
Let $X$ be a complete Riemannian manifold with $H_1(X;\Z) = 0$
and let $G$ be a second countable locally compact Hausdorff
unimodular group
which act on $X$ properly and isometrically.
Assume that $X$ is $G$-compact.

Let $\map{\mu_{G}}{K\!K^{G}(C_0(X), \C)}{K_0(C^{*}(G))}$
be the index map,
and $E = L^{1} \oplus \ldots \oplus L^{N}$ be a finite direct sum of
smooth Hermitian $G$-line bundles
each of which has vanishing first Chern class; $c_1(L^{k}) = 0
\in H^{2}(X;\R)$.
\begin{thmalph}\label{MainThm}
Let $A$ be a properly supported $G$-invariant elliptic operator
of order $0$.
If $[A]\in K\!K^{G}(C_0(X), \C)$ satisfies that
the Kronecker pairing
$\angles{[E], [A]}_{G}$ is not equal to $0 \in \R$
for some $E$,
then $\mu_{G}([A]) \neq 0 \in K_{0}(C^{*}(G))$.
\end{thmalph}
The Kronecker pairing
$\angles{[E], [A]}_{G}$ will be defined later
in the definition \ref{DfnKronecker}.
Instead of taking the paring on the level of cohomology
$H^{*}(X)$ and homology $H_{*}(X)$
as in \cite{Ma03} and \cite{Ha-Sc08},
we will consider it on the level of equivariant $K$-theory.
In the case of free actions of discrete groups,
they coincide.

\begin{rems}
\begin{enumerate}
\item
In this paper $C^{*}(G)$ is either reduced or maximal 
$C^{*}$-algebra.
So even if $G$ is discrete, 
theorem \ref{MainThm}
is not covered by {\cite[Theorem 1.2]{Ha-Sc08}}
since in \cite{Ha-Sc08}, only the maximal group $C^{*}$-algebras
are considered to use the universal property of them.
\item
Tha assumption that $c_1(L)=0 \in H^{2}(X;\R)$ implies that
such line bundle $L$ over $X$ is topologically trivial, however,
remark that it does not follow that $L$ is trivial as a $G$-line bundle. 

One of the classical examples is as follows;
let us consider the case of free action of a discrete group,
namely, let $M$ be a closed Riemannian manifold, 
$G$ be its fundamental group $\pi_1(M)$,
and $X$ be its universal cover $\widetilde{M}$.
The projection will be denoted by $\map {\pi} {X} {M}$.
Moreover, let $\map {f} {M} {BG}$ be the classifying map,
and consider a line bundle $f^{*}L_0$
obtained by the pull back by $f$ of any Hermitian line bundle $L_0$ over $BG$.
Then, let $L$ be its lift
 $\pi^{*}f^{*}L_0$.
Since $f$ is lifted to
$\map {\widetilde{f}} {X} {EG}$ and 
$EG$ is contractible, we obtain $c_1(L) = 0$.

Similar argument also works in the case of proper actions.
Let $G$ be a locally compact group acting properly on
a complete Riemannian manifold $X$.
Let $\EG$ be the universal example of proper action of $G$,
$\map {f} {X} {\EG}$ be a $G$ equivariant map,
and $L_0$ be any Hermitian $G$-line bundle over $\EG$.
Then, the pull back $L:= f^{*}L_0$ is a Hermitian $G$-line bundle
over $X$, whose first Chern class vanishes since
$\EG$ is contractible as a topological space.
\end{enumerate}
\end{rems}
\begin{rem}
In {\cite{Ha-Sc08}}, the problem on $B\Gamma$
is reduced to the case of manifolds,
using the following lemma;
\begin{lem}
{\upshape{\cite[Theorem 6.2]{Ba-Hi-Sc07}}}
For any $h\in K_{0}(B\Gamma)$,
there exists a geometric $K$-cycle $(M,E,\phi)$
realizing $h$.
Precisely there exists $(M,E,\phi)$, where 
$M$ is a spin$^c$-manifold, $E$
is a smooth Hermitian vector bundle,
and $\map{\phi}{M}{B\Gamma}$ is a continuous map,
such that 
$$
\phi_{*}\parens{[E]\hattensor [M]} = h.
$$
\end{lem}
Unfortunately we do not have the corresponding result
for $G$-equivariant $K$-homology of proper $G$-spaces,
so the general case of proper $G$-topological space
including $\EG$
could not be reduced
to the case of $G$-manifolds.
\end{rem}
As a corollary of Theorem \ref{MainThm},
 we obtain the following Gromov-Lawson-Rosenberg conjecture
for low dimensional cohomology classes. 
\begin{coralph}\label{MainCor}
Let $X$ and $G$ be as above and additionally we assume 
$X$ is a spin manifold and the scalar curvature of $X$ is positive.
Then
for any Hermitian $G$-vectorbundle $E$ satisfying the conditions above,
the following higher $G$-$\widehat{\mathcal{A}}$-genus vanishes.
$$
\widehat{\mathcal{A}}_{G}(X;E)
:=
\int_{X} c(x) \widehat{\mathcal{A}}(TX)\wedge \mathrm{ch}(E)
=0,
$$
where $c\in C_c(X)$ denotes an arbitrary cut-off function.
\end{coralph}
The higher $G$-$\widehat{\mathcal{A}}$-genus is defined later
in Definition \ref{HigherAgenus} using 
the $L^{2}$-index formula; Theorem \ref{HWindex}
{\cite[Proposition 6.11]{Wa14}}.
In the case of a free action a descrete group $\Gamma$,
our higher $G$-$\widehat{\mathcal{A}}$-genus of $X$ coincides with
the higher $\widehat{\mathcal{A}}$-genus of $X/\Gamma$
by $\Gamma$-index theorem.
This gives a generalization of \cite[Corollary 1]{Ma99}
to non-discrete groups.

Remark that the following Novikov conjecture
for low dimensional cohomology classes
does not follow unless the acting group $G$ is discrete
and the action on $X$ is free.
This is because we do not know, unfortunately,
the homotopy equivalence of the $(C^{*}(G))$-index of
the signature operator,
$\mu_{G}([\dif^{X} + \delta^{X}]) \in K_{0}(C^{*}(G))$
in general cases.
\begin{conj}
\settingsGXa
Let $X'$ be another complete Riemannian $G$-compact 
proper $G$-manifold and assume that 
there exists a $G$-equivariant orientation
preserving smooth homotopy equivalence
$\map{h}{X'}{X}$.
Let $\map{f}{X}{\EG}$ be a $G$-equivariant map.
Then for any smooth Hermitian $G$-vector bundle
$E$ over $X$
given by $L=f^{*}E_0$ for some $G$-vector bundle $E_0$
over $\EG$,
the higher signature for $E_0$ is homotopy invariant, that is,
$$
\int_{X} c(x) \mathcal{L}(X)\wedge \mathrm{ch}(E)
=
\int_{X'} c'(x') \mathcal{L}(X')\wedge \mathrm{ch}(h^{*}E),
$$
where $c\in C_c(X)$ denotes an arbitrary cut-off function
and $\mathcal{L}(X)$ denotes the $L$-class of $X$.
\end{conj}
\section{Preliminaries}
\subsection{Proper Actions}
\begin{dfn}
Let $G$ be a second countable locally compact Hausdorff
group.
Let $X$ be a complete Riemannian manifold.
\begin{enumerate}
\item
$X$ is called a $G$-Riemannian manifold
if $G$ acts on $X$ isometrically.
\item
The action of $G$ on $X$ is said to be proper
or $X$ is called a proper $G$-space
if the following continuous map is proper;
$$
G\times X \to X \times X, \quad
(\gm , x)\mapsto (\gm x, x).
$$
\item
The action of $G$ on $X$ is said to be co-compact
or
$X$ is called $G$-compact space if the quotient space
$X/G$ is compact.
\end{enumerate}
\end{dfn}
\begin{dfn}
A compactly supported smooth function
$c \in C^{\infty}_{c}(X)$ is called a cut-off function
if
$c \geq 0$ is non-negative everwhere and it satisfies
$$
\int_{G} c(\gminv x) \mathrm{d}\gm =1
$$
for any $x\in X$ .
\end{dfn}
\begin{lem}
Any $G$-compact manifold
with proper action admits a cut-off function.
\end{lem}
\begin{proof}
Let $c' \in C^{\infty}_{c}(X)$ be an arbitrary
non-negative function whose support intersects
every orbit, namely, $\supp (c') \cap G(\{ x \}) \neq \emptyset$.
There actually exists such a function because
$X$ is $G$-compact.
Then consider the function $c$ by normalizing $c'$ on each orbit;
$$
c(x) := \frac{c'(x)}{\int_{G}c'(\gminv x) \mathrm{d}\gm}.
$$
This $c$ provides a cut-off function.
\\
Note that since the action of $G$ is proper,
$\braces{ \gm \in G \: \big| \: \gminv x \in \supp(c')}
\subset G$ is compact and hence the value
$\int_{G}c'(\gminv x) \mathrm{d}\gm$ is always finite 
for any $x\in X$.
\end{proof}
\begin{dfn}
The action of $G$ on $X$
induces actions on $TX$ and $T^{*}\!X$ given by
$$
\begin{array}{cccc}
&\gm \colon T_{x}X &\to &T_{\gm x}\qquad \quad
\\
&\qquad v &\mapsto &\gm(v):=\gm_{*}v
\end{array}
\qqand
\begin{array}{ccccc}
&\gm \colon T^{*}_{x}X &\to &T^{*}_{\gm x}\qquad \qquad \quad
\\
&\qquad \xi &\mapsto &\gm(\xi):=(\gminv)^{*}\xi.
\end{array}
$$
The action on $\mathfrak{X}(X)$ and $\Omega^{*}(X)$ is given by
$$
\gm[V]
:= 
\gm_{*}V
\qqand
\gm[\omega]
:= 
({\gminv})^{*} \omega
$$
for $\gm\in G$, $V\in\mathfrak{X}(X)$
and $\omega \in \Omega^{*}(X)$. Obviously,
$\gm[\omega \wedge \eta] = \gm[\omega] \wedge \gm[\eta]$
and $\dif (\gm [\omega]) = \gm[\dif \omega]$.
\end{dfn}
\begin{lem}
\label{IntegralGInvariantForms}
Let $\omega \in \Omega^{n}(X)$ and 
$\eta \in \Omega^{n-1}(X)$ be $G$-invariant differential forms,
where $n = \dim(X)$
and $c \in C_c^{\infty}(X)$ be a cut-off function. Then,
\begin{enumerate}
\item
$\int_{X} c(x) \omega$ does not depend on the choice of 
a cut-off function.
\item
$\int_{X} c(x) \dif \eta = 0$.
\end{enumerate}
\end{lem}
\begin{proof}
Firstly, note that for any $\gm \in G$, $\map{\gm}{X}{X}$ is
a diffeomorphism, so 
$$
\int_{X} \zeta = \int_{X} \gm[\zeta]
$$
holds for any compactly supported
differential form $\zeta \in \Omega_c^{n}(X)$.

Let $c' \in C_c^{\infty}(X)$ be another cut-off function and set
$a := \int_{X} c(x) \omega - \int_{X} c'(x) \omega$. Then,
\begin{eqnarray*}
a
&=&
\int_{X} \gm[c(x) \omega] - \int_{X} \gm[c'(x) \omega]
\\
&=&
\int_{X} c(\gminv x) \gm[\omega] - \int_{X} c'(\gminv x) \gm[\omega]
\\&=&
\int_{X} c(\gminv x) \omega - \int_{X} c'(\gminv x) \omega
\\&=&
\int_{X} \parens{c(\gminv x)- c'(\gminv x)}
\omega \:,
\\
\text{hence,}\quad
\int_{G} a \:\dif \gm
&=&
\int_{X} \parens{\int_{G} c(\gminv x)\:\dif \gm
-
\int_{G} c'(\gminv x)\:\dif \gm}
\omega
\\&=&
\int_{X} \parens{1-1}\omega = 0.
\end{eqnarray*}
Therefore $a=0$.
As for the second part, set $b:= \int_{X} c(x) \dif \eta$. Then,
\begin{eqnarray*}
&&
b=
-\int_{X} \dif c(x) \wedge \eta
=
-\int_{X} \gm\brackets{\dif c(x)\wedge \eta}
=
-\int_{X} \dif c(\gminv x)\wedge \gm[\eta]
=
-\int_{X} \dif c(\gminv x)\wedge \eta \:,
\\
\text{hence,}
&&
- \int_{G} b \:\dif \gm
=
\int_{X}\dif \parens{
\int_{G} c(\gminv x) \:\dif \gm} \wedge \eta
=
\int_{X}\dif (1) \wedge \eta = 0.
\end{eqnarray*}
Therefore $b=0$.
\end{proof}
\begin{rem}
\label{IntegralOnFundamentalDomain}
In the case when $G = \Gamma$ is a discrete group acting
on $X$ freely, if $\omega \in \Omega^{n}(X)$ is
a $\Gamma$-invariant form on $X$, then
the value $\int_{X} c(x) \omega$ is equal to $\int_{F} c(x) \omega$
for an arbitrary fundamental domain $F \subset X$.
Therefore if $\omega ' \in \Omega(X/\Gamma)$
is a differential form covered by 
the $\Gamma$-invariant form $\omega$,
we have $\int_{X} c(x) \omega = \int_{X/\Gamma} \omega '$.
\end{rem}
\begin{prop}[Slice theorem]\label{SliceThm}
Let $G$ be a second countable locally compact Hausdorff
group and act
properly and isometrically on $X$.
Then for any neighborhood $O$ of any point $x\in X$
there exists a compact subgroup $K \subset G$
including the stabilizer at $x$,
$K \supset G_{x} := \braces{\gm \in G | \gm x =x}$
and
there exists a $K$-slice $\{ x \} \subset S \subset O$.
\end{prop}
Here $S \subset X$ is called $K$-slice if the followings 
are satisfied;
\begin{itemize}
\item
$S$ is $K$-invariant; $K(S) = S$,
\item
the tubular subset $G(S) \subset X$ is open,
\item
there exists a $G$-equivariant map
$\map{\psi}{G(S)}{G/K}$ satisfying $\psi^{-1}([e]) = S$,
called a slice map.
\end{itemize}
We will give a proof later only in the case when $G$ is a Lie group,
in which case $K$ can be taken as $G_{x}$ itself.
The proof in the general case is found
in \cite[Theorem 3.3.]{Ab78} and \cite[Theorem 3.6.]{An05}.
\begin{cor}\label{CorSliceThm}
We additionally assume that $X$ is $G$-compact.
Then for any open covering
$X = \bigcup_{x \in X} O_{x}$,
there exists a sub-family of
finitely many open subsets $\{O_{x_i}, \ldots O_{x_N}\}$
such that
$$\bigcup_{\gm \in G} \bigcup_{i=1}^{N} \gm(O_{x_i})=X.$$

In particular, $X$ is of bounded geometry,
namely, the injective radius is bounded below and
the norm of Riemannian curvature is bounded.
\end{cor}
\begin{proof}
Due to the slice theorem we have 
compact subgroups $K_{x}$ including the stabilizers $G_{x}$
and $K_{x}$-slice $\{ x \} \subset S_{x} \subset O_{x}$.
Since $X/G$ is compact, and $\braces{[G(S_{x})]}_{x\in X}$
is an open covering of $X/G$,
one can take 
finite sub-cover $X/G = \bigcup_{i=1}^{N} [G(S_{x_i})]$
and hence obtain a finite $G$-invariant open cover
$X = \bigcup_{i=1}^{N} G(S_{x_i})$.
On the other hand, obviously $G(O_{x_i}) \supset G(S_{x_i})$ so
it follows that
$X = \bigcup_{\gm \in G} \bigcup_{i=1}^{N} \gm(O_{x_i})$.

In order to prove the bounded geometry of $X$,
take an arbitrary 
relatively compact neighborhood
$\{ x \} \subset O_{x} \subset X$ for each $x\in X$
and find $\{O_{x_i}, \ldots O_{x_N}\}$ as above.
Noting that each of $\gm(O_{x_i})$ is isometric
to $O_{x_i}$ for every $\gm \in G$,
the lower bound of injective radius and the norm of
Riemannian curvature on whole $X$ is determined only on
$O_{x_1} \cup \ldots \cup O_{x_N}$,
which is relatively compact.
Therefore they are bounded.
\end{proof}
\begin{proof}[Proof of the Proposition \ref{SliceThm}]
Let $G$ be a Lie group. Set $K = G_{x}$.
Due to the properness of the action,
$G_{x}$ is a compact sub-group.
Besides every orbit set $G\{x\} \subset X$ is
a closed subset of $X$ since the orbit map
${G}\to{X}; \; \gm \mapsto \gm x$ is proper.
At the same time, 
$G/G_{x} \to X; \quad [\gm]\mapsto \gm x$ is
an injective immersion.
In particular, the orbit set $G\{x\} \subset X$ is a sub-manifold 
in $X$ differomorphic to $G/G_{x}$.

Let $N (G\{x\})$ denote the normal bundle of  $G\{x\} \subset X$,
namely $N_{y} (G\{x\})$ is
 the orthogonal complement of $T_{y}(G\{x\})$
in $T_{y}X$.
Since the action of $G$ on $X$ is isometrically,
the action on $N (G\{x\})$ preserves the metric of the fibers.

Let $N_{x}^{r}G\{x\} := \braces{v\in N (G\{x\})
\mid \norm{v} <r}$
for positive $r\in \R$.
Choose a small $\delta > 0$ so that the
following restricted exponential map
$$\map{\exp_{x}} {N_{x}^{4\delta}(G\{x\})} {X}$$
is diffeomorphism onto its image
and the image does not intersect with $G\{x\}$ at any other points 
than $x\in G\{x\}$. Then
$$
\exp\colon N^{\delta}(G\{x\})
\to X
;\quad 
(y,v)\mapsto \exp_{y}(v) 
$$
is a diffeomorphism onto its image.


Now let $S:= \exp_{x} (N^{\delta}_{x}G\{ x \})$,
with $\delta$ replaced by a smaller value
so that $S$ is contained in $O$ if needed.
This $S$ is obviously $G_{x}$-invariant.
Moreover the tubular subset $G(S)$ is equal to the image of
the exponential map 
$$
G(S)
= \bigcup_{y\in G\{ x \}}\exp_{y} (N^{\delta}_{y}G\{ x \})
= \exp(N^{\delta}(G\{x\}))
$$
which is diffeomorphic to a disk bundle $N^{\delta}(G\{x\})$
and hence open.
One can take a slicing map $\map{\psi}{G(S)}{G/K}$ as follows;
$$
\begin{array}{cccccccc}
\psi \colon G(S)=&\exp(N^{\delta}(G\{x\}))
&\xrightarrow[\cong]{\exp^{-1}}
&N^{\delta}(G\{x\})
&\xrightarrow{\text{proj.}}
&G\{x\} \cong
&G/G_{x}
\\
&\exp_{y}(v) 
&\longmapsto
&(y,v)
&\longmapsto
&y.&&
\end{array}
$$
The inverse image $\psi^{-1}$ is actually equal to
$\exp(N^{\delta}_{x}(G\{x\})) = S$.
\end{proof}
\subsection{Universal Example for Proper Actions}

We refer to \cite[Section 1.]{Ba-Co-Hi76}.
\begin{dfn}[Universal example for proper actions]
\label{UnivExample}
Let $G$ be a second countable locally compact Hausdorff group.
A universal example for proper actions of $G$,
denoted by $\EG$,
is a proper $G$-space satisfying the following property;
If $X$ is any proper $G$space,
then there exists a $G$-equivariant map
$\map{f}{X}{\EG}$,
and any such maps are $G$-homotopic,
that is, homotopic through $G$-equivariant maps.
\end{dfn}
\begin{lem}\text{}
\begin{enumerate}
\item
There exists a universal example for proper actions of $G$,
and it is unique up to $G$-homotopy.
\item
$\EG$ is contractive as a topological space.
\end{enumerate}
\end{lem}
{\textit{Outline of Proof.}}
The construction of $\EG$ is as follows;
let 
$$
W := \bigsqcup_{\substack{H \subset G\\
\text{compact subgroup}}} G/H
$$ be the disjoint union of the homogeneous
spaces $G/H$, over all compact subgroups of $G$.
Then take the join with itself infinitely many times; 
$$
\EG = W* W* \ldots.
$$
From this construction it follows that
$\EG$ is contractive.

Moreover if we have two, $\EG$ and $(\EG)'$,
then we have $G$-maps
$\map{f}{(\EG)'}{\EG}$ and $\map{f'}{\EG}{(\EG)'}$,
each of which is unique up to $G$-homotopy.
Moreover $f\circ f'$ and $f'\circ f \id_{(\EG)'}$
are $G$-homotopic to $\id_{\EG}$ and $\id_{(\EG)'}$
respectively.
\section{Preparation for Description of the Index Maps}
\subsection{Hermitian $G$-Vector Bundles,
Connections, and Characteristic Classes}
\begin{dfn}[Hermitian $G$-vector bundles, $G$-invariant connections]
\text{}\\
Let $X$ be  a proper $G$-Riemannian manifold.
\begin{enumerate}
\item
A smooth $G$-vector bundle $E$
is a smooth vector bundle over $X$
on which $G$ smoothly acts and
the action is compatibility  with the 
the action on the base space $X$,
and the action induced on fibers are linear isomorphisms,
that is, $$\pi (\gm \sigma) = \gm \pi (\sigma)$$
holds for any $\sigma \in E$ and $\gm \in G$,
where $\map{\pi}{E}{X}$ denotes the projection,
and
$\map{\gm}{E_{x}}{E_{\gm x}}$ is a linear isomorphism.
If $E = E^{0} \hat{\otimes} E^{1}$ is graded,
the action of $G$ is assumed to preserve the gradings.
\item
Let $C^{\infty}(E)$ denote the space consisting of all
smooth sections $\map{s}{X}{E}$. 
$G$ acts on $C^{\infty}(E)$ by the formula;
$$
\gm[s](x) = \gm(s(\gminv x)).
$$
\item
A $G$-vector bundle is called Hermitian $G$-vector bundle
if 
$E$ is equipped with a $G$-invariant smooth Hermitian metric
$h(\cdot, \cdot) = \angles{\cdot, \cdot}_{E_x}$,
that is,
$\angles{\gm \sigma_1, \gm \sigma_2}_{E_{\gm x}}
=
\angles{\sigma_1, \sigma_2}_{E_{x}}$
holds for any $\sigma_1, \sigma_2 \in E_{x}$  
and $\gm \in G$.
\item
A connection $\nabla$ on a Hermitian $G$-vector bundle
is called a $G$-invariant connection if 
$\nabla$ is compatible with the Hermitian metric and
commutes with the action of $G$, that is,
$$
\dif \angles{s(x), t(x)}
=
\angles{{\nabla} s(x), t(x)}
+
\angles{s(x), {\nabla} t(x)}
\qqand
\gm[\nabla s] = \nabla (\gm[s])
$$
hold for any $s, t \in C^{\infty}(E)$  
and $\gm \in G$.
\end{enumerate}
\end{dfn}
\begin{dfn}
A Hermitian $G$-vector bundle $E$ over $X$ defines the
following element $[E]$ in $K\!K$-theory;
$$
[E] = \parens{C_{0}(E), 0} \in K\!K^{G}(C_0(X), C_0(X))
$$
with the action of $C_0(X)$ on $C_0(E)$
being the point-wise multiplication.

Here, $C_0(X)$ denotes the $C^{*}$-algebra
consisting of all $\C$-valued continuous functions
on $X$ vanishing at infinity equipped with the sup-norm,
and $C_0(E)$ denotes the Hilbert $C_0(X)$ module
consisting of all continuous sections of $E$
vanishing at infinity
equipped with the scalar product given by the
point wise Hermitian metric.
\end{dfn}
\begin{lem}\label{ExistenceOfGConnection}
\settingsGXa
\begin{enumerate}
\item
Any smooth $G$-vector bundle $E$ over $X$ admits
a smooth $G$-invariant Hermitian metric.
\item
Any Hermitian $G$-vector bundle $E$ over $X$ admits
a $G$-invariant connection.
\end{enumerate}
\end{lem}
\begin{proof}
Choose any cut-off function $c \in C_c(X)$.
The only we have to do is to take the average using $c$.
For the first part,
fix any smooth Hermitian metric $h$ on $E$
which is not necessarily $G$-invariant.
Then set 
$$
\tilde{h}_{E_{x}}(\sigma_1, \sigma_2)
:=
\int_{G}c(g^{-1}x) {h}_{E_{g^{-1}x}}
\parens{g^{-1}\sigma_1, g^{-1}\sigma_2}
\:\dif g
\qqfor \sigma_1,\sigma_2 \in E_{x}.
$$
Note that for any fixed $x\in X$,
$\supp \{c(\gminv x)\} \subset G$
is compact due to the properness of the action
so that the integral above makes sense.
$\tilde{h}$ is a desired metric.
In fact for any fixed $\gm \in G$,
\begin{eqnarray*}
\tilde{h}_{E_{\gm x}}(\gm \sigma_1, \gm \sigma_2)
&=&
\int_{G}c(g^{-1}\gm x) {h}_{E_{g^{-1}\gm x}}
\parens{g^{-1}\gm \sigma_1, g^{-1}\gm \sigma_2}
\:\dif g
\\
&=&
\int_{G}c(g^{-1}\gm x) {h}_{E_{g^{-1}\gm x}}
\parens{g^{-1}\gm \sigma_1, g^{-1}\gm \sigma_2}
\:\dif (\gminv g )
\\
&=&
\tilde{h}_{E_{x}}(\sigma_1, \sigma_2).
\end{eqnarray*}
It is follows from $c \geq 0$ that $\tilde{h}$
is positive definite.

As for the second part,
let $E$ be a Hermitian $G$-vector bundle over $X$
and fix a connection compatible with the Hermitian metric
$\map{\nabla}{C^{\infty}(E)}{C^{\infty}(T^{*}\!X \otimes E)}$
which is not necessarily commutes with the action of $G$.
Then set
$$
\widetilde{\nabla}s(x)
:=
\int_{G}
c(g^{-1}x) g \nabla g^{-1}s(x) \:\dif g
\qqfor
s \in C^{\infty}(E).
$$
\end{proof}
Note that the action of $G$ on $E$ and $X$
induces the action on $E\otimes T^{*}\!X$.
$\widetilde{\nabla}$ is a desired connection.
In fact for any fixed $\gm \in G$,
\begin{eqnarray*}
\gm[\widetilde{\nabla}s](x)
&=&
\gm (\widetilde{\nabla}s(\gminv x))
\\
&=&
\int_{G}
c(g^{-1}\gminv x) \gm g\nabla g^{-1}s(\gminv x) \:\dif g
\\
&=&
\int_{G}
c(\{\gm g\}^{-1} x) \{\gm g\}
\nabla \{\gm g\}^{-1} \gm s(\gminv x)
\:\dif \{\gm g\}
\\
\widetilde{\nabla}(\gm[s])(x)
&=&
\int_{G}
c(g^{-1}x) g\nabla g^{-1}\gm s(\gminv x) \:\dif g,
\\
\text{therefore}\qquad
\gm[\widetilde{\nabla}s](x)
&=&
\widetilde{\nabla}(\gm[s])(x).
\end{eqnarray*}
The compatibility with the Hermitian metric follows from;
\begin{eqnarray*}
\angles{\widetilde{\nabla} s_1(x), s_2(x)}
+
\angles{s_1(x), \widetilde{\nabla} s_2(x)}
&=&
\int_{G}c(g^{-1})
\angles{g \nabla g^{-1}s_1(x) , s_2(x)}
+
\angles{s_1(x), g \nabla g^{-1}s_2(x) }
\:\dif g
\\
&=&
\int_{G}c(g^{-1})
\angles{\nabla g^{-1}(s_1(x) , g^{-1} s_2(x)}
+
\angles{s_1(x), g^{-1} \nabla g^{-1}(s_2(x) }
\:\dif g
\\
&=&
\int_{G}c(g^{-1})
\dif \angles{g^{-1}s_1(x), g^{-1} s_2(x)}
\:\dif g
\\
&=&
\int_{G}c(g^{-1})
\dif \angles{s_1(x), s_2(x)}
\:\dif g
\\&=& \dif \angles{s_1(x), s_2(x)}.
\end{eqnarray*}
The Lipnitz rule follows from;
\begin{eqnarray*}
\widetilde{\nabla}(fs)(x)
&=&
\int_{G}
c(g^{-1}x) g \nabla f g^{-1}s(x) \:\dif g
\\
&=&
\dif f\otimes
\int_{G}c(g^{-1}x) g g^{-1}s(x) \:\dif g
+
f(x) \int_{G}
c(g^{-1}x) g \nabla g^{-1}s(x) \:\dif g
\\
&=&
\dif f \otimes s(x)
+
f(x) \widetilde{\nabla}s(x).
\end{eqnarray*}

\begin{rem}
Let $E$ be a smooth Hermitian $G$-vector bundle,
then $G$ acts on $\End(E)$
by
$$\gm[T] = \gm T \gminv \qqfor T\in \End(E), \gm \in G.$$
Then, 
$\gm(T(\sigma)) = \gm[T](\gm\sigma)$ and
$\gm[S\circ \! T] = \gm[S]\circ\! \gm[T]$ 
hold for any $S,T\in \End(E)_{x}$ and $\sigma \in E_{x}$.
\end{rem}
\begin{dfn}[Chern classes and Chern character]
Let $E$ be a smooth Hermitian $G$-vector bundle
over $X$.
Choose a $G$-invariant connection
$\nabla$ on $E$
and let $\Theta^{\nabla} \in C^{\infty}
(\turnV{2}(T^{*}\!X) \otimes \mathfrak{u}(E))$
denote the curvature $2$-form.
Then define the Chern classes $c_k(E,\nabla)
\in \Omega^{\! 2k}(X)$
and the Chern character $\mathrm{ch}(E,\nabla)
\in \Omega^{*}(X)$
by the formula;
$$
\sum_{k} c_k(E,\nabla) \lambda^{k}
=
\mathrm{det}\parens{
\frac{i\lambda \Theta^{\nabla}}{2\pi}-\id_{E}
}
\qqand
\mathrm{ch}(E,\nabla)
=
\tr \parens{\exp \parens{
\frac {i \Theta ^{\nabla}} {2\pi} } }.
$$
We may omit $\nabla$ when it is clear from the
conetxt.
As usual, $c_{k}(E,\nabla)$ and $\mathrm{ch}(E,\nabla)$
are closed form.
\end{dfn}
\begin{rem}
The following two lemmas are
valid also for any other characteristic classes
constructed using the Chern Weil theory.
\end{rem}
\begin{lem}
Let $\nabla$ be a $G$-invariant connection on $E$.
Then $\Theta^{\nabla}$,
$c_k(E,\nabla)$ and $\mathrm{ch}(E,\nabla)$ are
$G$-invariant differential form.
\end{lem}
\begin{proof}
Note that $\Theta^{\nabla}\! \wedge s = \dif^{\nabla}\dif^{\nabla}s$,
where $\map{\dif^{\nabla}}
{C^{\infty}(\turnV{k}T^{*}\!X \otimes E)} 
{C^{\infty}(\turnV{k+1}T^{*}\!X \otimes E)}$ denotes the
exterior covariant derivative defined  by
$$
\dif^{\nabla}(\omega \otimes s)= 
\dif \omega \otimes s 
+ (-1)^{k}\omega \wedge \nabla s
\qqfor
\omega \in \Omega^{k}(X), s \in C^{\infty}(E),
$$
which commutes with the action of $G$.
Hence, $\Theta^{\nabla}$ is $G$-invariant
and we can conclude that $c_k(E,\nabla)$ and
$\mathrm{ch}(E,\nabla)$ are $G$-invariant
since they are expressed by polynomials in
$\tr(\Theta^{\nabla})$, $\tr((\Theta^{\nabla})^2)$,
$\tr((\Theta^{\nabla})^3), \cdots$.
\end{proof}
\begin{lem}
\label{GInvariantChernWeil}
Let $E$ be a smooth Hermitian $G$-vector bundle
over $X$.
Let $\nabla^{0}$ and $\nabla^{1}$ be $G$-invariant
connections on $E$.
Then there exist $G$-invariant
differential form of odd degree $\theta_{k} \in \Omega^{2k-1}(X)$
for each $k$
and $\theta \in \Omega^{\text{odd}}(X)$
such that
$$
\dif \theta_{k} = c_{k}(E,{\nabla^{1}}) - c_{k}(E,{\nabla^{0}})
\qqand
\dif \theta = \mathrm{ch}(E,{\nabla^{1}}) - \mathrm{ch}(E,{\nabla^{0}}).
$$
\end{lem}
In particular $\int_{X}c(x) \mathrm{ch}(E,{\nabla})$
and $\int_{X}c(x) P(c_1(E,{\nabla}), c_2(E,{\nabla}),\ldots)$,
where $P$ is an arbitrary multi-variable polynomial,
are independent of
the choice of $G$-invariant connection $\nabla$
(and a cut-off function $c$ )
due to the
lemma \ref{IntegralGInvariantForms}.
\begin{proof}
Since $(\nabla^{1}- \nabla^{0})(fs) = f(\nabla^{1}- \nabla^{0})s$
for $f\in C^{\infty}(X)$ and $s\in C^{\infty}(E)$,
there exists $\eta \in C^{\infty}(\turnV{1}T^{*}\!X\otimes \End(E))$
such that $$(\nabla^{1}- \nabla^{0})s = \eta \wedge s.$$
Note that $\eta$ is $G$-invariant since $\nabla^{0}$ and $\nabla^{1}$
are $G$-invariant connections.
Consider a smooth family of connections for $t \in [0,1]$
$$ \nabla^{t} := \nabla^{0} + t\eta. $$
We need a local expression of the curvature 2-forms.
Fix a local trivialization $E|_{U} \to U \times {\C^{N}}$
and express 
$\nabla^{0} = \dif + \omega$.
Then we have 
$$
\nabla^{t} = \dif + \omega + t\eta
\qqand
\Theta(t) :=
\Theta^{\nabla^{t}} = d\omega(t)
+ \omega(t) \wedge \omega(t),
$$
where $\omega(t) = \omega + t\eta$.
Note that also $\Theta(t)$ are $G$-invariant
since $\nabla^{t}$ are commute with the action of $G$.

We will omit the symbol $\wedge$
for the rest of the proof
if there is no chance of leading to misunderstandings.
Due to the Bianchi identity for $\nabla^{t}$, that is,
$
\dif \Theta(t)
= \Theta(t) \omega(t)
- \omega(t) \Theta(t)
$, we have
\begin{eqnarray}
\dif \Theta(t)^{m}
&=&
\sum_{j=1}^{m} \Theta(t)^{m-j}
\Big(\dif\Theta(t)\Big)
\Theta(t)^{j-1}
\nonumber\\
&=&
\sum_{j=1}^{m} \Theta(t)^{m-j}
\Big(\Theta(t) \omega(t)
- \omega(t) \Theta(t)\Big)
\Theta(t)^{j-1}
\nonumber\\
&=&
\sum_{j=1}^{m} \Theta(t)^{m-j+1}
\omega(t) \Theta(t)^{j-1}
-
\sum_{j=1}^{m} \Theta(t)^{m-j}
\omega(t) \Theta(t)^{j}
\nonumber\\
&=&
\Theta(t)^{m} \omega(t)
- \omega(t) \Theta(t)^{m}.
\label{0020}
\end{eqnarray}
In particular $\tr (\Theta(t)^{m})$ are closed forms
due to the symmetry of the trace;
\begin{eqnarray}
\dif \tr (\Theta(t)^{m}) = \tr (\dif \Theta(t)^{m})
= \tr \braces{ \Theta(t)^{m} \omega(t)
- \omega(t) \Theta(t)^{m} }
=0.
\label{0025}
\end{eqnarray}
Due to (\ref{0020}), the symmetry of the trace and the fact
that $\omega$ and $\eta \Theta(t)^{m}$
are of odd degree,
\begin{eqnarray}
&&
\tr\braces{
\big(
\eta \omega(t)
+ \omega(t) \eta
\big)
\Theta(t)^{m}
}
\nonumber\\
&=&
\tr\braces{
\eta
\omega(t)
\Theta(t)^{m}
-
\eta 
\Theta(t)^{m}
\omega(t)
}
\nonumber\\
&=&
\tr\braces{
\eta \wedge \big(
\omega(t) 
\Theta(t)^{m}
{-}
\Theta(t)^{m}
\omega(t)
\big)
}
\nonumber\\
&=&
-\tr\braces{
\eta \wedge \dif\Theta(t)^{m}
}.
\label{0030}
\end{eqnarray}
Using (\ref{0030})
and again the symmetry of the trace,
\begin{eqnarray} 
\ordif{}{t}{}\parens {\tr (\Theta(t)^{k}) }
=
\tr\parens{\ordif{}{t}{}\Theta(t)^{k}}
&=&
\tr\braces{
\sum_{j=1}^{k}\Theta(t)^{k-j}
\parens{\ordif{}{t}{} \Theta(t)} \Theta(t)^{j-1}
}
\nonumber
\\
&=&
\sum_{j=1}^{k}\tr\braces{
\parens{\ordif{}{t}{} \Theta(t)}
\Theta(t)^{k-1}
}
\nonumber
\\
&=&
k\tr\braces{
\big(
\dif \eta
+ \eta \omega(t)
+ \omega(t) \eta
\big)\:
\Theta(t)^{k-1}
}
\nonumber
\\
&=&
k \parens{\tr\braces{
\dif \eta
\wedge
\Theta(t)^{k-1}
}
-
\tr\braces{
\eta \wedge \dif\Theta(t)^{k-1}
}}
\nonumber
\\
&=&
k\tr \braces{\dif \parens{ \eta \wedge \Theta(t)^{k-1} } }
\nonumber
\\
&=&
\dif\parens{ k \tr \parens{\eta \wedge \Theta(t)^{k-1}
}}.
\nonumber
\\
\text{Then we obtain } \qquad
\tr (\Theta(1)^{k}) - \tr (\Theta(0)^{k})
&=&
\dif\parens{
k \tr \parens{\eta \int_{0}^{1}\Theta(t)^{k-1} \: \dif t
}}
\nonumber
\\
&=& \dif \zeta_{k},
\label{0040}
\end{eqnarray}
where $
\zeta_{k} = k \tr \parens{\eta \int_{0}^{1}\Theta(t)^{k-1} \dif t}
$ , which are $G$ invariant
since $\eta$ and $\Theta(t)$
are globally defined $G$-invariant differential forms.

Next, for $\lfloor\frac{\dim(X)}{2}\rfloor$-variable polynomials $P$,
let us introduce $G$-invariant differential forms
$$
P\angles{t}:=
P\parens{\tr(\Theta(t)),\: \tr(\Theta(t)^2),\:
\tr(\Theta(t)^3), \ldots\: }.
$$
And we claim that every $P$ satisfies the following two conditions; 
\begin{itemize}
\item[(A)]
$P\angles{t}$ is a closed form; $\dif P \angles{t} = 0$.
\item[(B)]
There exists a $G$-invariant differential form
$\zeta_{P} \in \Omega^{*}(X)$
such that 
$
\dif \zeta_{P} = P\angles{1} - P\angles{0}
$.
\end{itemize}
If this is true, we have finished the proof of the
lemma \ref{GInvariantChernWeil}
since $c_k(E,\nabla)$ and $\mathrm{ch}(E,\nabla)$
are expressed by polynomials in 
$\tr(\Theta^{\nabla}), \tr((\Theta^{\nabla})^2),
\tr((\Theta^{\nabla})^3), \cdots$.

Let us verify the claim.
If two polynomials $P$ and $Q$ satisfy these conditions,
obviously so does $aP+bQ$.
We will verify that so does the product $PQ$.
In fact, we have
$$\dif (P\angles{t}Q\angles{t}) =
\dif P\angles{t} Q\angles{t}
+ P\angles{t}
 \dif Q\angles{t} =0, $$
so $PQ$ satisfies the condition (A).
Besides, we have
\begin{eqnarray*}
P\angles{1}Q\angles{1} - P\angles{0}Q\angles{0}
&=& 
(P\angles{1}- P\angles{0})Q\angles{1} 
+
P\angles{0}(Q\angles{1} - Q\angles{0})
\\
&=&
\dif \zeta_{P} Q\angles{1} 
+
P\angles{0}\dif \zeta_{Q}
\\
&=&
\dif \parens{
\zeta_{P} Q\angles{1} 
+
P\angles{0} \zeta_{Q}
}
\\
&=&
\dif \zeta_{PQ}.
\end{eqnarray*}
Since all of the $\zeta_{P}, Q\angles{1}, P\angles{0}$ and $\zeta_{Q}$
are $G$-invariant, so does 
$\zeta_{PQ} := \zeta_{P} Q\angles{1} 
+
P\angles{0} \zeta_{Q}$, which implies that
$PQ$ satisfies the condition (B).

Moreover, it follows from (\ref{0025}) and (\ref{0040}) that
all of the monomial $P_{k}$ of the form
$$
P_{k} \parens{\tr(\Theta(t)),\: \tr(\Theta(t)^2),\:
\tr(\Theta(t)^3), \ldots\: }
= \tr(\Theta(t)^{k}),
$$
which generate all the polynomials,
satisfy the two conditions above.
Therefore we conclude that every polynomial $P$ satisfies 
(A) and (B).
\end{proof}
\subsection{Properly Supported $G$-Invariant Elliptic Operators}
Let $X$ be  a proper $G$-Riemannian manifold,
and $\bb{V}^{0}, \bb{V}^{1}$ be Hermitian $G$-vector bundle.
(Later on, we will consider a
$\Z/2\Z$-graded bundle $\bb{V} = 
\bb{V}^{0} \widehat{\oplus} \bb{V}^{1}$.)
The goal of this subsection is to define a $K$-homology classes
determined by properly supported $G$-invariant elliptic operators.
We refer to \cite[Section 3]{Ka} and \cite{Ho71}.

Let $\map{\pi}{T^{*}\!X}{X}$ be the projection.
For a short while, we forget the action of $G$.
\begin{dfn}\text{}
\begin{description}
\item[Pseudo-differential operators.]
For $m\in \R$,
a pseudo-differential operator
$\map{A_0}{C_c^{\infty}(\bb{V}^0)}{C^{\infty}(\bb{V}^1)}$
of order $m$ is an operator of the form
\begin{eqnarray}
\label{ConstructPseudo}
A_0 s(x) =
\frac{1}{(2\pi)^{n}}
\int_{X \times T^{*}_{x}X}
e^{i\Phi(y,(x,\xi))} \mathfrak{a}(y, (x,\xi)) s(y)
\:\dif y \:\dif \xi
\qqfor
s\in C_c^{\infty}(\bb{V}_0)
\end{eqnarray}
with smooth function $\map{\mathfrak{a}}{X\times T^{*}\!X}
{\Hom(\pi^{*}\bb{V}^0, \pi^{*}\bb{V}^1)}$ satisfying the condition
that for any 
compact subset $K \subset X \times X$
and multi indices
$a=(a_1,\ldots, a_n),
a'=(a'_1,\ldots, a'_n) b=(b_1,\ldots, b_n)$, where $n=\dim X$,
there exists a constant
$C_{a,a',b,K}$ is a constant depending on $K$, $a$, $a'$ and $b$
such that
$$
\norm{
\pardif{|a'|}{y^{a'}}{}\pardif{|a|}{x^a}{}\pardif{|b|}{\xi^b}{}
\mathfrak{a}(y,(x,\xi))
}
\leq
C_{a,a',b,K}\sqrt{1+\norm{\xi}^2}^{\>m-|b|}
\qquad\text{ for any }
(x,y) \in K,\; \xi \in T_{x}X.
$$
$\mathfrak{a}$ is assumed to vanish outside a neighborhood
$U$ of the diagonal of $X \times X$
so that a small neighborhood of 
the image of the zero-section 
of $TX$ is diffeomorphic to $U$ by
$(x,v) \mapsto (x,\exp_{x}(v)) $
and that the phase function $\Phi$ is
defined on this neighborhood $U$. 

Here $\Phi$, called the phase function,
is a $\C$-valued function given by
$\Phi(y, (x,\xi)) = \angles{\xi, \exp_{x}^{-1}(y)}_{T^{*}_{x}X}$.
For instance if $X=\R$, $\Phi(y, (x,\xi)) = 
\angles{\xi, y-x}$.

The space consisting of all pseudo-differential operators
$\map{A_0}{C_c^{\infty}(\bb{V}^0)}{C^{\infty}(\bb{V}^1)}$
will be denoted by $\Pseudo^{m}(\VVVVVV)$.
\item[Schwartz kernels.]
The Schwartz kernel or the distributional kernel for
$A_0 \in \Pseudo^{m}(\VVVVVV)$
is a distributional section
$\map{K_{A_0}} {X\times X}{ \bb{V}^{1} \boxtimes (\bb{V}^{0})^{*}}$
such that
$$
A_0 s(x) = \int_{X}
K_{A_0}(x,y) s(y) \:\dif y
\qqfor
s\in C_c^{\infty}(\bb{V}_0).
$$
Formally, $K_{A_0}$ is given by
the formula;
\begin{eqnarray*}
K_{A_0} (x,y) =
\frac{1}{(2\pi)^{n}}
\int_{T^{*}_{x}X}
e^{i\Phi(y,(x,\xi))} \mathfrak{a}(y, (x,\xi))
\:\dif \xi.
\end{eqnarray*}
in the distributional sense and precisely, $K_{A_0}$
is determined by
\begin{eqnarray}
\label{ConstructSchwartz}
K_{A_0}(w) =\frac{1}{(2\pi)^{n}}
\int_{X \times T^{*}\!X}
e^{i\Phi(y,(x,\xi))} \mathfrak{a}(y, (x,\xi)) w(x,y)
\:\dif y \:\dif x \:\dif \xi
\end{eqnarray}
holds for any $w\in C_c^{\infty}( (\bb{V}^{1})^{*} \boxtimes \bb{V}^0)$.
The simplest example for Schwartz kernel is the Dirac-delta
distributional function for the identity operator
$K_{\id}(x,y) = \delta_{x,y}$.
\item[Properly supported pseudo-differential operators.]
$A_0 \in \Pseudo^{m}(\VVVVVV)$ is said to be properly supported
or has proper support if
both of the maps $p_1$ and $\map{p_2}{\supp(K_{A_0})}{X}$
are proper, where $p_i$ denotes the projection onto the $i$-th
factor,
in other words, both of the subsets
$$
(\supp(K_{A_0}))\cap (K\times X) \qqand
(\supp(K_{A_0}))\cap (X\times K) \subset X \times X
$$
are compact for any compact subset $K \subset X$.
The sub-space of $\Pseudo^{m}(\VVVVVV)$
consisting of all properly supported pseudo-differential operators
will be denoted by $\PPseudo^{m}(\VVVVVV)$.
\item[Symbol functions.]
For $m\in \R$,
we define $\Sym^{m}(\VVVVVV)$ as the space
consisting of all sections 
$\map{\sigma}{T^{*}\!X}{\Hom(\pi^{*}\bb{V}^0, \pi^{*}\bb{V}^1)}$
satisfying that for any compact subset
$K \subset X$ and multi indices
$a=(a_1,\ldots, a_n), b=(b_1,\ldots, b_n)$, where $n=\dim X$,
there exists a constant
$C_{a,b,K}$ is a constant depending on $K$, $a$ and $b$
such that
\begin{eqnarray}
\norm{
\pardif{|a|}{x^a}{}\pardif{|b|}{\xi^b}{}
\sigma(x,\xi)
}
\leq
C_{a,b,K}\sqrt{1+\norm{\xi}^2}^{\>m-|b|}
\qquad\text{holds for any }
x\in K,\; \xi \in T_{x}X
\label{ConditionOfSymbol}.
\end{eqnarray}
An element $\sigma \in \mathrm{Sym}^{m}(\VVVVVV)$
is called a symbol function of order $m$
and
$\sigma \in [\Sym^{m}(\VVVVVV)]:= \Sym^{m}(\VVVVVV)
\big/ \Sym^{m-1}(\VVVVVV)$
is called a principal symbol.
For example, every polynomial function in $\xi$ of order $m$
is an element of $\Sym^{m}(\VVVVVV)$.
\item[Principal symbol of pseudo-differential operators.]
Let 
$A_0 \in \PPseudo^{m}(\VVVVVV)$ 
be a properly supported pseudo-differential operator 
of order $m$.
Take a sufficiently small Euclidean neighborhood $W$
and local co-ordinate $\map{\psi}{W}{\R^{n}}$.
Then the symbol $\sigma_{A_0}$ of $A_0\in \PPseudo(\VVVVVV)$
is defined by the formula
\begin{eqnarray}
\sigma_{A_0}(x,\xi) = e^{-i\angles{\psi(x),\xi}}A_0 e^{i\angles{\psi(x),\xi}}
\qqfor
x\in W,\; \xi \in \R^{n}.
\label{DfnPrincipalSymbol}
\end{eqnarray}
$\sigma_{A_0}$ belongs to $\Sym^{m}(\VVVVVV)$, and
$\sigma_{A_0}$ regarded as an element of $[\Sym^{m}(\VVVVVV)]$
is called the principal symbol of $A_0$.
We may abbreviate principal symbol as just symbol.
\end{description}
\end{dfn}
\begin{rem}\text{}
\begin{enumerate}
\item
Let us consider an operator on $\R^{n}$.
If $\map{\varphi} {\R^{n}} {\R^{n}}$ is a diffeomorphism,
then the principal symbol of the operator
$(\varphi^{*})^{-1}A_0\varphi^{*}$ is 
$$
\sigma_{(\varphi^{*})^{-1}A_0\varphi^{*}}(x,\xi)
=
\sigma_{A_0}(\varphi(x), \varphi^{*}(\xi))
\mod \Sym^{m-1}(\VVVVVV).
$$
This implies that the principal symbol
$\sigma_{A_0} \in [\Sym^{m}(\VVVVVV)]$ defined by
(\ref{DfnPrincipalSymbol})
does not depend on the choice of local co-ordinates. 
\item
If $A_0$ is given by
(\ref{ConstructPseudo}), its principal symbol
is calculated from $\mathfrak{a}$ by
$$
\sigma_{A_0}(x,\xi) = \mathfrak{a} (x,(x,\xi)) \quad
\in [\Sym^{m}(\VVVVVV)].
$$
Conversely a symbol function $\sigma\in \Sym^{m}
(\VVVVVV)$ has an amplitude\\
$\map{\mathfrak{a}}{X\times T^{*}\!X}
{\Hom(\pi^{*}\bb{V}^0, \pi^{*}\bb{V}^1)}$
defined by
\begin{eqnarray}
\label{ConstructAmplitude}
\mathfrak{a}(y,(x,\xi))
:=
\chi(x,y)
\sigma(x, \xi),
\end{eqnarray}
where $\map{\chi}{X\times X}{[0,1]}$ is a smooth function
satisfying that $\chi(x,x) = 1$ and $\chi(x,y) = 0$
if $\mathrm{dist}(x,y)>r$, where $r$ is the injective radius of $X$,
which is bounded below.
We can construct a pseudo differential operator
by (\ref{ConstructPseudo}) whose principal symbol is 
equal to the given $\sigma \in [\Sym^{m}(\VVVVVV)]$.
\item
If $A_0 \in \Pseudo^{m_0}(\VVVVVV)$ and
$A_1 \in \Pseudo^{m_1}(\bb{V}^{1}, \bb{V}^{2})$, then we have
\begin{eqnarray*}
\sigma_{A_1 A_0} = \sigma_{A_1}\sigma_{A_0}
&&\mod \Sym^{m_0+m_1 -1}(\bb{V}^{0},\bb{V}^{2})
\\
\qqand
\sigma_{A_0^{*}} = (\sigma_{A_0})^{*} &&\mod \Sym^{m_0 -1}
(\bb{V}^{0},\bb{V}^{1}).
\end{eqnarray*}
\item
Let us start with a symbol $\sigma \in [\Sym^{m}(\VVVVVV)]$.
When we construct an amplitude (\ref{ConstructAmplitude}),
if we replace $r$ by the minimum among the injective radius
and $1$ so that $\supp(\chi)$ is in bounded distance 
from the diagonal in $X\ \times X$,
any pseudo-differential operator constructed by the formula
(\ref{ConstructAmplitude}) and (\ref{ConstructPseudo})
from a given symbol is always properly supported.
This is because
if $w$ in (\ref{ConstructSchwartz}) has a support
outside the $r$-neighborhood of
the diagonal in $X  \times X$,
$\mathfrak{a}(y, (x,\xi)) w(x,y) = 0$ for any $(x,y) \in X\times X$
therefore
$K_{A_0}$ has a support contained  in the $r$-neighborhood of
the diagonal in $X  \times X$.
So both $(\supp(K_{A_0}))\cap (K\times X)$ and
$(\supp(K_{A_0}))\cap (X\times K) $ are 
contained in a $2r$-neighborhood of $K \times K$ which is
bounded.
\item
For any properly supported pseudo-differential operator
$A_0 \in \Pseudo^{m}(\VVVVVV)$,
the image of $C_c(\bb{V}^{0})$ is contained in $C_c(\bb{V}^{1})$.
For these reasons
we basically assume that $A_0$ is properly supported.
\end{enumerate}
\end{rem}
\begin{dfn}[Ellipticity]
An operator $A_0 \in \PPseudo^{m}(\VVVVVV)$ is called Elliptic if 
the principal symbol $\sigma_{A}(x,\xi)$ is invertible at
infinity in $\xi$, more precisely,
there exists a symbol $\tau \in [\Sym^{-m}(\bb{V}^{1}, \bb{V}^{0})]$
such that
\begin{eqnarray}
\lim_{\norm{\xi}\to \infty}
\norm{\tau (x,\xi)\sigma_{A_0}(x,\xi) - \id_{\bb{V}^{0}}}
=
\lim_{\norm{\xi}\to \infty}
\norm{\sigma_{A_0}(x,\xi)\tau(x,\xi) - \id_{\bb{V}^{1}}}
=0
\label{EllipticityInv}
\end{eqnarray}
uniformly in $x\in K$ for any compact subsets $K \subset X$.
\end{dfn}
\begin{rem}\text{}
\begin{enumerate}
\item
Without loss of generality
we may consider the case of elliptic operators of 
order $0$
and furthermore the condition (\ref{EllipticityInv})
may be replaced by the condition that
the principal symbol is unitary at infinity, that is,
\begin{eqnarray}
\lim_{\norm{\xi}\to \infty}
\norm{\sigma_{A_0}^{*}(x,\xi)\sigma_{A_0}(x,\xi) - \id_{\bb{V}^{0}}}
=
\lim_{\norm{\xi}\to \infty}
\norm{\sigma_{A_0}(x,\xi)\sigma_{A_0}^{*}(x,\xi) - \id_{\bb{V}^{1}}}
=0
\label{EllipticityUnitary}
\end{eqnarray}
uniformly in $x\in K$ for any compact subsets $K \subset X$.
This is because
$\sigma_{A_0} \in [\Sym^{0}(\VVVVVV)]$ can be modified to
$$
\sigma' := \parens{\sigma_{A_0}\sigma_{A_0}^{*}}^{-1/2}\sigma_{A_0}
$$
at infinity 
so that $\sigma'$ satisfies the condition 
(\ref{EllipticityUnitary})
through a homotopy by
a family of symbols
$\{
\parens{\sigma_{A_0}\sigma_{A_0}^{*}}^{-t/2}\sigma_{A_0}
\}_{t\in [0,1]}$
which are invertible, at infinity in $\xi$.
\item
In the case of elliptic operators $A$ on
$\bb{V} = \bb{V}^{0} \widehat{\oplus} \bb{V}^{1}$
of the form
$\displaystyle{
A=\matrixtwo
{0} {A_0^{*}}
{A_0} {0}
}$,
the requirement (\ref{EllipticityUnitary})
will be replaced by
$\displaystyle{
\lim_{\norm{\xi}\to \infty}
\norm{\sigma_{A}(x,\xi)^{2} - \id_{\bb{V}}}
=0}$.
\end{enumerate}
\end{rem}
Let us recall the action of $G$.
\begin{dfn}[$G$-invariant operators]
$G$ acts on $\PPseudo^{m}(\VVVVVV)$
by the formula;
$$
\gm[A_0] = \gm A_0 \gminv.
$$
Let $\PPseudo_{G}^{m}(\VVVVVV)$ denote the subspace of
$\PPseudo^{m}(\VVVVVV)$
consisting of all $G$-invariant operators.
\end{dfn}
\begin{rem}
It directly follows that the action of $G$ on the operators 
are compatible with the action on the sections, that is,
$\gm[A_0 s] = \gm[A_0]\>\gm[s]$ for $s\in C_c^{\infty}(\bb{V}^{0})$.

Note that the action of $G$ on $\bb{V}^0 \widehat{\oplus} \bb{V}^1$
induces the action on the bundle 
$\Hom(\pi^{*}\bb{V}^0, \pi^{*}\bb{V}^1)$ over $T^{*}\!X$.
Any $G$-invariant operator $A_0 \in \PPseudo_{G}^{m}(\VVVVVV)$
has the $G$-invariant principal symbol.
Conversely we have the following;
\end{rem} 
\begin{lem}
If $\map{\sigma}{T^{*}\!X}{\Hom(\pi^{*}\bb{V}^0, \pi^{*}\bb{V}^1)}$
is a $G$-invariant principal symbol of oder $m$,
then there exists a $G$-invariant operator
$A \in \PPseudo_{G}^{m}(\VVVVVV)$
whose principal symbol is $\sigma$.
\end{lem}
\begin{proof}
When we take an amplitude $\mathfrak{a}$ of $\sigma$
by (\ref{ConstructAmplitude}),
choose $\map{\chi}{X\times X}{[0,1]}$
so that $\chi(x,y)$ depends only on $\mathrm{dist}(x,y)$.
More precisely,
let $\map{\bar{\chi}}{\R_{\geq 0}}{[0,1]}$
be a smooth function such that
$\bar{\chi}(t)=1$ for $t\leq r/4$ and
$\bar{\chi}(t)=0$ for $r\geq r/2$, where $r$ is a constant 
less than the injective radius.
and set $\chi(x,y) := \bar{\chi}(\mathrm{dist}(x,y))$.
Since
the action of $G$ is isometric,
the amplitude $\mathfrak{a}(y,(x,\xi))$ constructed
by (\ref{ConstructAmplitude}) using this $\chi$
is $G$-invariant.
Then the pseudo-differential operator
$A\in \PPseudo^{m}(\VVVVVV)$
constructed by the formula
(\ref{ConstructPseudo})
is $G$-invariant.
\end{proof}
We will construct a $K$-homology classes
$[A] := (L^{2}(\bb{V}), \: A) \in K\!K^{G}(C_0(X), \C)$
induced by properly supported $G$-invariant elliptic operators.
in the last of this subsection.
In order to argue the $L^{2}$ boundedness of elliptic operators,
let us introduce the following theorem by H\"{o}mander;
\begin{thm}\label{ThmHormander}
{\upshape \cite[Theorem 2.2.1]{Ho71}}
Let $P\in \PPseudo^{0}(\VVVVVV)$ be a properly supported 
pseudo-differential operator of order $0$.
Suppose that for any compact subset $K \subset X$
$$
\limsup_{\norm{\xi} \to \infty}\norm{\sigma_{P}(x,\xi)}<C
\qqfor x\in K.
$$
Then there exists a properly supported pseudo differential operator
$B\in \PPseudo^{0}(\VVVVVV)$ of order $0$
and a self adjoint integral operator $R$ with continuous and properly
supported kernel such that
$$
PP^{*} + BB^{*} - C^{2} = R.
$$
Moreover, if the support of the Schwartz kernel of $P$ is compact
in $X \times X$, then both $B$ and $R$ will also have
compactly supported Schwartz kernels.
\QED
\end{thm}
\begin{lem}
\label{LemCoMo}
Let $P\in \mathbb{B}(L^{2}(\bb{V}^{0}),L^{2}(\bb{V}^{1}))$ be 
a bounded operator with Schwartz kernel
supported in a compact subset $K\times K \subset X \times X$.
Then $\widetilde{P}:=\int_{G} \gm[P] \:\dif \gm$ is well defined as
a bounded operator in $\mathbb{B}(L^{2}(\bb{V}^{0}),L^{2}(\bb{V}^{1}))$
and the inequation
$\norm{\widetilde{P}}_{\mathrm{op}} \leq C_{K} \norm{P}_{\mathrm{op}}$
holds,
where $C_K$ is a constant depending on $K$.
\end{lem}
\begin{proof}
We will follow the proof of \cite[Lemma 1.4 - 1.5]{Co-Mo82}.
Fix an arbitrary smooth section with compact support
$s\in C_c^{\infty} (\bb{V}^{0})$ and let us consider
$F=F_s \in L^{2}\parens{G;L^{2}(\bb{V}^{1})}$
given by
$$
F_s(\gm) := \gm[P]s.
$$
Note that
for any $\gm \in G$
the Schwartz kernel of $\gm[P]$ is contained in
$\gm(K)\times \gm(K)$.
This is because for any $s\in C_c^{\infty}(\bb{V}^{0})$,
$\supp(\gm[P]s) \subset \gm(K)$ and  $\gm[P]s=0$
whenever $\supp(s) \cap \gm(K) = \emptyset$.
In particular, since the action of $G$ on $X$ is proper,
$F$ has compact support in $G$ and hence $F \in 
L^{2}\parens{G;L^{2}(\bb{V}^{1})}$.
In addition, since
the action of $G$ on $X$ is proper,
$\gm(K) \cap \eta(K) = \gm( K \cap \gminv \eta (K))
= \emptyset$
if $\gminv \eta \in G$ is outside some compact neighborhood
$Z \subset G$ in particular,
$\angles{F(\gm), F(\eta)}_{L^{2}(\bb{V}^{1})} = 0$
for such $\gm$ and $\eta \in G$.
Remind that $Z$ is determined only by $K$
so independent of $s$.
Then,
\begin{eqnarray*}
\left\Vert
\int_{G}F(\gm) \:\dif \gm
\right\Vert_{L^{2}(\bb{V}^{1})}^{2}
&=&
\norm{
\int_{G}F(\gm) \:\dif \gm
}_{L^{2}(\bb{V}^{1})}
\norm{
\int_{G}F(\eta) \:\dif \eta
}_{L^{2}(\bb{V}^{1})}
\\
&\leq&
\int_{G}\int_{G}
\Big| \angles{F(\gm), F(\eta)}_{L^{2}(\bb{V}^{1})} \Big|
 \:\dif \gm  \:\dif \eta
\\
&\leq&
\int_{G} \norm{F(\gm)}_{L^{2}(\bb{V}^{1})}
\parens{
\int_{G}
\chi_{Z}(\gminv \eta) \norm{F(\eta)}_{L^{2}(\bb{V}^{1})}
 \:\dif \eta } \dif \gm
\\
&\leq&
\norm{F}_{L^{2}(G)}
\norm{\chi_{Z}}_{L^{2}(G)}
\norm{F}_{L^{2}(G)}
\\
&\leq&
|Z|\norm{F}_{L^{2}(G)}^{2},
\end{eqnarray*}
where $\map{\chi_{Z}}{G}{[0,1]}$ is the characteristic function of $C$,
that is $\chi_{Z}(\gm) = 1$ for $x \in Z$ and
$\chi_{Z}(\gm) = 0$ for $x \notin Z$.

On the other hand, take a compactly supported smooth function
$f\in C_c^{\infty}(X;[0,1])$ such that $f(x) = 1$ for $x \in K$
and $f(x) = 0$ if $\mathrm{dist}(x,K)>r$.
Since the support of the Schwartz kernel of $P$ is contained in 
$K \times K$ and hence $P = Pf$, we have
\begin{eqnarray*}
\norm{F}_{L^{2}(G)}^{2}
=
\int_{G}\norm{F(\gm)}_{L^{2}(\bb{V}^{1})}^{2} \:\dif \gm
&=&
\int_{G}\norm{
\gm Pf \gminv s
}_{L^{2}(\bb{V}^{1})}^{2} \:\dif \gm
\\
&\leq&
\int_{G}\norm{P}_{\mathrm{op}}^{2}\norm{
f \gminv s
}_{L^{2}(\bb{V}^{0})}^{2} \:\dif \gm
\\
&\leq&
\norm{P}_{\mathrm{op}}^{2}
\int_{G} \int_{X}
|f(x)|^{2} \norm{\gminv s(x)}_{\bb{V}^{0}}^{2}
 \:\dif x \:\dif \gm
\\
&\leq&
\norm{P}_{\mathrm{op}}^{2}
\int_{G} \int_{X}
|f(\gminv x)|^{2} \norm{s(x)}_{\bb{V}^{0}}^{2}
 \:\dif x \:\dif \gm
\\
&\leq&
\norm{P}_{\mathrm{op}}^{2}
\sup_{x\in X}\parens{\int_{G}
|f(\gminv x)|^{2} 
\:\dif \gm}
\norm{s}_{L^{2}(\bb{V}^{0})}^{2}.
\end{eqnarray*}
Since the action of $G$ is proper, 
$\braces{ \gm \in G \: \big| \: \gminv x \in \supp(f)} \subset G$ 
is compact so the value $\int_{G}
|f(\gminv x)|^{2} 
\:\dif \gm$ is always finite for any fixed $x\in X$.
Besides, since $X/G$ is compact, we have
$$\sup_{x\in X}\parens{\int_{G}
|f(\gminv x)|^{2} 
\:\dif \gm}
=
\sup_{[x]\in X/G}\parens{\int_{G}
|f(\gminv x)|^{2} 
\:\dif \gm}
< \infty$$
Remind that
this value depends only on $K$ and independent of $s$.
We conclude that
$$
\norm{\int_{G}\gm[P]s \:\dif \gm}_{L^{2}(\bb{V}^{1})}^{2}
=
\norm{\int_{G}F(\gm) \:\dif \gm}_{L^{2}(\bb{V}^{1})}^{2}
\leq
|Z|\norm{F}_{L^{2}(G)}^{2}
\leq
|Z|C \cdot \norm{P}_{\mathrm{op}}^{2}
\norm{s}_{L^{2}(\bb{V}^{0})}^{2},
$$
where $C=\sup_{x\in X}\parens{\int_{G}
|f(\gminv x)|^{2} 
\:\dif \gm}< \infty$.
\end{proof}
\begin{lem}
\label{BoundedElliptic}
Let $A_0 \in \PPseudo^{0}_{G}(\VVVVVV)$ be
an elliptic operator of order $0$ which is properly supported and
$G$-invariant.
If $X/G$ is compact, then $A_0$ is $L^{2}$-bounded.
\end{lem}
\begin{proof}
$A_0$ satisfies the assumptions of theorem
\ref{ThmHormander}. Let $c \in C_c^{\infty}(X)$ be
an arbitrary cut-off function. And let us consider $P:=cA_0$.
it has the Schwartz kernel supported in
a compact subset
$\parens{\supp(c) \times X} \cap \supp(K_{A_0})
\subset X \times X$ and still satisfies the 
assumptions of the proposition.
Then we have
$$
0 \leq PP^{*} + BB^{*} = C^{2} + R.
$$
And it follows that $P$ is bounded
since $BB^{*}$ is positive and $C^{2} + R$ is bounded.

Next, note that $A_0 = \int_{G} \gm[cA_0] \:\dif \gm
= \int_{G} \gm[P] \:\dif \gm$ since $A_0$ itself is $G$-invariant.
Then by the previous lemma \ref{LemCoMo},
$A_0$ is $L^{2}$ bounded.
\end{proof}
\begin{lem}\label{CommutatorWithMultiplication}
\upshape{\cite[Lemma 3.7]{Ka}}
Let $A=
\matrixtwo
{0}{A_0^{*}}
{A_0}{0}
\in \PPseudo_{G}^{0}(\bb{V})$ be a properly supported $G$-invariant
elliptic operator on
$\bb{V} = \bb{V}^{0} \widehat{\oplus} \bb{V}^{1}$.
Then for any $f\in C_0(X)$, the multiplication operators by $f$
commutes with $A$ modulo compact operators $\bb{K}(L^{2}(\bb{V}))$,
and $f(1-A^{2})$ is compact.
\end{lem}
\begin{proof}
It is enough to prove when $f$
has a compact support.
The commutator $Af-fA$ is a pseudo-differential operator of
negative order
because the principal symbol of order $0$ is
$$
\sigma_{Af-fA}=\sigma_{A}f-f\sigma_{A} = 0 \in [\Sym^{0}(\bb{V})].
$$
Also, the symbol of $f(1-A^{2})$ satisfies
$$
\sigma_{f(1-A^{2})}(x) = f(x)(1-\sigma_{A^{2}}(x,\xi)) \to 0
\qquad
\text{as}
\quad \norm{\xi}\to\infty.
$$
Let $P$ be $Af-fA$ or $f(1-A^{2})$.
Then $P$ has the compactly supported Schwartz kernel and
satisfies the assumption of 
theorem \ref{ThmHormander} for any small constant $C>0$.
We obtain
$$
PP^{*} + BB^{*} = R
$$
with $R$ having the compactly supported Schwartz kernel.
Hence $R$ is a compact operator, which implies that so is $P$.
\end{proof}
\begin{rem}
Note that the symbol $\sigma_{A}
=
\matrixtwo
{0}{\sigma_{A_0}^{*}}
{\sigma_{A}}{0}
$
of a properly supported $G$-invariant elliptic operator $A$
of order $0$
induces a map
$\map{\sigma_{A}}{\pi^{*}\bb{V}}{\pi^{*}\bb{V}}$,
where $\map{\pi}{T^{*}\!X}{X}$
denotes the projection.
Additionally, note that
the space of $C_0(T^{*}\!X)$-compact operators on
$C_0(\pi^{*}\bb{V})$ is $C_0(\End(\pi^{*}\bb{V}))$
and that $C_0(X)$ acts on $C_0(\pi^{*}\bb{V})$
by the point-wise multiplication.
Hence, for any $f \in C_0(X)$
$
\sigma_{A}f - f\sigma_{A}( = 0)
$
and $
f(\sigma_{A}^2-1)
$
are compact as an operator on a Hilbert $C_0(T^{*}\!X)$-module.
\end{rem}
\begin{dfn}[The K-homology classes determined by elliptic operators]
\label{DfnKOfElliptic}
\text{}
\begin{enumerate}
\item
Let
$A=
\matrixtwo
{0}{A_0^{*}}
{A_0}{0}
\in \PPseudo_{G}^{0}(\bb{V})$ 
be a properly supported
$G$-invariant elliptic operator of order $0$
on a graded Hermitian $G$-vector bundle 
$\bb{V} = \bb{V}^{0} \widehat{\oplus} \bb{V}^{1}$ over $X$.
Due to the lemma \ref{BoundedElliptic} and 
\ref{CommutatorWithMultiplication}, we can define 
a $K$-cohomology class
$$[A] := (L^{2}(\bb{V}), \: A) \in K\!K^{G}(C_0(X), \C)$$
with the action of $C_0(X)$ on $L^{2}(\bb{V})$ being the
point-wise multiplication.
\item
Its symbol $\sigma_{A}
=
\matrixtwo
{0}{\sigma_{A_0}^{*}}
{\sigma_{A}}{0}
$
 determines the following element in $K\!K$-theory;
$$
[\sigma_{A}]:=
\parens{
C_0(\pi^{*}\bb{V}), \: \sigma_{A}
}
\in
K\!K^{G}(C_0(X), C_0(T^{*}\!X)).
$$
\end{enumerate}
\end{dfn}
\section{Index Maps}
\subsection{The Analytic Index Map}
\settingsGXa
Let $\bb{V}$ be a smooth (finite dimensional)
Hermitian $G$-vector bundle over $X$.
We summarize the notations;
\begin{itemize}
\item
$C^{\infty}_c(\bb{V})$ denotes the space of 
compactly supported smooth sections $\map{s}{X}{\bb{V}}$.
\item
$C_0(\bb{V})$ denotes the $C^{*}$-algebra consisting of 
continuous sections $\map{s}{X}{\bb{V}}$ vanishing at infinity
equipped with the sup-norm.
\item
$L^{2}(\bb{V})$ denotes the Hilbert $C_0(\bb{V})$-module
consisting of $L^{2}$-sections $\map{s}{X}{\bb{V}}$.
$C_0(\bb{V})$ acts by point-wise multiplication.
\end{itemize}
\begin{dfn}[Analytical Index]
Let us consider additionally the grading on $\bb{V}$, that is
$\bb{V}=\bb{V}^{0} \widehat{\oplus} \bb{V}^{1}$.
Let
$$
\map{A=
\matrixtwo
  {0}{A_0^{*}}
  {A_0}{0}
}
{L^{2}(\bb{V})}{L^{2}(\bb{V})}
$$
be a properly supported
$G$-invariant elliptic operator of order $0$.
We define the analytical index as
$$
\ind_{C^{*}(G)} (A) :=
[\Ker (A_0)]
-
 [\Ker (A_0^{*})]
\in K_{0}(C^{*}(G)).
$$
\end{dfn}
\subsection{The Topological Index Map}
\settingsGXa
Let $A$ be a properly supported
$G$-invariant elliptic operator of order $0$ on a graded Hermitian 
$G$-vector bundle $\bb{V}$.
Then as in definition \ref{DfnKOfElliptic},
$A$ defines the $K$-cohomology class
$$
[A] = \parens{L^{2}(\bb{V}),m, A} \in K\!K^{G}(C_0(X),\C),
$$
where the action $m$ of $C_0(X)$ on $L^{2}(\bb{V})$
is ordinary point-wise multiplication.
We would like to define the index map
$$
\map{\mu_{G}}{K\!K^{G}(C_0(X),\C)}{K_{0}(C^{*}(G))}.
$$
\begin{dfn}
\label{DfnjG1}
Let $G$ be a second countable locally compact Hausdorff group,
$A$ and $B$ be $G$-algebras and
$(\mathcal{E},\varphi, T) \in \bb{E}^{G}(A,B)$
be a Kasparov $G$-module over $(A,B)$.
We will introduce the following notations;
\begin{enumerate}
\item
$\widehat{T}$ is defined to be an operator on
$C^{*}(G;\mathcal{E})$ given by
$$
\widehat{T}[\widehat{e}](\gamma) = T(\widehat{e}(\gamma))
\qqfor \widehat{e} \in C_c(G;\mathcal{E}).
$$
\item
The action $\widehat{\varphi}$ of $C^{*}(G;A)$
on $C^{*}(G; \mathcal{E})$
is given by the following convolution product;
$$
\widehat{\varphi}[\widehat{a}](\widehat{e})
\colon \gamma \mapsto 
\int_{\gm_1 \gm_2=\gm}\varphi(\widehat{a}(\gm_1))
\gm_1(\widehat{e}(\gm_2))
\mathrm{d}\gm_1
\qqfor
\widehat{a} \in C_c(G;A),\quad \widehat{e}\in C_c(G;\mathcal{E}).
$$
\end{enumerate}
\end{dfn}
\begin{lem}
{\upshape{\cite[Theorem 3.11]{Ka88}}}
Let $G$ be a second countable locally compact Hausdorff group.
For any $G$-algebras $A$ and $B$ there exists
a natural homomorphism
$$
\map{j^{G} } {K\!K^{G}(A,B)} {K\!K(C^{*}(G;A), C^{*}(G;B))}
$$
constructed as follows;
If $x=(\mathcal{E},\varphi, T) \in K\!K^{G}(A,B)$,
$$j^{G}(x) = \parens{C^{*}(G;\mathcal{E}),
\widehat{\varphi}, \widehat{T}}.$$
Furthermore
if $x\in K\!K^{G}(A,B)$ and $y\in K\!K^{G}(B,D)$,
then $j^{G}(x\hattensor_{B}y) = j^{G}(x) \hattensor_{C^{*}(G;B)}
j^{G}(y)$.
\QED
\end{lem}
\begin{lem}
\label{CutoffKHomologyClass}
\settingsGXa
Using a cut-off function $c\in C_c(X)$,
one can define an idempotent $p \in C_c(G;C_0(X))$
by the formula;
$$\check{c}(\gm)(x) = \sqrt{c(x) c(\gminv x)}.$$
In particular it defines an element of K-homology
denoted by $[c]\in K_0(C^{*}(G;C_0(X)))$.
Moreover the element of $K$-homology
$[c]\in K_0(C^{*}(G;C_0(X)))$ does not depend on
the choice of cut-off functions.
\end{lem}
\begin{proof}
It is verified by the direct calculation that
$\check{c}$ is an idempotent;
\begin{eqnarray*}
(\check{c}*\check{c})(\gm)(x)
&=&
\int_{\gm_1\gm_2 =\gm}
\check{c}(\gm_1)(x)\cdot \gm_1[p(\gm_2)](x) \mathrm{d}\gm_1
\\
&=&
\int_{\gm_1\gm_2 =\gm}
\sqrt{c(x) c(\gm_1 ^{-1}x)} 
\sqrt{c(\gm_1 ^{-1}x) c(\gm_2 ^{-1}(\gm_1 ^{-1}x))}
\mathrm{d}\gm_1
\\
&=&
\sqrt{c(x) c(\gminv x)}
\int_{\gm_1\gm_2 =\gm} c(\gm_1 ^{-1}x) \mathrm{d}\gm_1
\\
&=&
\check{c}(\gm)(x).
\end{eqnarray*}
As for the independence of the choice of cut-off function,
it is followed from the fact that the subset in $C_c(X)$
of all cut-off functions on $X$ is convex.
Precisely,
if we have two cut-off functions $c_0$ and $c_1$,
$c_t := (1-t)c_0 + tc_1$ provides a family of cut-off
functions for $t\in[0,1]$.
$\check{c_t}$ provides a homotopy by idempotents connecting
$\check{c_0}$ and $\check{c_1}$.
\end{proof}

\begin{dfn}[Topological Index]\label{BCmap}
Define
$$
\map{\mu_{G}}{K\!K^{G}(C_0(X), \C)}{K_{0}(C^*(G))}
$$
as the composition of
\begin{itemize}
\item
$\map{j^{G}}{K\!K^{G}(C_0(X), \C)}
{K\!K(C^*(G;C_0(X)) , C^*(G))}$
with
\item
$\map{[c]\hattensor}{K\!K(C^*(G;C_0(X)) , C^*(G))}
{K\!K(\C , C^*(G))} \simeq K_{0}(C^*(G))$.
\end{itemize}
In other words,
\begin{eqnarray*}
\mu_{G}(\mathchar`-) :=
[c]\hattensor_{C^*(G;C_0(X))} j^{G}(\mathchar`-)
\quad \in K_{0}(C^*(G) ).
\end{eqnarray*}
\end{dfn}

\begin{prop}
{\upshape{\cite[Theorem 5.6.]{Ka}}}
$$\ind_{C^{*}(G)} (A) = \mu_{G}([A]).$$
\end{prop}
\newcommand{\SAA}{\sigma_{A}}
\newcommand{\SQAA}{\sqrt{1-\sigma_{A}^{2}}}
\subsection{Reduction to the Dirac Operators}
\label{SubsectionReductionOnDirac}
{\upshape{\cite[Remark 4.4.]{Ka}}}
One can reduce the general case of index theorem
for elliptic operators to the case of Dirac operators.
In this subsection, we will see that
any $K$-homology class
$[A]\in K\!K^{G}(C_0(X),\C)$, where $A$ is a
properly supported $G$-equivariant elliptic operator 
of order $0$ can be realized as
a Dirac operator $\mathcal{D}$ on a larger manifold
$Y$ obtained by a fiber bundle over $X$.
To be more specific, we will obtain
$[A] = (\pi^{\sharp})^{*}[\mathcal{D}] $ with the map
$\map{(\pi^{\sharp})^{*}}{K\!K^{G}(C_0(Y),\C)}{K\!K^{G}(C_0(X),\C)}$
induced by $\map{\pi^{\sharp}}{C_0(X)}{C_0(Y)}$.
If this has been done, it is sufficient to consider the case when
$A$ is a Dirac operator on a Clifford bundle $\bb{V}$ over $Y$
and
$[A] = \parens{L^{2}(\bb{V}) ,\
\frac{A}{\sqrt{1+A^2}}
} \in K\!K^{G}(C_0(Y), \C)$.

\begin{dfn}
Let $BX \subset T^{*}X$ be the unit closed disk bundle
in the cotangent bundle, whose boundary
is the sphere bundle $\partial BX = SX$.
Then construct $$\Sigma {X} = BX \cup_{SX} BX$$
gluing two copies of $BX$ together along $SX$.
The zero in the second $BX$ will be denoted by $\infty$.

In order to avoid the confusion, let us distinguish the notations 
of projections $\map{\pi_{T}}{T^{*}X}{X}$ and
$\map{\pi_{\Sigma}}{\Sigma {X}}{X}$.
\end{dfn}
The action of $G$ on $X$
induces the proper co-compact action on $\Sigma {X}$
because the action is isometric and each fiber of $\pi_{\Sigma}$
is compact.
This action commutes with the projection $\pi_{\Sigma}$.
Let $\map{\pi_{\Sigma}^{\sharp}}{C_0(X)}{C_0(\Sigma {X})}$
be the homomorphism induced by the projection $\pi_{\Sigma}$.

In addition, note that $\Sigma {X}$ is of a structure of 
smooth almost complex manifold.
\begin{dfn}
Let $\mathcal{D}_{\Sigma {X}}$
be the Dolbeault operator on $\Sigma {X}$.
In the natural way, it determines the following $K$-homology class;
$$
[\mathcal{D}_{\Sigma {X}}]
:=
\parens{L^{2}\parens{\turnV{\!0,*}T^{*}(\Sigma {X})}
,\
\frac{\mathcal{D}_{\Sigma {X}}}
{\sqrt{1+\mathcal{D}_{\Sigma {X}}^2}}
}
 \in K\!K^{G}(\Sigma {X}, \C).
$$
\end{dfn}

Now let $A={\scriptsize{
\matrixtwo
{0}{A_0^{*}}
{A_0}{0}
 }}
\in \PPseudo_{G}^{0}(\bb{V})$ be
a properly supported elliptic operator of order $0$
on a graded Hermitian $G$-vector bundle
$\bb{V}=\bb{V}^{0} \widehat{\oplus} \bb{V}^{1}$.
Without loss of generality
we may assume its symbol $\sigma_{A}$ satisfies 
\begin{eqnarray}
\sigma_{A}^{2}(x,\xi) = 1\quad \text{for }\norm{\xi}\geq 1
\qqand
\norm{\sigma_{A}^{2}} \leq 1 \quad \text{everywhere.}
\label{AssumptionOnsigma}
\end{eqnarray}
In particular $\sigma_{A}|_{SX}
\colon
\parens{\pi_{T}^{*}\bb{V}}_{SX} \to \parens{\pi_{T}^{*}\bb{V}}_{SX}$
is invertible
then it provides a gluing map
 from $\parens{\pi_{T}^{*}\bb{V}}_{BX}$
to $\parens{\pi_{T}^{*}\bb{V}}_{BX}$.
From now on, we will abbreviate
$\sigma_{A}
=
\matrixtwo
{0}{\sigma_{A_0}^{*}}
{\sigma_{A_0}}{0}
$
to
$\sigma
=
\matrixtwo
{0}{\! \sigma_{0}^{*}}
{\sigma_{0}}{\!0}
$
in order to avoid the complication.
\begin{dfn}
\text{}
\begin{itemize}
\item
Let us construct a new bundle $\bb{V}(\sigma)$
over $\Sigma {X}$ by gluing
$\parens{\pi_{T}^{*}\bb{V}^{0}}_{BX}$
and $\parens{\pi_{T}^{*}\bb{V}^{1}}_{BX}$ together;
\begin{eqnarray*}
\bb{V}(\sigma) := (\pi_{T}^{*}\bb{V})_{BX}
\cup_{\sigma|_{SX}}
(\pi_{T}^{*}\bb{V})_{BX}.
\end{eqnarray*}
Remark that the gluing map $\sigma$
reverses the gradings, $\bb{V}(\sigma)$ has the grading
\begin{eqnarray}
\begin{array}{ccccc}
\bb{V}(\sigma) &=
& \bb{V}(\sigma)^{0}
&\widehat{\oplus}& \bb{V}(\sigma)^{1}
\\
&:=&
(\pi_{T}^{*}\bb{V}^{0})_{BX}
\cup_{\sigma_{0}|_{SX}}
(\pi_{T}^{*}\bb{V}^{1})_{BX}
&\widehat{\oplus}
&
(\pi_{T}^{*}\bb{V}^{1})_{BX}
\cup_{\sigma_{0}^{*}|_{SX}}
(\pi_{T}^{*}\bb{V}^{0})_{BX}.
\end{array}
\label{DfnBundleOverSigma}
\end{eqnarray}
\item
In addition consider another bundle just given by the pull-back
$\pi_{\Sigma}^{*}(\bb{V}^{1})$
over $\Sigma {X}$.
Remark that 
$\pi_{\Sigma}^{*}(\bb{V}^{1}) \cong  \bb{V}(\sigma)^{0}$
on the second $BX \subset \Sigma {X}$.

These $G$-vector bundles $ \bb{V}(\sigma)^{0}$
and $\pi_{\Sigma}^{*}(\bb{V}^{1})$
determine elements in $K\!K$-theory;
\begin{eqnarray*}
[\bb{V}(\sigma)^{0}]
&=&
\parens{C_0(\bb{V}(\sigma)^{0}) , 0} 
\\
\qqand
[\pi_{\Sigma}^{*}(\bb{V}^{1})]
&=&
\parens{C_0(\pi_{\Sigma}^{*}(\bb{V}^{1})) , 0} 
\in K\!K^{G} (
C_0(\Sigma {X}), C_0(\Sigma {X})).
\end{eqnarray*}
\item
Define an element of $K\!K$-group denoted by
$[[\sigma]]\in 
K\!K^{G} (
C_0(X), C_0(\Sigma {X}))$ as
\begin{eqnarray*}
[[\sigma]] &:=& 
(\pi_{\Sigma}^{\sharp})^{*}\brackets{
\bb{V}(\sigma)^{0}
\;\widehat{\oplus}\;
\pi_{\Sigma}^{*} (\bb{V}^{1})
}
\\
&=&
\parens{C_0(\bb{V}(\sigma)^{0}
\;\widehat{\oplus}\;
\pi_{\Sigma}^{*} (\bb{V}^{1})) , 0}
\in 
K\!K^{G} (
C_0(X), C_0(\Sigma {X})).
\end{eqnarray*}
Remark that we can define such class
because each fiber of $\Sigma X$ is compact,
in other words, $\pi_{\Sigma}$ is a proper map.
We cannot define a $K\!K$-theory class in
$K\!K^{G} (C_0(X), C_0(T^{*}\! X))$ of the same form.
\end{itemize}
\end{dfn}
Our main claim in this subsection is the following;
\begin{prop}
\label{PropTwistedDiracOnSigma}
$$
[[\sigma]]\hattensor_{C_0(\Sigma {X})}
[\mathcal{D}_{\Sigma {X}}]
=[A]
\in K\!K^{G}(C_0(X), \C).
$$
In particular, let $\mathcal{D}_{\sigma}$ be a Dirac operator
obtained by the twisting the canonical Dolbeault operator
$\mathcal{D}_{\Sigma {X}}$ by the $\Z / 2\Z$-graded bundle
$\bb{V}(\sigma)^{0}
\;\widehat{\oplus}\;
\pi_{\Sigma}^{*} (\bb{V}^{1})$.
Then it is a realization of initial $A$ through
$\map{\pi_{*}}{K\!K^{G}(C_0(\Sigma {X}), \C)}{K\!K^{G}(C_0(X), \C)}$,
in other words,
$$
(\pi_{\Sigma}^{\sharp})^{*} ([\mathcal{D}_{\sigma}])
=
[A]
\in K\!K^{G}(C_0(X), \C).
$$
\end{prop}

To verify this, we need theorem
\ref{ThmSymbolIndex} by G. Kasparov stated below.
\\

Firstly, let us introduce the $K$-cohomology class of $T^{*}\! X$
determined by the Dolbeault operator.
Since the manifold $T^{*}\!X$ has the canonical structure of an
almost complex manifold in particular,
$\mathrm{Spin}^{c}$-structure.
\begin{dfn}
Let $\mathcal{D}_{T^{*}\!X}$ be the canonical Dirac operator on 
$T^{*}\!X$ called the Dolbeault operator.
\end{dfn}
It will be regarded as an essentially self-adjoint operator
acting on $L^{2}\parens{\turnV{\!0,*}T^{*}(T^{*}\!X)}$.
\begin{lem}
Put $\mathcal{T}:= {\frac
{\mathcal{D}_{T^{*}\!X}}
{\sqrt{1+\mathcal{D}_{T^{*}\!X}^{2}}}
}$ and 
$\mathcal{H} := L^{2}\parens{\turnV{\!0,*}T^{*}(T^{*}\!X)} $.
Consider the action of $C_0(T^{*}\!X)$
on $\mathcal{H}$ by point-wise multiplication.
Then $(1-\mathcal{T}^{2})f$ and
$f\mathcal{T}-\mathcal{T}f$ are compact operators
on $\mathcal{H}$ for any $f \in C_0(T^{*}\!X)$.

In particular we can define the following $K$-homology class 
determined by Dolbeault operator on $T^{*}\!X$;
$$
[\mathcal{D}_{T^{*}\!X}]
:=
\parens{\mathcal{H} ,\ \mathcal{T}}
=
\parens{
L^{2}\parens{\turnV{\!0,*}T^{*}(T^{*}\!X)},\ 
{\frac
{\mathcal{D}_{T^{*}\!X}}
{\sqrt{1+\mathcal{D}_{T^{*}\!X}^{2}}}
}}
\in K\!K^{G}(C_0(T^{*}\!X), \C).
$$
\end{lem}
\begin{proof}
For a while, we will abbreviate $\mathcal{D}_{T^{*}\!X}$ as
$\mathcal{D}$.
It is sufficient to consider the case when $f \in C^{\infty}_c(T^{*}\!X)$.
Noting that
$1-\mathcal{T}^{2} = \parens{ 1+\mathcal{D}^{2} }^{-1}$ and
using the Rellich lemma,
$(1-\mathcal{T}^{2})f \in \mathbb{K}(\mathcal{H})$.
As for $f\mathcal{T}-\mathcal{T}f$, we use the formula;
\begin{eqnarray*}
\mathcal{T}
=
\frac{2}{\pi}\int_{0}^{+\infty}
\frac { \mathcal{D} } 
{ 1+\lambda^{2}+\mathcal{D}^{2} } \: \dif \lambda.
\end{eqnarray*}
This integral converges in the strong topology, not in
the operator norm. Then we have
\begin{eqnarray*}
f\mathcal{T}-\mathcal{T}f
&=&
\frac{2}{\pi}\int_{0}^{+\infty}\parens{
f \frac { \mathcal{D} } 
{ 1+\lambda^{2}+\mathcal{D}^{2} } 
-
\frac { \mathcal{D} } 
{ 1+\lambda^{2}+\mathcal{D} } f
}\: \dif \lambda 
\\
&=&
\frac{2}{\pi}\int_{0}^{+\infty}
\frac {1}{ 1+\lambda^{2}+\mathcal{D}^{2} }
\brackets{
(1+\lambda^{2}+\mathcal{D}^{2}) f \mathcal{D}
-
\mathcal{D} f (1+\lambda^{2}+\mathcal{D}^{2})
}
\frac {1}{ 1+\lambda^{2}+\mathcal{D}^{2} }
\: \dif \lambda 
\\
&=&
\frac{2}{\pi}\int_{0}^{+\infty}
\frac {1}{ 1+\lambda^{2}+\mathcal{D}^{2} }
\brackets{
(1+\lambda^{2})
\parens{ f \mathcal{D} - \mathcal{D} f }
-
\mathcal{D}
\parens{ f \mathcal{D} - \mathcal{D} f }
\mathcal{D}
}
\frac {1}{ 1+\lambda^{2}+\mathcal{D} }
\: \dif \lambda 
\\
&=&
\frac{2}{\pi}\int_{0}^{+\infty}
\frac { \sqrt{ 1+\lambda^{2} } } { 1+\lambda^{2}+\mathcal{D}^{2} }
\parens{ f \mathcal{D} - \mathcal{D} f }
\frac { \sqrt{ 1+\lambda^{2} } } { 1+\lambda^{2}+\mathcal{D}^{2} }
\: \dif \lambda
\\&&-
\frac{2}{\pi}\int_{0}^{+\infty}
\frac { \mathcal{D} } { 1+\lambda^{2}+\mathcal{D}^{2} }
\parens{ f \mathcal{D} - \mathcal{D} f }
\frac { \mathcal{D} } { 1+\lambda^{2}+\mathcal{D}^{2} }
\: \dif \lambda.
\end{eqnarray*}
Note that $\parens{ f \mathcal{D} - \mathcal{D} f }$
is a multiplication operator by a compactly supported function
$\dif f$ due to the Leibniz rule.
Therefore 
$$
\parens{ f \mathcal{D} - \mathcal{D} f }
\frac {1} { 1+\lambda^{2}+\mathcal{D}^{2} }
\qqand
\parens{ f \mathcal{D} - \mathcal{D} f }
\frac { \mathcal{D} } { 1+\lambda^{2}+\mathcal{D}^{2} }
$$
are compact operators by the Rellich lemma.
Moreover, since we have
$$
\norm{
\frac { \sqrt{ 1+\lambda^{2} } } { 1+\lambda^{2}+\mathcal{D}^{2} }
}
\leq \frac{1} {\sqrt{ 1+\lambda^{2} } }
\qqand
\norm{\frac { \mathcal{D} } { 1+\lambda^{2}+\mathcal{D}^{2} }}
\leq \frac{1} {\sqrt{ 1+\lambda^{2} } },
$$
the last two integrals converge in the operator norm topology
and are compact operators.
\end{proof}
\begin{thm}
\label{ThmSymbolIndex}
{\upshape {\cite[Theorem 4.2]{Ka}}} 
$$
[A] = [\sigma]\hattensor_{C_0(T^{*}\!X)}[\mathcal{D}_{T^{*}\!X}]
\in K\!K^{G}(C_0(X), \C).
$$
Recall that $[\sigma] \in K\!K^{G}(C_0(X), C_0(T^{*}\!X))$
is the $K\!K$-theory class given by
$[\sigma]=
\parens{
C_0(\pi^{*}\bb{V}), \: \sigma
}
$ as in definition
\ref{DfnKOfElliptic}.
\QED
\end{thm}

To obtain proposition \ref{PropTwistedDiracOnSigma},
we should rephrase theorem \ref{ThmSymbolIndex}
in terms of $\Sigma X$ instead of $T^{*}\! X$.
Let us regard $T^{*}\!X$ as a subset of $\Sigma X$ by
\begin{eqnarray}
T^{*}\!X
&=&
BX\cup_{SX}(T^{*}\!X \setminus \mathrm{int}(BX))
\nonumber
\\
&\cong&
BX\cup_{SX}(BX \setminus \{ 0\text{-section} \})
=
\Sigma {X} \setminus \{ (x, \infty_x) \}
\hookrightarrow
\Sigma {X}.
\label{RestrictionofSigma}
\end{eqnarray}
Noting that the identification
$T^{*}\!X \setminus \mathrm{int}(BX)
\to
BX \setminus \{ 0\text{-section} \}
;\quad
(x,\xi) \mapsto \parens{x,\frac{\xi}{\norm{\xi}^2} }$
is $G$-equivariant since the action of $G$ does not change the
length of co-tangent vectors,
we will consider
the following $G$-equivariant inclusion map;
\begin{eqnarray*}
 \iota \colon
C_0(T^{*}\!X)
\to
C_0(\Sigma {X}).
\end{eqnarray*}
\begin{lem}
\label{LemSymbolClass}
Consider
$\iota_{*} \colon
K\!K(C_0(X), C_0(T^{*}\!X))
\to
K\!K(C_0(X), C_0(\Sigma {X}))$.
Then,
$
\iota_{*}( [\sigma] ) = [[\sigma]]
$.
\end{lem}

\begin{proof}
Remark that
\begin{eqnarray*}
\bb{V}(\sigma)^{0}
&=&
(\pi_{T}^{*}\bb{V}^{0})_{BX}
\cup_{\sigma_{0}|_{SX}}
(\pi_{T}^{*}\bb{V}^{1})_{BX}
\\
\pi_{\Sigma}^{*}\bb{V}^{1}
&=&
(\pi_{T}^{*}\bb{V}^{1})_{BX}
\cup_{\id|_{SX}}
(\pi_{T}^{*}\bb{V}^{1})_{BX}.
\end{eqnarray*}
Let $ \bb{V}_{T}(\sigma)^{0} $ denote the restriction
of $ \bb{V}(\sigma)^{0} $
on $ T^{*}\! X \subset \Sigma {X} $ using (\ref{RestrictionofSigma}).
Firstly we have
\begin{eqnarray*}
[\sigma]
&=&
\parens{
C_0\parens{
\pi_{T}^{*}\bb{V}
},\
\matrixtwo
{0}{\sigma_{0}^{*}}
{\sigma_{0}}{0}
}
\\
&=&
\parens{
C_0\parens{
\bb{V}_{T}(\sigma)^{0}
\hatsum
\pi_{T}^{*}\bb{V}^{1}
},\
\matrixtwo
{0}{\sigma_{0}^{*}}
{\sigma_{0}}{0}
\cup_{SX}
\matrixtwo{0}{1}{1}{0}
}
\in K\!K^{G} (C_0(X), C_0(T^{*}\!X)).
\end{eqnarray*}
This is due to the unitary isomorphism 
$
\map
{
\matrixtwo
{1}{0}{0}{1}
\cup_{SX}
\matrixtwo
{\sigma_0}{0}{0}{1}
}
{\pi_{T}^{*}\bb{V}}
{\bb{V}_{T}(\sigma)^{0} \hatsum \pi_{T}^{*}\bb{V}^{1}}
$.

Then, $\iota_{*} \colon
K\!K(C_0(X), C_0(T^{*}\!X))
\to
K\!K(C_0(X), C_0(\Sigma {X}))$
maps $[\sigma]$ to the class
\begin{eqnarray*}
\parens{
Z \>,\;
\matrixtwo
{0}{\sigma_{0}^{*}}
{\sigma_{0}}{0}
\cup_{SX}
\matrixtwo
{0}{1}
{1}{0}
}
\in K\!K(C_0(X), C_0(\Sigma {X})),
\end{eqnarray*}
where
$$
Z:=
\braces{ s \in C_0\parens{
\bb{V}(\sigma)^{0}
\hatsum
\pi_{\Sigma}^{*}\bb{V}^{1}
}
\> \Big| \;
s(x, \infty)=0
\text{
for any $x\in X$.
}}.
$$
This $K$-theory class is equal to the class
$$
\parens{
C_0\parens{
\bb{V}(\sigma)^{0}
\hatsum
\pi_{\Sigma}^{*}\bb{V}^{1}
},
\matrixtwo
{0}{\sigma_{0}^{*}}
{\sigma_{0}}{0}
\cup_{SX}
\matrixtwo
{0}{1}
{1}{0}
}
\in K\!K(C_0(X), C_0(\Sigma {X}))
$$
by the homotopy given by the following $IC_0(\Sigma X)$-module;
$$
\braces{
f \in IC_0\parens{
\bb{V}(\sigma)^{0}
\hatsum
\pi_{\Sigma}^{*}\bb{V}^{1}
}
\> \Big| \;
f (1) (x,\infty)=0
\text{
for any $x\in X$.
}}.
$$
Finally, it is equal to
$$
[[\sigma]] = 
\parens{C_0(\bb{V}(\sigma)^{0}
\hatsum
\pi_{\Sigma}^{*} (\bb{V}^{1})) ,\ 0}
\in 
K\!K^{G} (
C_0(X),\ C_0(\Sigma {X})).
$$
by the operator homotopy
$
\matrixtwo
{0}{t\sigma_{0}^{*}}
{t\sigma_{0}}{0}
\cup_{SX}
\matrixtwo
{0}{t}
{t}{0}
$.
\end{proof}
\begin{lem}
\label{LemDolbeaultClass}
Consider
$\iota ^{*}
\colon
K\!K^{G}(C_0(\Sigma {X}) , \C)
\to
K\!K^{G}(C_0(T^{*}\! X) , \C)
$.
Then,
$\iota ^{*}([\mathcal{D}_{\Sigma {X}}] )
=[\mathcal{D}_{T^{*}\!X}]$.
\end{lem}
\begin{proof}
$
\iota ^{*}([\mathcal{D}_{\Sigma {X}}] )
$
is represented by the class
$$
\parens{L^{2}\parens{\turnV{\!0,*}T^{*}(\Sigma {X})}
,\
\frac{\mathcal{D}_{\Sigma {X}}}
{\sqrt{1+\mathcal{D}_{\Sigma {X}}^2}}
}
\in K\!K^{G}(C_0(T^{*}\! X) , \C)
$$
with the point-wise multiplication action by $C_0(T^{*}\! X)$
via $\iota$.

For a while, let $\iota$ denote also the inclusion of the spaces
$\iota \colon T^{*}\! {X} \hookrightarrow \Sigma {X}$.
For each $k=0,1,\ldots,$ we can find a smooth function
$\map {f_{k}} {T^{*}\! {X}} {\R_{>0}}$
so that the following maps are isometric on each fiber; 
$$
\begin{array}{ccc}
\turnV{\!0,k}T^{*}(\Sigma {X})
&\to&
\turnV{\!0,k}T^{*}\!(T^{*}\! {X})
\\
\omega
&\mapsto&
f_{k}{\cdot}\iota^{*}\omega.
\end{array}
$$
Assembling these maps,
we obtain the fiber-wise isometry
$$
\widehat{\varphi} := \bigoplus_{k} f_{k}{\cdot}\iota^{*}
\colon
\turnV{\!0,*}T^{*}(\Sigma {X})
\to
\turnV{\!0,*}T^{*}\!(T^{*}\! {X}).
$$
Next, let $h$ denote the Radon-Nikodym derivative
$\frac{\mathrm{vol}_{\Sigma X}}{\mathrm{vol}_{T^{*}\! X}}$
so that $\iota$ induces the isomorphism between measure spaces;
$$
\parens{ \Sigma X ,\ \mathrm{vol}_{\Sigma X} }
\cong
\parens{ T^{*}\! X ,\ h {\cdot} \mathrm{vol}_{T^{*}\! X} }.
$$
Using $\widehat{\varphi}$ and $h$,
$L^{2}\parens{\turnV{\!0,*}T^{*}(\Sigma {X})}$
can be identified with 
$L^{2}\parens{\turnV{\!0,*}T^{*}\!(T^{*}\! {X})}$
by
$$
\begin{array}{ccc}
L^{2}\parens{\turnV{\!0,*}T^{*}(\Sigma {X})}
&\to&
L^{2}\parens{\turnV{\!0,*}T^{*}\!(T^{*}\! {X})}
\\
s
&\mapsto&
\sqrt{h} {\cdot} \widehat{\varphi}(s).
\end{array}
$$
This isomorphism will be denoted by 
$\varphi := \sqrt{h} {\cdot} \widehat{\varphi}$.
Under this unitary equivalence, we have
$$
\varphi \circ D_{T^{*}\! X} \circ \varphi^{-1}
=D_{\Sigma {X}}.
$$
Besides, since the functional calculus commutes with the
unitary conjugation,
we obtain
$$
\varphi \circ
\frac{\mathcal{D}_{T^{*}\! X}}
{\sqrt{1+\mathcal{D}_{T^{*}\! X}^2}}
\circ \varphi^{-1}
=
\frac{\mathcal{D}_{\Sigma {X}}}
{\sqrt{1+\mathcal{D}_{\Sigma {X}}^2}}.
$$
Note that this unitary equivalent map
is compatible with the action of
$C_0 (T^{*}\! X)$.
Therefore, as elements in $K\!K(C_0(T^{*}\! X) , \C)$,
\begin{eqnarray*}
\iota^{*} ([\mathcal{D}_{\Sigma {X}}])
&=&
\parens{L^{2}\parens{\turnV{\!0,*}T^{*}(\Sigma {X})}
,\
\frac{\mathcal{D}_{\Sigma {X}}}
{\sqrt{1+\mathcal{D}_{\Sigma {X}}^2}}
}
\\
&=&
\parens{L^{2}\parens{\turnV{\!0,*}T^{*}(\Sigma {X})}
,\
\varphi \circ
\frac{\mathcal{D}_{T^{*}\! X}}
{\sqrt{1+\mathcal{D}_{T^{*}\! X}^2}}
\circ \varphi^{-1}
}
\\
&=&
\parens{\varphi L^{2}\parens{\turnV{\!0,*}T^{*}(\Sigma {X})}
,\
\frac{\mathcal{D}_{T^{*}\! X}}
{\sqrt{1+\mathcal{D}_{T^{*}\! X}^2}}
}
\\
&=&
\parens{L^{2}\parens{\turnV{\!0,*}T^{*}(T^{*}\! X)}
,\
\frac{\mathcal{D}_{T^{*}\! X}}
{\sqrt{1+\mathcal{D}_{T^{*}\! X}^2}}
}
\\
&=&
[\mathcal{D}_{T^{*}\! X}]
\end{eqnarray*}
\end{proof}

\begin{proof}[Proof of Proposition \ref{PropTwistedDiracOnSigma}.]
Due to theorem \ref{ThmSymbolIndex},
lemma \ref{LemSymbolClass} and
lemma \ref{LemDolbeaultClass}, we obtain
\begin{eqnarray*}
[A] &=& [\sigma]\hattensor_{C_0(TX)}[\mathcal{D}_{T^{*}\!X}]
\\
&=&
 [\sigma]\hattensor_{C_0(TX)} \iota^{*}([\mathcal{D}_{\Sigma {X}}])
\\
&=&
\iota_{*}( [\sigma] )\hattensor_{C_0(\Sigma{ X })}
[\mathcal{D}_{\Sigma {X}}]
\\
&=&
[[\sigma]]\:\hattensor_{C_0(\Sigma{ X })}
[\mathcal{D}_{\Sigma {X}}].
\end{eqnarray*}
It implies the second part as
\begin{eqnarray*}
[A]=[[\sigma]]\:\hattensor_{C_0(\Sigma {X})}
[\mathcal{D}_{\Sigma {X}}]
&=&
(\pi_{\Sigma}^{\sharp})^{*} \parens{[\bb{V}(\sigma)^{0}
\hatsum
\pi_{\Sigma}^{*}\bb{V}^{1}]}
\:\hattensor_{C_0(\Sigma {X})}
[\mathcal{D}_{\Sigma {X}}]
\\
&=&
(\pi_{\Sigma}^{\sharp})^{*} \parens{[\bb{V}(\sigma)^{0}
\hatsum
\pi_{\Sigma}^{*}\bb{V}^{1}]
\:\hattensor_{C_0(\Sigma {X})}
[\mathcal{D}_{\Sigma {X}}]}
\\
&=&
(\pi_{\Sigma}^{\sharp})^{*}([\mathcal{D}_{\sigma}]).
\end{eqnarray*}
\end{proof}
\begin{cor}
Let $L$ be a Hermitian $G$-line bundle over $X$.
Consider the induced bundle $\pi_{\Sigma}^{*}L$
over $\Sigma X$,
which determines the $K\!K$-theory element
$[\pi_{\Sigma}^{*}L]
=
\parens{C_0(\pi_{\Sigma}^{*}L) , 0}
\in 
K\!K^{G}(C_0(\Sigma X), C_0(\Sigma X))$. 
Then,
$$
\mu_{G}([L]\hattensor [A]) 
=
\mu_{G}\parens{
 [\pi_{\Sigma}^{*}L] \hattensor [\mathcal{D}_{\sigma}]
}.
$$
\end{cor}
\begin{proof}
As we will not consider the projection from $T^{*}\! X$,
$ \pi_{\Sigma} $ will be abbreviated to $\pi$.
The map between the algebras of continuous functions
induced by $\pi$ will be denoted by
$\map{ \pi^{\sharp} }{ C_0(X) }{ C_0(\Sigma X) }$
and it will also regarded as a $K\!K$-theory class
$\pi^{\sharp} \in K\!K^{G}(C_0(X) , C_0(\Sigma X))$.
Note that $j^{G}  (\pi^{\sharp}) \in 
 K\!K(C^{*}(G;C_0(X)) , C^{*}(G;C_0(\Sigma X)))$
represents a natural map $C^{*}(G;C_0(X)) \to C^{*}(G;C_0(\Sigma X))$
given by
$
f \mapsto \bar{f}(\gm)(x,\xi) := f(\gm)(\pi (x,\xi) ) = f(\gm)(x)
$.

Let $c\in C_c(X)$ be an arbitrary cut-off function.
Then $\pi^{\sharp} c = c\circ \pi \in C_c(\Sigma X)$
is a cut-off function for the action of $G$ on $\Sigma X$.
They determine a $K$-homology classes
$[c]\in K\!K(\C, C^{*}(G;C_0(X)))$ and
$[c\circ \pi]\in K\!K(\C, C^{*}(G;C_0(\Sigma X)))$
as in lemma \ref{CutoffKHomologyClass},
and they are related by the formula
$[c] \hattensor j^{G}  (\pi^{\sharp}) = [c\circ \pi] $.

Using the proposition and noting that 
$[\pi^{*}L] $ is an element in
$K\!K^{G}(C_0(\Sigma X), C_0(\Sigma X))$
and
$ (\pi^{\sharp})^{*}[\pi^{*}L]
 =  (\pi^{\sharp})_{*} [L]
\in K\!K^{G}(C_0(X), C_0(\Sigma X)) $, we obtain
\begin{eqnarray*}
\mu_{G}([L]\hattensor [A]) 
&=&
\mu_{G}
\parens{
[L]\hattensor (\pi^{\sharp})^{*} [\mathcal{D}_{\sigma}]
}
\\
&=&
[c] \hattensor j^{G}\parens{
 (\pi^{\sharp})_{*} [L]\hattensor [\mathcal{D}_{\sigma}]
}
\\
&=&
[c] \hattensor j^{G}\parens{
  (\pi^{\sharp})^{*}[\pi^{*}L] \hattensor [\mathcal{D}_{\sigma}]
}
\\
&=&
[c] \hattensor j^{G}  (\pi^{\sharp}) \hattensor j^{G} \parens{
 [\pi^{*}L] \hattensor [\mathcal{D}_{\sigma}]
}
\\
&=&
[c\circ \pi] \hattensor j^{G}
\parens{
 [\pi^{*}L] \hattensor [\mathcal{D}_{\sigma}]
}
\\
&=&
\mu_{G}\parens{
 [\pi^{*}L] \hattensor [\mathcal{D}_{\sigma}]
}.
\end{eqnarray*}
\end{proof}
\subsection{Expression of $\mu_{G}([A])$}
\label{SubsectionExpressionOfMu}
To see more specific expression of the image of $\mu$,
let us first introduce a Hilbert $C^{*}(G)$-module $\E^G$
by taking the completion of $C_c(\mathbb{V})$.
which will be used to represent $\mu_{G}([A])\in K\!K(\C, C^{*}(G))$.

First we define on $C_c(\bb{V})$ the structure of a pre-Hilbert module
over $C_c(G)$ using the action of ${G}$ on $C_c(\bb{V})$
given by $g[s](x) = g(s(g^{-1} x))$ for $g\in {G}$.
\begin{itemize}
\item
The action of $C_c({G})$ on $C_c(\bb{V})$ from the right
is given by
\begin{eqnarray}
\label{ActionOnCalE}
[s\cdot b] = \int_{{G}} {g}[s] \cdot b({g}^{-1}) \dif {g}
\in C_c(\bb{V})
\end{eqnarray}
for $s\in C_c(\bb{V})$ and $b\in C_c({G})$.
Let $\rho^{{G}}$ denote this right action.
\item
The scalar product valued in $C_c({G})$ is given by
\begin{eqnarray}
\label{ScalarProdOnCalE}
\angles{s_1,s_2}_{\E^{{G}}}({g}) = \angles{s_1, {g}[s_2]}_{L^{2}(\bb{V})}
\end{eqnarray}
for $s_i\in C_c(\bb{V})$.
\end{itemize}
\begin{dfn}
Define $\E^{{G}}$ as the completion of $C_c(\bb{V})$
under the norm $\sqrt{\norm{\angles{s,s}}_{C^{*}({G})}}$.
\end{dfn}

In order to realize $\E^{{G}}$ as
a submodule of $C^{*}({G},C_0(X))$
for convenience, we define the map
$\map{i^{{G}}}{C_c(\bb{V})}{C_c({G} ;L^{2}(\bb{V}))}$ and 
$\map{q^{{G}}}{C_c({G};L^{2}(\bb{V}))}{C_c(\bb{V})}$
by the formula;
\begin{eqnarray*}
i^{{G}}[s]({g})(x) &=& \sqrt{c(x)} {g}[s](x)
\\
q^{{G}}[\widehat{s}](x) &=& \int_{{G}}\sqrt{c({g}_1^{-1} x)}
\widehat{s}({g}_1^{-1})({g}_1^{-1} x) \dif {g}_1
\end{eqnarray*}
for $s\in C_c(\bb{V})$ and $\widehat{s} \in C_c({G};L^{2}(\bb{V}))$,
where $c\in C_c(X)$ denotes a cut-off function.

$i^{{G}}$ preserves the right ${G}$-action and scalar product,
here $C_c({G};L^{2}(\bb{V}))$ is equipped with the following
$C^{*}({G})$-valued inner product
$\angles{\widehat{s}_1, \widehat{s}_2}_{C^{*}({G};L^{2}(\bb{V}))}({g})
=
\int_{{g}_1 {g}_2 = {g}}
\angles{\widehat{s}_1({g}_1), \widehat{s}_2({g}_2)}_{L^{2}(\bb{V})}
\dif {g}_1$.

$q^{{G}}$ is the adjoint of $i^{{G}}$ that is,
$
\angles{i^{{G}}(s), \widehat{s}}_{C^{*}({G}; L^{2}(\bb{V}))}
 = 
\angles{s, q^{{G}}(\widehat{s})}_{\E^{{G}}}
$.
\begin{thm}
{\upshape{\cite[Theorem 5.6]{Ka}}}
$\mu_{{G}} ([A]) \in K\!K(\C, C^{*}({G}))$
is represented by $\parens{\E^{{G}}, \overline{A^{{G}}}}$,
where $\map{\overline{A^{{G}}} = q^{{G}} A^{G} i^{{G}}}
{\E^{{G}}}{\E^{{G}}}$
and $\map{A^{{G}}} {C^{*}({G} ;L^{2}(\bb{V}))}
 {C^{*}({G} ;L^{2}(\bb{V}))}$
is given by
$$A^{{G}}[\widehat{s}]({\gm}) = A[\widehat{s}({\gm})]$$
for $\widehat{s} \in C_c(G;L^{2}(\bb{V}))$
as in definition \ref{DfnjG1}.
This operator is equal to
$$\overline{A^{{G}}}=\int_{{G}}\gm[\sqrt{c}A\sqrt{c}]\:\dif \gm. $$
$\overline{A^{{G}}}$ defines
a bounded operator both on $L^{2}(\bb{V})$ and $\E^{{G}}$ 
{\upshape \cite[Proposition 5.1]{Ka}}.
\end{thm}
\subsection{Kronecker Pairing and its Formula}
Now we can define the Kronecker pairing
$\angles{[E], [A]}_{G} \in \R$
for $[A]\in K\!K^{G}(C_0(X), \R)$ and 
$[E]=\parens{C_0(E),0}\in K\!K^{G}(C_0(X), C_0(X))$
in Theorem \ref{MainThm}.
\begin{dfn}\label{DfnKronecker}
$$
\angles{[E], [A]}_{G} \in \C
:= 
\tr_{G}\circ \mu_{G}([E]\hattensor [A])
\in \R,
$$
where $\tr_{G}$ is the canonical trace on $C^{*}(G)$.
\end{dfn}
\begin{rem}
We will see that this definition is compatible with the case
of \cite{Ha-Sc08}, that is, the case when
 a discrete group $\Gamma$ acts on $X$
freely
and $X/\Gamma$ is compact.
In such case, one has
$$
K^{0}_{\Gamma}(C_0(X)) \simeq
K^{0}(C^{*}(\Gamma; C_0(X))) \simeq
K^{0}(C(X/\Gamma))
$$
and $C(X/\Gamma)$ is unital, one also has a map
$$
K^{0}(C(X/\Gamma)) \to K^{0}(\C) \simeq \Z
$$
induced by the inclusion of $\C$ into $C(X/\Gamma)$
as constant functions.
Let $A$ be a $\Gamma$-invariant Dirac type operator
on X, and $E$ be a $\Gamma$-Hermitian vector bundle over $X$.
Then they defines a Dirac type operator on $X/\Gamma$
denoted by $A'$ and a Hermitian vector bundle
$E'$ over $X/\Gamma$. Let $A'_{E'}$ denote the Dirac type operator
twisted by $E'$.
The composition $K^{0}_{\Gamma}(C_0(X)) \to \Z$
maps $[E]\hattensor[A]$ to the ordinary index
of Dirac type operator $A'_{E'}$ on $X/\Gamma$,
which will be denoted by $\ind(A'_{E'})$.
And it is equal to the usual $K$-theoretical pairing
$\angles{[E'],[A']}$ for $[E'] \in K^{0}(X/\Gamma)$
and $[A'] \in K_{0}(X/\Gamma)$
Therefore, it is equal to the pairing between ordinary
cohomology and homology
$\angles{\mathrm{ch}(E'), \mathrm{ch}(A')}$,
which appears in \cite{Ha-Sc08}.

On the other hand our definition of the Kronecker pairing
is defined by the composition of $\mu_{G}$ and $\tr_{G}$
$$
\tr_{G}\circ \mu_{G}\colon
K^{0}_{\Gamma}(C_0(X)) \to K_{0}(C^{*}(\Gamma)) \to \R.
$$
Due to Atiyah's $L^2$-index theorem \cite{At76},
they coincide with each other, namely,
$\angles{[E],[A]}_{\Gamma} = \ind(A_{E})$.
\end{rem}

We need a formula of Kronecker pairing
in the term of the integration of cohomology classes.
\begin{thm}\label{HWindex}
{\upshape{\cite[Proposition 6.11.]{Wa14}}}
Let $X$ be a complete Riemannian manifold, where 
a second countable locally compact Hausdorff
then the $L^2$-index
of $A$ is given by the formula
\begin{eqnarray}
\label{TraceMu}
\tr_{G}\ind_{G}(A) = \int_{TX} c(x) \mathrm{Td}(T_{\C}X)
\wedge \mathrm{ch}(\sigma_A),
\end{eqnarray}
where $\tr_{G}$ is the canonical trace on $C^{*}(G)$,
$\sigma_A$ denotes the symbol of $A$, and
unimodular group $G$ acts
properly, co-compactly and isometrically.
If $A$ is  a properly supported
$G$-invariant elliptic operator of order $0$,
$c\in C_0(X)$ is any cut-off function.
\end{thm}

Also we obtain a formula of the
Kronecker pairing $\angles{[E], [A]}_{G}$
using the cohomology.
Choose an arbitrary $G$-invariant connection $\nabla$ on $E$
and apply the theorem to the Dirac type operator $A_{E}$
twisted by $(E,\nabla)$ to obtain
\begin{eqnarray*}
\angles{[E], [A]}_{G}
&=&
\tr_{G} \mu_{G} ([E]\hattensor [A])
\quad =\quad 
\tr_{G} \ind_{G} (A_{E})
\\
&=&
\int_{TX} c(x) \mathrm{Td}(T_{\C}X)
\wedge \mathrm{ch}(\sigma_{A})\wedge \mathrm{ch}(E,\nabla).
\end{eqnarray*}
This value is independent of the choice of the $G$-invariant
connection on $E$ according to
the lemma \ref{GInvariantChernWeil}.

Using this formula, we can define 
higher $G$-$\widehat{\mathcal{A}}$-genus in
corollary \ref{MainCor}.
\begin{dfn}\label{HigherAgenus}
Let $E$ be a smooth Hermitian $G$-vector bundle
over $X$. Then we define
$$
\widehat{\mathcal{A}}_{G}(X;E)
=
\int_{X} c(x) \widehat{\mathcal{A}}(TX)\wedge \mathrm{ch}(E)
\in \R.
$$
This is the spacial case of (\ref{TraceMu}) when
$A$ is the Dirac operator on the canonical spinor bundle over $X$
twisted by $E$.
\end{dfn}
In the case of a free action of a discrete group $\Gamma$ on $X$,
by remark \ref{IntegralOnFundamentalDomain},
$\widehat{\mathcal{A}}_{G}(X;E)$ is equal to 
$\int_{X/\Gamma} \widehat{\mathcal{A}}(T(X/\Gamma))\wedge \mathrm{ch}(E)$,
namely,
the higher $\widehat{\mathcal{A}}$-genus on $X/\Gamma$,
so this definition \ref{HigherAgenus} is one of generalizations of 
the higher $\widehat{\mathcal{A}}$-genus to the case of
proper actions of non-discrete groups.

\section{Construction of Almost Flat Bundles and their Indices}
Observing that for $E$ of the form
$E=L^{1}\oplus \cdots \oplus L^{N}$,
where $L^{k}$ are $G$-line bundles,
$\mu_{G}([E]\hattensor [A]) =
\sum_{k} \mu_{G}([L^{k}]\hattensor [A])$
and $\mathrm{ch}(E)= \sum_{k} \mathrm{ch}(L^{k})$,
it is sufficient to consider the case that
$E$ is a single smooth $G$-line bundle $L$.
Our strategy to prove the Theorem \ref{MainThm}
is that firstly, construct a family of 
almost flat bundles $L_{t}$ over $X$ for $t\in [0,1]$
using $L$, and calculate the index of twisted Dirac operators 
$A_{L_t}$.
In order to do this,
we introduce a $U(1)$-cocycle $\alpha \in C^{2}(G;U(1))$
coming from the $G$-line bundle $L$
following \cite{Ma99} and \cite{Ma03}.
However we will consider the central extension using $\alpha$
while V. Mathai employed the $\alpha$-twisted actions and
$\alpha$-twisted group $C^{*}$-algebras.
\subsection{Central Extension}
We would like to construct line bundles $L_{t}$
over $X$ for $t\in [0,1]$
whose curvature have the norm decreasing to zero.
However $L_{t}$ are not $G$-line bundles
despite $L$ is $G$-line bundle.
We will introduce $U(1)$-cocycle $\alpha\in C^{2}(G;U(1))$
coming from
$C^{2}(G;\R) \xrightarrow{\exp (i\theta)}C^{2}(G;U(1))$
and central $U(1)$-extensions $G_{\alpha^{t}}$,
and construct each of $L_{t}$ as a $G_{\alpha^{t}}$-line bundle 
over $X$.
\begin{dfn}
\label{Dfnetapsi}
Chose an arbitrary $G$-invariant connection $\nabla$ on $L$
and let $i\omega \in \Omega(X;i\R)$
denote the curvature of $(L,\nabla)$.
Then its first Chern class is equal to $c_1(L;\nabla) =
\frac{-1}{2\pi}\omega 
\in \Omega(X;\R)$.
Fix a base point $x_0 \in X$.
For any $\gm\in G$,
we will construct a $\R$-valued smooth function
$\psi_{\gm}\in C^{\infty}(X)$ which will be 
used to define a $U(1)$-cocycle $\alpha\in C^{2}(G;U(1))$.
\begin{itemize}
\item
Since $[c_1(L)] = 0 \in H^{2}(X;\R)$,
in particular $[\omega] = 0 \in H^{2}_{\text{dR}}(X;\R)$,
there exists $\eta \in \Omega^{1}(X;\R)$
such that $\dif \eta = \omega$.
\item
Since $\omega$ is $G$-invariant,
we have $\dif (\gm^{*}\eta - \eta)
= \gm^{*}\omega - \omega =0$.
The assumption of $H_1(X;\Z) = 0$ implies that
there exists $\psi_{\gm}\in C^{\infty}(X;\R)$ satisfying 
$\dif \psi_{\gm} = \gm^{*}\eta - \eta$ for each $\gm \in G$.
We may assume for every $\gm \in G$, $\psi_{\gm}(x_0) = 0 $
for a fixed base point $x_0 \in X$.
\end{itemize}
\end{dfn}
\begin{lem}
$\psi_{e} (x) =0$ for every $x\in X$,
where $e$ denotes the identity element in $G$.

Moreover for any $\gm_1$ and $\gm_2 \in G$,
$$
\psi_{\gamma_1}(\gamma_2 x)
+\psi_{\gamma_2}(x)-\psi_{\gamma_1\gamma_2}(x)
$$
is independent of $x\in X$
\end{lem}
\begin{proof}
Since $\dif \psi_{e} = 0$, it follows that
 $\psi_{e} (x) = \psi_{e} (x_0) =0$  for every $x\in X$.
In addition
\begin{eqnarray*}
\dif \parens{
\psi_{\gamma_1}(\gamma_2 x)
+\psi_{\gamma_2}(x)-\psi_{\gamma_1\gamma_2}(x)
}
&=&
\brackets{\gm_1 ^{*} \eta - \eta}(\gm_2 x) 
+ \brackets{\gm_2 ^{*} \eta - \eta}(x)
- \brackets{(\gm_1 \gm_2)^{*}\eta - \eta }(x)
\\
&=&
\brackets{\gm_2 ^{*} \gm_1 ^{*} \eta - \gm_2 ^{*}\eta }(x) 
+ \brackets{\gm_2 ^{*} \eta - \eta}(x)
- \brackets{\gm_2 ^{*} \gm_1 ^{*} \eta - \eta }(x)
\\
&=&
0.
\end{eqnarray*}
Then the second statement follows.
\end{proof}
\begin{lem}\label{DfnU(1)Cocycle}
Set
\begin{eqnarray*}
\sigma(\gamma_1,\gamma_2)
&:=&
\psi_{\gamma_1}(\gamma_2 x)
+\psi_{\gamma_2}(x)-\psi_{\gamma_1\gamma_2}(x)
\\
&=&
\psi_{\gamma_1}(\gamma_2 x_0).
\end{eqnarray*}
Then $\sigma$ defines a $\R$-valued $2$-cocycle $\sigma \in C^2(G;\R)$.
In particular $\alpha:= \exp[i\sigma]$
is a $U(1)$-valued $2$-cocycle $\alpha \in C^2(G;U(1))$.
\end{lem}
\begin{proof}
We have to check
\begin{eqnarray*}
\sigma(\gm, e) = \sigma (e, \gm) &=& 0
\qqand \\
\sigma(\gm_1 ,\gm_2)+\sigma(\gm_1 \gm_2, \gm_3)
&=&
\sigma(\gm_1 ,\gm_2 \gm_3)+\sigma(\gm_2, \gm_3).
\end{eqnarray*}
The first equation follows directly from $\psi_e(x) = 0$ and
$\psi_{\gm}(x_0)=0$ for every $x\in X$ and $\gm \in G$.
The second equation follows from direct calculation.
Due to the previous lemma,
\begin{eqnarray*}
 \psi_{\gm_1}(\gm_2 \gm_3 x_0)
+
\psi_{\gm_2}(\gm_3 x_0)
-
\psi_{\gm_1 \gm_2}(\gm_3 x_0)
&=&
\psi_{\gm_1}(\gm_2 x_0)
+
\psi_{\gm_2}(x_0)
-
\psi_{\gm_1 \gm_2}(x_0)
\\
&=&
\psi_{\gm_1}(\gm_2 x_0).
\end{eqnarray*}
Then we obtain
\begin{eqnarray*}
\sigma(\gm_1 ,\gm_2)+\sigma(\gm_1 \gm_2, \gm_3)
&=&
\psi_{\gm_1}(\gm_2 x_0) + \psi_{\gm_1 \gm_2}(\gm_3 x_0)
\\
&=&
\psi_{\gm_1}(\gm_2 x_0)
+
\braces{
\psi_{\gm_1}(\gm_2 \gm_3 x_0)
+
\psi_{\gm_2}(\gm_3 x_0)
-
\psi_{\gm_1}(\gm_2 x_0)
}
\\
&=&
\psi_{\gm_1}(\gm_2 \gm_3 x_0) + \psi_{\gm_2}(\gm_3 x_0)
\\
&=&
\sigma(\gm_1, \gm_2 \gm_3) + \sigma(\gm_2, \gm_3).
\end{eqnarray*}
\end{proof}

\begin{lem}
If we choose another $G$-invariant
connection $\nabla'$ on $L$,
then
the resulting $\alpha' \in C^2(G;U(1))$
is cohomologous to $\alpha$.
In particular
at the $U(1)$-cohomology level,
$\alpha \in C^2(G;U(1))$ is independent of the choice of 
$G$-invariant connection $\nabla$ and $\eta$
used in the definition
of $\alpha$.
\end{lem}
\begin{proof}
If we let $i\omega' \in \Omega^{2}(X;i\R)$ denote the curvature
of $(L,\nabla')$.  
Then its first Chern class $c_1(L,\nabla')$ is equal to
$\frac{-1}{2\pi} \omega' \in \Omega^{2}(X;\R)$.
According to the lemma \ref{GInvariantChernWeil},
there exists a $G$-invariant $1$-form
$\theta \in \Omega^{1}(X)$
satisfying $\omega' - \omega =\dif \theta$.
As in definition \ref{Dfnetapsi},
since $[\omega'] \in H^{2}_{\mathrm{dR}}(X;\R)$,
there exists $\eta' \in \Omega^{1}(X)$
such that $\dif \eta' = \omega'$.
Since $\dif \eta' - \dif \eta = \omega' - \omega = \dif \theta$,
there exists $h \in C^{\infty}(X)$ which satisfies
$\dif h= \eta' - \eta - \theta$.
We may assume $h(x_0)=0$ for the fixed base point $x_0\in X$.
Similarly, define $\psi' \in C^{\infty}(X)$ for each $\gm \in G$
by $\dif \psi'_{\gm} = \gm^{*} \eta' -\eta'$ and $\psi'_{\gm}(x_0)=0$.
Now since $\gm ^{*} \theta = \theta$, we obtain
\begin{eqnarray*}
\dif(\psi'_{\gm} - \psi_{\gm})
&=&
\gm^{*}(\eta' - \eta) - (\eta' - \eta)
\\
&=&
\gm^{*}(\eta' - \eta - \theta) - (\eta' - \eta -\theta)
\\
&=&
\dif(\gm^{*}h-h)
\\
\text{and}\qquad
\psi'_{\gm}(x_0) = \psi_{\gm}(x_0) = h(x_0) &=& 0 \text{ ,}
\\
\text{which imply}\qquad \qquad \qquad \qquad
\psi'_{\gm} - \psi_{\gm}  &=& \gm^{*}h - h - \gm^{*}h(x_0)
\text{ .}
\end{eqnarray*}
Therefore the $U(1)$-cocycle defined using $\omega'$
is
\begin{eqnarray*}
\alpha'(\gamma_1,\gamma_2)
&=&
\exp \brackets{ 
i 
\psi'_{\gamma_1}(\gamma_2 x_0)}
\\
&=&
\exp \brackets{ 
i 
\psi_{\gamma_1}(\gamma_2 x_0)}
\exp \brackets{
i
\parens{\gm_1^{*}h(\gm_2 x_0) - h(\gm_2 x_0) - \gm_1^{*}h(x_0)}
}
\\
&=&
\alpha(\gm_1, \gm_2)
\delta \beta(\gm_1, \gm_2)^{-1},
\end{eqnarray*}
where $\beta \in C^1(G;U(1))$ is
a $U(1)$-cocycle defined by
$\beta(\gm) = \exp \brackets{ih(\gm x_0)}$
and recall that
the derivative $\map{\delta}{C^1(G;U(1))}{C^2(G;U(1))}$
is given by the formula
$$
(\delta \beta) (\gm_1, \gm_2)
=
\beta(\gm_1)\beta(\gm_2)\beta(\gm_1 \gm_2)^{-1}.
$$
To conclude, $\alpha' \in C^2(G;U(1))$ is cohomologous to $\alpha$. 
 \end{proof}
\begin{dfn}
For each $t\in [0,1]$,
take a $U(1)$-cocycle $\alpha^{t} = \exp[it\sigma]$.
Let $G_{\alpha^t}$ be
the central $U(1)$-extension of $G$ determined by
the 2-cocycle $\alpha^t \in C^2(G;U(1))$.
In other words
$G_{\alpha^t}$ is $G\times U(1)$
as a topological measure space and 
equipped with the following product; 
$$
(\gamma_1, u_1)(\gamma_2, u_2) =
(\gamma_1 \gamma_2\>,\; \alpha (\gamma_1, \gamma_2)^t u_1u_2).
$$
Suppose that the Haar measure of $U(1)$ is normalized
so that
$\int_{U(1)}\dif u =1$.
\end{dfn}

Since there is a one to one correspondence 
between
the isomorphism classes of central $U(1)$-extension of $G$
and 
$H^{2}(G;U(1))$ when $G$ acts on $U(1)$ trivially,
$G_{\alpha^t}$ is determined uniquely up to isomorphism.
\subsection{Almost Flat $\Galp$-Line Bundles}
\begin{dfn}
For each $t\in [0,1]$ we will construct a
$G_{\alpha^t}$-line bundle $L_t$ over $X$
whose curvature is equal to $it\omega$.
\begin{itemize}
\item
$L_t$ is topologically equal to a trivial bundle $X \times \C$.
\item
The connection of $L_t$ is given by $\nabla^{t} = \dif +it\eta$.
\end{itemize}
\end{dfn}
Since $L_t$ is a line bundle, its curvature is given by
$\dif (it\eta) + (it\eta) \wedge (it\eta) = \dif (it\eta)
= it\omega$.
\begin{lem}\text{}
\begin{enumerate}
\item
$G_{\alpha^t}$ acts on $\L_t$ given by the formula;
$$(\gamma, u)(x,z)=(\gamma x\>,\; \exp[it\psi_{\gamma}(x)]uz)
\qqfor
(\gamma, u) \in G_{\alpha^t}, \;
x \in X,\;
z \in \C=(L_{t})_{x}.
$$
\item
$\nabla^{t}$ is a $G_{\alpha^t}$-invariant connection, that is,
the action of $G_{\alpha^t}$ and
the connection $\nabla^{t}$ on $L_{t}$ commute
with each other.
In particular
let $A_{L_t}$ denote the twisted Dirac operator acting on
the sections of $\mathbb{V} \hattensor L_t$.
Then $A_{L_t}$ is $G_{\alpha^t}$-equivariant.
\end{enumerate}
\end{lem}
\begin{proof}
This actually gives the action of $G_{\alpha^t}$.
In fact,
\begin{eqnarray*}
(\gm_1, u_1)\braces{(\gm_2, u_2)(x,z)}
&=&
\parens{
\gm_1 \gm_2 x,
\exp \brackets{it\psi_{\gm_1}(\gm_2 x)}
\exp \brackets{it\psi_{\gm_2}(x)}
u_1 u_2 z
}
\\
&=&
\parens{
\gm_1 \gm_2 x,
\exp \brackets{it\psi_{\gm_1 \gm_2}(x)}
\alpha(\gm_1, \gm_2)^{t}
u_1 u_2 z
}
\\
&=&
\braces{(\gm_1, u_1)(\gm_2, u_2)}(x,z).
\end{eqnarray*}
Since $(\gamma, u )^{-1} =
\parens{\gamma^{-1}, \alpha(\gamma, \gamma^{-1})^{-t}u^{-1}}$
in $G_{\alpha^{t}} $,
remark that
the action of $G_{\alpha ^{t}}$ on $C^{\infty}(L_{t})$ is 
given by the formula;
\begin{eqnarray*}
(\gamma, u)[s](x)
&=&
u^{-1}\alpha(\gamma, \gamma^{-1})^{-t}
\exp\brackets{it\psi_{\gamma^{-1}}}\gamma[s](x)
\\
&=&
u^{-1}\exp\brackets{-it\psi_{\gamma}(\gamma^{-1}x)}\gamma[s](x).
\end{eqnarray*}
As for the second statement,
we will proof in the case of $t=1$.
The claim for general $t\in[0,1]$ follows
if we replace $\eta$ and $\psi$ by $t\eta$ and $t\psi$
respectively.
Let $s \in C^{\infty}(L)$ be an arbitrary smooth section of $L$.
\newcommand{\EEEE}{\exp\brackets{-i\psi_{\gm}(\gminv x)}}
To restate our claim,
$ (\gm,u)\brackets{\nabla s} = \nabla \parens{(\gm,u)[s]} $.

The left hand side is equal to
\begin{eqnarray}\label{EEEELHSa}
(\gm,u)\brackets{(\dif +i\eta)s}
=
u^{-1}\EEEE \gm\brackets{(\dif +i\eta)s}.
\end{eqnarray}

The right hand side is equal to 
\begin{eqnarray}
(\dif +i\eta)\parens{(\gm,u)[s]}
&=&
u^{-1}(\dif +i\eta) \parens{\EEEE \gm[s] }
\nonumber
\\
\label{EEEERHSa}
&=&
u^{-1}\EEEE  
\braces{
\dif (\gm[s])
-i\psi_{\gm}(\gminv)\cdot \gm[s]
+i\eta \cdot \gm[s]
}.
\end{eqnarray}
After deviding by $u^{-1}\EEEE$,
\begin{eqnarray}
\frac{\text{(\ref{EEEELHSa})}}{u^{-1}\EEEE}
&=&
\:\dif \gm[s] + i\parens{\gminv}^{*}\eta\cdot \gm[s]
\nonumber
\\
\label{EEEELHSb}
&=&
\:\dif \gm[s] + i\parens{\eta + \dif \psi_{\gminv}(x)}\cdot \gm[s],
\\
\label{EEEERHSb}
\frac{\text{(\ref{EEEERHSa})}}{u^{-1}\EEEE}
&=&
\dif (\gm[s])
-i\psi_{\gm}(\gminv)\cdot \gm[s]
+i\eta \cdot \gm[s].
\end{eqnarray}
Then we conclude that
\begin{eqnarray*}
\text{(\ref{EEEELHSb})}-\text{(\ref{EEEERHSb})}
=i\braces{\dif \psi_{\gminv}(x) + \dif \psi_{\gm}(\gminv x)}\gm[s]
=i\braces{\dif \psi_{e}(x)}\gm[s]
=0.
\end{eqnarray*}
 \end{proof}
\subsection{Indices of Almost Flat Bundles}
Now we can apply the theorem \ref{HWindex}
to each $A_{L_t}$ to obtain
\begin{prop}
\begin{eqnarray}\label{ApplyHW}
\tr_{G_{\alpha^{t}}} \mu_{G_{\alpha^{t}}} ([L_t]\hattensor[A])
=
\tr_{G_{\alpha^{t}}} \ind_{G_{\alpha^{t}}} (A_{L_t})
=
\int_{TX} c(x) \mathrm{Td}(T_{\C}X)
\wedge \mathrm{ch}(\sigma_{A}) \wedge \exp\parens{
\frac{-t}{2\pi}\omega}.
\end{eqnarray}
\end{prop}
Later on we will regard it as a polynomial in $t$.
\subsection{The Index Map $\mu_{\Galp}$}
We will compare two index maps $\mu_{G}$ and $\mu_{\Galp}$.
In particular, 
The main purpose of this subsection is
to verify the following proposition;
\begin{prop}
\label{IndexOfGroupExtension}
Let $A$ be a properly supported $G$-invariant
elliptic operator on $\bb{V}$ over $X$ of order $0$.
Consider $K$-homology class
$[A] \in K\!K^{G}(C_0(X),\C)$
and apply
$\map{\mu_{G}}{K\!K^{G}(C_0(X),\C)}{K_0(C^{*}(G))}$ to it.
One can regard also $[A]$ as a $K$-homology class
in $K\!K^{G_{\alpha}}(C_0(X),\C)$
and apply
$\map{\mu_{G_{\alpha}}}
{K\!K^{G_{\alpha}}(C_0(X),\C)}
{K_0(C^{*}(G_{\alpha}))}$ to it.
Then, 
$$
\mu_{\Galp}([A])=0 \in K_0(C^{*}(G_{\alpha}))
\qquad
\text{if and only if}
\qquad
\mu_{G}([A])=0 \in K_0(C^{*}(G)) .
$$
\end{prop}
It is convenient to introduce the following natural homomorphisms.
\begin{dfn}
Let $B$ be a $G$-algebra, in particular $B = \C$ or $C_0(X)$.
Define
\begin{itemize}
\item
$\map{\lambda}{C^{*}(G; B)}{C^{*}(G_{\alpha}; B)}$
 given by
 $$\lambda[b](\gamma,u) := b(\gamma)
\qquad \text{for } b \in C_c(G; B) \text{, and}
$$
\item
$\map{\theta}{C^{*}(G_{\alpha}; B)}{C^{*}(G; B)}$
is given by
$$
\theta
\left[ \widetilde{b} \right] (\gamma) :=
\int_{U(1)} b(\gamma, u)\: \dif u
\qquad \text{for } \widetilde{b} \in C_c(G_{\alpha}; B).
$$
\end{itemize}
\end{dfn}
Both of the operators are bounded with respect to 
the $C^{*}$-norms $\norm{\cdot}_{C^{*}(G;B)}$ and 
$\norm{\cdot}_{C^{*}(\Galp;B)}$,
so they define the homomorphisms 
between ${C^{*}(G; B)}$ and ${C^{*}(G_{\alpha}; B)}$.

In fact, for any covariant representation $(\pi, \rho)$
for $(B,G)$ and for $b\in C_c(G)$,
$$
\sigma_{\pi, \tilde{\rho}}[\lambda[b]]
=
\int_{\Galp} \pi(\lambda[b](\gm,u)) \tilde{\rho}(\gm,u)
\:\dif \gm \dif u
=
\int_{G} \pi(b(\gm))\rho'(\gm)\:\dif \gm
=
\sigma_{\pi,\rho'}[b],
$$
where $\rho'(\gm)$ is a representation of $G$ given by
$\rho'(\gm):= \int_{U(1)} \tilde{\rho}(\gm,u) \:\dif u$
is a representation on $G$.

Moreover for any covariant representation $(\pi, \tilde{\rho})$ for
$(B,\Galp)$
and for $\tilde{b} \in C_c(\Galp)$,
$$
\sigma_{\pi, \rho}\brackets{\theta\brackets{\tilde{b} }}
=
\int_{G} \pi(\theta\brackets{\tilde{b} }(\gm)) \rho(\gm) \:\dif \gm
=
\int_{\Galp} \pi(\tilde{b}(\gm,u)) \bar{\rho}(\gm,u) \:\dif u \dif \gm
=
\sigma_{\pi,\bar{\rho}}\brackets{\tilde{b} },
$$
where $\bar{\rho}(\gm,u)$ is a representation of $\Galp$
given by $\bar{\rho}(\gm,u) := \rho(\gm)$.

\begin{rem}
It follows immediately that $\theta \circ \lambda = \mathrm{id}$.
Note that $\widetilde{b} \in C_c(G_{\alpha}; B)$
is constant in $(1,u)\in U(1) \subset \Galp$ if and only if
$\lambda \circ \theta \brackets{\widetilde{b}}
 = \widetilde{b}$ holds.
\end{rem}

Now we will use the expression of $\mu_{G'}([A])$ in the
subsection \ref{SubsectionExpressionOfMu};
$\mu_{G'}([A])=\parens{\E^{G'}, \overline{A^{G'}}}$
for $G' = G$ or $\Galp$ and compare them.
\begin{lem}
If we regard $\mathcal{E}^{\Galp}$ as a module over
$C^{*}(G)$ through
$\map{\lambda}{C^{*}(G; B)}{C^{*}(G_{\alpha}; B)}$
equipped with the norm
$\theta \angles{\cdot, \cdot}_{\E^{\Galp} }$
valued in $C^{*}(G)$,
then it is isomorphic to $\E^{G}$ that is,
\begin{eqnarray}
\label{IsomBetweenCalE}
\parens{\E^{G}; \angles{\cdot, \cdot}_{\E^{G}}}
\simeq
\parens{\E^{\Galp};
\theta \angles{\cdot, \cdot}_{\E^{\Galp} }}
\qquad \text{as $C^{*}(G)$-modules.}
\end{eqnarray}
\end{lem}
\begin{proof}
Consider
the structure of Hilbert $C^{*}(\Galp)$-module
(\ref{ActionOnCalE}) and (\ref{ScalarProdOnCalE})
on $\mathcal{E}^{\Galp}$.
We will let the right action of $G'=G$ or $\Galp$
on $C_c(\bb{V})$
be denoted by $\underset{G'\!\!}{\cdot}$
when we have to distinguish them.
An important observation is that
the action of $U(1)$-component of $\Galp$ on $\bb{V}$
is trivial.
Therefore the action $\rho^{\Galp}$ of $C_c(\Galp)$
on $C_c(\bb{V})$
factors through the action of $C_c(G)$, that is, 
for $s\in C_c(\bb{V})$ and $\tilde{b}\in C_c(\Galp)$,
\begin{eqnarray*}
s {\underset{\Galp\!\!\!}{\cdot}} \tilde{b}
&=&
\int_{\Galp} {(\gm, u)}[s] \cdot
\tilde{b}({(\gm, u)}^{-1})
\dif {\gm} \dif u
\\
&=&
\int_{G} \gm[s] \cdot \parens{
\int_{U(1)} \tilde{b}(\gm, \alpha(\gm,\gminv)^{-1}u^{-1}) \dif u}
\dif {\gm}
\\
&=&
\int_{G} \gm[s] \cdot \theta\brackets{\tilde{b}}(\gm^{-1})
\dif {\gm}
\\
&=&
s{\underset{G}{\cdot}}
 \theta\brackets{\tilde{b}}.
\end{eqnarray*}
Replacing $\tilde{b}$ by $\lambda [b]$ for an arbitrary
$b\in C_c(G)$, we also obtain
$$
s {\underset{\Galp\!\!\!}{\cdot}} \lambda [b] =
s {\underset{G}{\cdot}} b.
$$
In addition
the scalar product
$\angles{\cdot, \cdot}_{\E^{\Galp}}$
 on $C_c(\bb{V})$ is constant in $(1,u)\in U(1) \subset \Galp$,
in the other words,
\begin{eqnarray*}
\angles{\cdot, \cdot}_{\E^{\Galp}}
=
\lambda
\parens{\angles{\cdot, \cdot}_{\E^{G}}
}
\qquad \text{and} \qquad\
\theta\parens{
\angles{\cdot, \cdot}_{\E^{\Galp}}
}
=
\angles{\cdot, \cdot}_{\E^{G}}
\qquad \text{on }C_c(\bb{V}). 
\end{eqnarray*}
This implies the two norms 
$\:\norm{s}_{\E^{G}}:=\sqrt{\norm{\angles{s, s}_{\E^{G}}}_{C^{*}(G)}}\:$
and 
$\:\norm{s}_{\E^{\Galp}}:=
\sqrt{\norm{\angles{s, s}_{\E^{\Galp}}}_{C^{*}(\Galp)}}\:$
on $C_c(\bb{V})$
are equivalent because both $\lambda$ and $\theta$ are bounded.
In particular $\E^{G} \simeq \E^{\Galp}$.
\end{proof}
\begin{proof}
[Proof of the Proposition \ref{IndexOfGroupExtension}]
Again, since the $U(1) \hookrightarrow \Galp$
acts trivially on $\bb{V}$,
\begin{eqnarray}
\overline{A^{\Galp}}
&=&
\int_{\Galp}(\gm, u)[\sqrt{c}A\sqrt{c}] \:\dif \gm \dif u
\nonumber
\\
&=&
\int_{\Galp}\gm[\sqrt{c}A\sqrt{c}] \:\dif \gm \dif u
\nonumber
\\
&=&
\int_{G}\gm[\sqrt{c}A\sqrt{c}] \:\dif \gm \int_{U(1)} \dif u
\qquad =
\overline{A^{G}}.
\label{CompatibiltyOperator}
\end{eqnarray}
Now assume
$ \mu_{G} = (\E^{G}, \overline{A^{G}}) = 0 \in K\!K(\C, C^{*}(G))$.
Then, by the definition,
it is homotopic to $(\{ 0\},0)$.
Precisely,
there exists a Kasparov module
$(\mathcal{E}, T) \in \bb{E}(\C, IC^{*}(G))$
satisfying 
$$
\mathcal{E} {\otimes}_{e_1} C^{*}(G) \cong \E^{G},\quad
T {\otimes}_{e_1} \mathrm{id} = \overline{A^{G}}
\qqand
\mathcal{E} {\otimes}_{e_0} C^{*}(G) = \{ 0 \},
$$
where $IC^{*}(G) = C([0,1];C^{*}(G))$ and
$\map{e_t} {IC^{*}(G)} {C^{*}(G)}$ denotes
the evaluation map at $t\in [0,1]$.

On the other hand, consider
$\mu_{\Galp} = (\E^{\Galp}, \overline{A^{\Galp}})
\cong (\E^{G}, \overline{A^{G}})$ 
as an element in $\bb{E}(\C, C^{*}(\Galp))$.
The Hilbert $IC^{*}(G)$-module $\mathcal{E}$
can be regarded as a Hilbert $IC^{*}(\Galp)$-module
through 
$$\map{1 \otimes \theta}{IC^{*}(G)}{IC^{*}(\Galp)}.$$
Therefore
$(\mathcal{E}, T)$,
which provides the homotopy between 
$(\E^{G}, \overline{A^{G}})$ and $(\{ 0\},0)$ in 
$\bb{E}(\C, C^{*}(G))$ also provides a homotopy between 
$(\E^{\Galp}, \overline{A^{\Galp}})$ and $(0,0)$ in 
$\bb{E}(\C, C^{*}(\Galp))$ in particular
$\mu_{\Galp} = (\E^{\Galp}, \overline{A^{\Galp}}) = 0
\in K\!K(\C, C^{*}(\Galp))
$.

As for the converse implication, which can be verified similarly,
however we can prove the more specific equality;
$$
\mu_{G}([A]). = \theta_{*}(\mu_{\Galp}([A])) 
$$
To see this, compare
\begin{eqnarray*}
\mu_{G} = \parens{\E^{G}, \overline{A^{G}}}
\qquad \text{and} \qquad
\theta_{*}(\mu_{\Galp}([A]))
=
\parens{
\E^{\Galp}{\otimes}_{\theta} C^{*}(G), \overline{A^{\Galp}} {\otimes} 1
}.
\end{eqnarray*}
It is sufficient to construct an isomorphism
$\varphi \colon
\E^{\Galp}\hattensor_{\theta} C^{*}(G)
\to \E^{G}$
as $C^{*}(G)$-modules
satisfying the equation;
$\varphi {\circ} \parens{ \overline{A^{\Galp}} {\otimes} 1}
= \overline{A^{G}} {\circ} \varphi$.
Set
$\varphi\colon s \otimes b \mapsto s \cdot b$.
Remark that
$s$ is regarded as an element of both of
$\E^{G}$ and $\E^{\Galp}$ through the isomorphism
(\ref{IsomBetweenCalE}) based on $\mathrm{id}_{C_c(\bb{V})}$.
Now 
we will have the invariance of the scalar products
under $\varphi$.
Recall that 
the scalar product on $\E^{\Galp} {\otimes}_{\theta} C^{*}(G)$
is given by
$$
\angles{s_1 {\otimes} b_1,\ s_2 {\otimes} b_2}
=
b_1^{*}*
\theta \parens{
\angles{s_1,\ s_2}_{\E^{\Galp}}
}*b_2
\in C^{*}(G),
$$
where $*$ denotes the usual convolution product on
$C^{*}(G)$. Noting that 
$\angles{\cdot, \cdot}_{\E^{\Galp}}
=
\lambda
\parens{\angles{\cdot, \cdot}_{\E^{G}}
}$,
we have
\begin{eqnarray*}
\angles{s_1 {\otimes} b_1,\ s_2 {\otimes} b_2}
&=&
b_1^{*}*
\theta \parens{
\angles{s_1,\ s_2}_{\E^{\Galp}}
}*b_2
\\
&=&
b_1^{*}*
\angles{s_1,\ s_2}_{\E^{G}}*
b_2
\\
&=&
\angles{s_1\cdot b_1,\ s_2\cdot b_2}_{\E^{G}}
\\
&=&
\angles{\varphi(s_1{\otimes} b_1),\
\varphi(s_2{\otimes} b_2)}_{\E^{G}}.
\end{eqnarray*}
This implies that $\varphi$ is continuous and injective.

Using an approximating unit $\{ e_m \}_{m\in \mathcal{M}}$
for $C^{*}(G)$,
the subset of all elements of the following form is dense
in $\E^{G}$.
\begin{eqnarray*}
\varphi(s{\otimes} e_m)
=
s\cdot e_m
=
q^{G}(i^{G}(s)*e_m)
\qquad {s \in C_c(\bb{V})}, \quad {m\in \mathcal{M}}.
\end{eqnarray*}
Therefore $\varphi$ is an isomorphism.
Compatibility with the operators is obvious since
$b$ does not involve $x\in X$ and (\ref{CompatibiltyOperator}).
\end{proof}

\section{Construction of the Infinite Product Bundle and its Index}
\begin{dfn}
Let us introduce a group
$$
\Ga := \braces{
((\gm_1,u_1),\ (\gm_2,u_2),\ (\gm_3,u_3),\ \ldots)
\in \prod_{k\in\N} G_{\alpha^{1/k}} \Big| \;
 \gm_1=\gm_2=\gm_3=\cdots 
}.
$$
We consider the product topology on 
$\displaystyle{ \prod_{k \in \N} G_{\alpha^{1/k}} }$.
 
In other words $\Ga$ is a central $U(1)^{\N}$ extension
of $G$ with respect to $\{ \alpha^{1/k} \}_{k}$.
$\Ga = G \times U(1)^{\N}$ as a topological measure space and 
equipped with the following product; 
$(\gm, \{u_k\})(\gm', \{u_k ' \}) =
\parens{ \gm \gm', \braces{ \alpha (\gm, \gm')^{1/k}u_k u_k '}}$.
We consider the product topology on $U(1)^{\N}$.
\end{dfn}
$\Ga$ acts on $X$ with the trivial $U(1)^{\N}$-action,
namely, via the natural homomorphism
$\Ga \to G$.
Notice that the fiber $U(1)^{\N}$ is compact
and hencce this action is proper.

$\Ga$ also acts on each $L_{1/n}$ via the natural homomorphism
$\Ga \to G_{\alpha^{1/n}}$ ignoring all the other components
than the $n$-th component $(\gm, u_n)$.

\begin{rem}
\label{RemMuGa}
Similar to Proposition \ref{IndexOfGroupExtension},
one may regard the $G$-equivariant $K$-homology class
$[A]\in K\!K^{G}(C_0(X),\C) $
also as an element of
$K\!K^{\Gaa}(C_0(X),\C)$.

Then it follows that
$$
\mu_{\Gaa}([A])
=0 \in K_{0}(C^{*}(\Ga))
\qquad \text{if and only if} \qquad
\mu_{G}([A])
=0 \in K_{0}(C^{*}(G)).
$$
The proof is the same after replacing $U(1)$ by $U(1)^{\N}$
because the action of the $U(1)^{\N}$-component is trivial. 

Similarly, if one starts with $G_{\alpha^{1/n}}$-equivariant
$K$-homology class
$\brackets{A_{1/n}} \in K\!K^{G_{\alpha^{1/n}}}(C_0(X),\C)$, 
it follows that
$$
\mu_{\Gaa}
\parens{\brackets{A_{1/n}}}
=0 \in K_{0}(C^{*}(\Ga))
\qquad \text{if and only if} \qquad
\mu_{G_{\alpha^{1/n}}}
\parens{\brackets{A_{1/n}}}
=0 \in K_{0}(C^{*}(G_{\alpha^{1/n}})).
$$

Also remind that
we can apply
Theorem \ref{HWindex}
{\upshape{\cite[Proposition 6.11.]{Wa14}}}
to $A_{L_{1/n}}$ regarded as $\Ga$-equivariant operator
to obtain the same value when we apply it 
to $A_{L_{1/n}}$ regarded as $G_{\alpha^{1/n}}$-equivariant operator;
$$
\tr_{\Gaa} \mu_{\Gaa}
\parens{\brackets{A_{1/n}}}
=
\int_{TX} c(x)  \mathrm{Td}(T_{\C}X)
\wedge \mathrm{ch}(\sigma_{A}) \wedge
\exp\parens{\frac{-1}{2\pi n} \omega}
=
\tr_{G_{\alpha^{1/n}}} \mu_{G_{\alpha^{1/n}}}
\parens{\brackets{A_{1/n}}}.$$
\end{rem}
This section proceeds as follows;
\begin{itemize}
\item
In the subsection \ref{SectionAssembling},
we construct a product bundle $\prod L_{1/k}$
and a flat bundle which is a quotient of $\prod L_{1/k}$.
\item
In the subsection \ref{SectionIndexOfProductBundles},
we calculate the index
$\mu_{\Gaa}\parens{\brackets{\prod L_{1/k}}\hattensor[A]}$ and
$\mu_{\Gaa}([W]\hattensor[A])$ and complete the
preparations for the proof of Theorem \ref{MainThm}.
\end{itemize}
\subsection{Assembling Almost Flat Bundles}
\label{SectionAssembling}
\begin{dfn}
Let us introduce the following $C^*$-algebras;
\begin{itemize}
\item
$\prod \C$ denotes the $C^{*}$-algebra
consisting of all bounded sequences with values in $\C$
equipped with the sup norm $\norm{(z_1, z_2, \cdots)} = \sup_{k}|z_k|$.
\item
$\dirsum{ \C} \subset \prod \C$ denotes an ideal
consisting of all sequences vanishing
at infinity.
This an ordinary direct sum of $C^{*}$-algebras
$\dirsum{\C} = {\displaystyle{\varinjlim_{k}
}} 
\parens{\bigoplus^{k} \C} $ .
\item
$Q$ is the quotient algebra,
$
Q := \parens{\prod \C} \big/ \parens{\dirsum{ \C}}.
$
\end{itemize}
All of them have the trivial gradings
and the trivial $\Galp$-actions.
\end{dfn}

We will follow \cite[Section 3]{Ha12}
to construct 
``infinite product bundle $\prod L_{1/n}$''
over $X$ which has a structure of
finite generated projective $\prod\C$-module.
We will do this in more general conditions than 
our requirements.
Let $B$ denote a unital $C^{*}$-algebra.

\newcommand{\path}{\mathcal{P}_1}
Let us fix some notations about the parallel transport.
The path groupoid $\path(X)$ consists of the points
in $X$ as objects and 
as morphisms, the space of thin homotopy classes of
piece-wise smooth paths connecting given two points
$$
\path(X)[x,y] :=
\braces{\text{thin homotopy classes of }
\map{p}{[0,1]}{X} \mid p(0)=x,\ p(1)=y}.
$$
The composition of morphisms 
is given by concatenation and re-parameterization
of paths.
A homotopy $\map{h} {[0,1]\times[0,1]} {X}$ between
two paths $p_0$ and $p_1$ from $x$ and $y$ satisfying that
$$
\left\{
\begin{array}{cccc}
h(j,\cdot) &=& p_j & \qqfor j=0,1,
\\
h(s,0) &=& x &
\\
h(s,1) &=& y & \qquad \text{for any } s \in [0,1]
\end{array}
\right.
$$
is called a
thin homotopy if it factors through a finite tree $T$; 
$$
[0,1]\times[0,1] 
\xrightarrow{h_1}
T 
\xrightarrow{h_2}
X
$$
and
$ \map{h_1(j,\cdot)} {[0,1]} {T} $ are
piece-wise linear for $j=0,1$.
For instance, re-parametrization is a thin homotopy,
and any path $p \in \path(X)[x,x]$ such that $p(t) = p(1-t)$
is thin homotopic to
the constant path at $x$.

If a Hilbert $B$-module $\Ga$-bundle $E$ over $X$ is given,
The transport groupoid $\mathcal{T}(E)$ consists of the points
in $X$ as objects and as morphisms,
$\mathcal{T}(E)[x,y] :=
\mathrm{Iso}_{B}\parens{E_x, E_y}$.
\newcommand{\area}{\mathrm{area}}
\begin{dfn}
A parallel transport of $E$ is a continuous functor
$\Phi^{E} \colon \path(X) \to \mathcal{T}(E)$.
$\Phi^{E}$
is called $\ep$-close to the identity
if for each $x\in X$ and
contractible loop $p\in \path(X)[x,x]$
,
it follows that
$$\norm{\Phi^{E}_{p} - \mathrm{id}_{E_x}} <\ep {\cdot} \area(D)$$
for any
two dimensional disk $D\subset X$ spanning $p$.
$D$ may be degenerated partially or completely.
\end{dfn}
\begin{rem}
Let $E$ be a Hermitian vector bundle
equipped with a connection $\nabla$ 
which is compatible with the Hermitian metric.
Let $\Phi^{E}$ be the parallel transport with respect to $\nabla$
in the usual sense.
If
its curvature $\omega$ has uniformly bounded
operator norm $\norm{\omega} < C$,
then
for any loop $p\in \path(X)[x,x]$
and any two dimensional disk $D\subset X$ spanning $p$,
it follows that
$\norm{\Phi^{E}_{p} - \mathrm{id}_{E_x}} <\int_{D}\norm{\omega}
<C {\cdot} \area(D)$
so it is $C$-closed to identity.
\end{rem}
\begin{prop}
\label{ConstructProductBundle}
Let $\{E^{k}\}$ be a sequence of Hilbert $B_{k}$-module
$\Ga$-bundles over $X$
with $B_{k}$ unital.
Assume that each parallel transport
$\Phi^{k}$ for $E^{k}$ is $\ep$-close to the identity
uniformly, that is, $\ep$
is independent of $k$. 

Then there exists a finitely generated Hilbert
$\prod_{k} B_{k}$-module
$\Ga$-bundle $V$ over $X$
with Lipschitz continuous transition functions in diagonal form
and so that the $k$-th component of this bundle
is isomorphic to the original $E^{k}$.

Moreover, 
if the parallel transport
$\Phi^{k}$ for each of $E^{k}$ comes from the 
$G$-invariant connection $\nabla^{k}$ on $E^{k}$,
$V$ is equipped with a continuous $G$-invariant
connection induced by $E^{k}$.
\end{prop}
\begin{proof}
We will essentially follow the proof of  \cite[Proposition 3.12.]{Ha12}.
For each $x\in X$ take a open ball $U_{x} \subset X$
of radius $\ll 1$ whose center is $x$.
Assume that each $U_{x}$ is geodesically convex.
Due to the corollary \ref{CorSliceThm} of the slice theorem,
there exists a sub-family of
finitely many open subsets $\{U_{x_i}, \ldots U_{x_N}\}$
such that
$X= 
\bigcup_{g\in \Ga} \bigcup_{i=1}^{N} g(U_{x_i})
$.

Fix $k$.
In order to simplify the notation,
let $U_i := U_{x_i}$
and $\map{\Phi_{y;x}}{E^{k}_{y}}{E^{k}_{x}}$ denote
the parallel transport of $E^{k}$
along the minimal geodesic from $y$ to $x$
for $x$ and $y$
in the same neighborhood $g(U_{i})$. 
Trivialize $E^{k}$ via 
$\map{\Phi_{y;x_i}}{E^{k}_{y}}{E^{k}_{x_i}}$
on each $U_{i}$.
Similarly  trivialize $E^{k}$ on each $g(U_{i})$ for $g\in \Ga$
via $\map{\Phi_{gy; gx_i}} {E^{k}_{gy}} {E^{k}_{gx_i}}$.
Note that since parallel transport commute with the action of $\Ga$,
namely, 
$\Phi_{gy; gx_i} = g\Phi_{y; x_i}g^{-1}$.

These provide a local trivializations for $E^{k}$ whose 
transition functions have uniformly bounded Lipschitz constants.
More precisely
we have to fix unitary isomorphisms 
$\map{\phi_{gx_i}}{E^{k}_{gx_i}}{\mathcal{E}^{k}}$
between the fiber on $gx_i$ and
the typical fiber $\mathcal{E}^{k}$.
Our local trivialization is
$
\map{\phi_{gx_i} \Phi_{y;gx_i}}{E^{k}_{y}}{\mathcal{E}^{k}} 
$.
The transition function on  
$g(U_{i}) \cap h(U_{j}) \neq \emptyset$ is given by
$$
y\mapsto
\psi_{g(U_{i}),h(U_{j})}(y):=
\parens{\phi_{hx_j} \circ \Phi_{y;hx_j}}
\parens{\phi_{gx_i}\circ \Phi_{y;gx_i}}^{-1}
\in \End_{B_{k}}(\mathcal{E}^{k}).
$$
Now we will estimate its Lipschitz constant.
If $y,z\in g(U_{i}) \cap h(U_{j})$,
\newcommand{\AAAA}{\Phi_{y;hx_j}}
\newcommand{\BBBB}{\Phi_{y;gx_i}^{-1}}
\newcommand{\CCCC}{\Phi_{z;hx_j}\Phi_{y;z}}
\newcommand{\DDDD}{\Phi_{y;z}^{-1}\Phi_{z;gx_i}^{-1}}
\begin{eqnarray*}
&&
\psi_{g(U_{i}),h(U_{j})}(y) - \psi_{g(U_{i}),h(U_{j})}(z)
\\
&=&
\parens{\phi_{hx_j} \Phi_{y;hx_j}}\parens{\phi_{gx_i} \Phi_{y;gx_i}}^{-1}
-
\parens{\phi_{hx_j} \Phi_{z;hx_j}}\parens{\phi_{gx_i} \Phi_{z;gx_i}}^{-1}
\\
&=&
\phi_{hx_j}
\braces{
\parens{\AAAA} \parens{\BBBB}
-
\parens{\CCCC} \parens{\DDDD}
}
\phi_{gx_i}^{-1}
\\
&=&
\phi_{hx_j}
\braces{
\parens{\AAAA - \CCCC} \parens{\BBBB}
+
\parens{\CCCC} \parens{\BBBB - \DDDD}
}
\phi_{gx_i}^{-1}.
\end{eqnarray*}
Since $\phi$'s and $\Phi$'s are isometry,
\begin{eqnarray}
\norm{\psi_{g(U_{i}),h(U_{j})}(y) - \psi_{g(U_{i}),h(U_{j})}(z)}
&\leq&
\norm{\AAAA - \CCCC}
+
\norm{\BBBB - \DDDD}
\nonumber\\
&=&
\norm{\Phi_{y; hx_j} \Phi_{z;y}\Phi_{hx_j; z} - \id_{E^{k}_{hx_j}}}
+
\norm{\Phi_{z; gx_i} \Phi_{y;z}\Phi_{gx_i; y} - \id_{E^{k}_{gx_i}}}
\nonumber\\
&\leq&
\ep {\cdot} (\area (D_1) + \area (D_2)).
\label{BoundedLipConst}
\end{eqnarray}
Here $D_1 \subset h(U_{j})$ is a two dimensional disk spanning
the piece-wise geodesic loops connecting
$hx_j$, $y$, $z$, and $hx_j$
and $D_2 \subset g(U_{i})$ is a two dimensional disk spanning
the piece-wise geodesic loop connecting
$gx_i$, $y$, $z$, and $gx_i$.

We claim that there exists a constant $C$ depending
only on $X$ such that
\begin{eqnarray}
\label{AreaOfD1D2}
\area (D_1),\ \area(D_2)
\leq
C{\cdot} \mathrm{dist}(y,z).
\end{eqnarray}
if we choose suitable disks $D_1$ and $D_2$.

We verify this 
using the geodesic coordinate
$\exp_{hx_j}^{-1} \colon h(U_{j}) \to T_{hx_j}$
centered at $hx_j \mapsto 0$.
More precisely,
let $p$ denote the minimal geodesic from
$y=p(0)$ to $z=p(\mathrm{dist}(y,z))$ with unit speed.
Consider
$$
D_0:=\braces{(r \cos \theta, r \sin \theta)\in \R^{2} \;\big|\;
0\leq r \>,\; 0\leq \theta \leq \mathrm{dist}(y,z) }
\subset \R^{2}
$$
and $\map{F}{D_0}{h(U_j) \subset X}$ given by
\begin{eqnarray*}
F(r \cos \theta, r \sin \theta)
:=
\exp_{\:hx_j}\parens{r \exp_{hx_j}^{-1}(p(\theta))}.
\end{eqnarray*}
Set $D_1 := F(D_0)$.
$F$ is injective if $\exp_{hx_j}^{-1}(y)$ and $\pm\exp_{hx_j}^{-1}(z)$
are on different radial directions, in which case
$F$ is a homeomorphism onto its image, and hence
$F(D_0)$ is a two dimensional disk spanning the target loop.
The Lipschitz constant of $F$ is bounded by a constant
depending on the curvature on $h(U_{j})$,
so there exists a constant $C_{h,j}$ depending on the
Riemannian curvature on $h(U_j)$ satisfying 
$$
\area(D_1) \leq C_{h,j} {\cdot} \area(D_0)
\leq C_{h,j} {\cdot} \mathrm{dist}(y,z).
$$
But actually, the constant $C_{h,j}$ can be taken independent of 
$h(U_j)$ due to the bounded geometry of $X$ (Corollary 
\ref{CorSliceThm}).
In the case of 
$\exp_{hx_j}^{-1}(y)$ and $\pm\exp_{hx_j}^{-1}(z)$
are on the same radial direction, $D_1$
is completely degenerated and $\area(D_1) = 0$.
We can construct $D_2$ in the same manner
so the claim (\ref{AreaOfD1D2}) has been verified.

Therefore combining (\ref{BoundedLipConst})
and (\ref{AreaOfD1D2}), we conclude that
the Lipschitz constants of the transition functions
of these local trivialization
are less than
$2C\ep$,
which are independent of $E^{k}$ and $U_{i}$ and $g\in \Ga$,
in particular, the product of them
$$
\Psi_{g(U_{i}),h(U_{j})} :=
\braces{\psi^{k}_{g(U_{i}),h(U_{j})}}_{k\in \N}
\colon
g(U_{i}) \cap h(U_{j})
\to
\End_{\parens{\prod B_{k}}}\parens{\prod_{k} \mathcal{E}^{k}}
$$
is Lipschitz continuous.
So we are allowed to use $\Psi$ to define the
Hilbert $\prod_k B_{k}$-module bundle $V$ as required.
Precisely
$V$ and be constructed as follows;
$$
V := \bigsqcup_{g,i} \parens{ g(U_{i}) \times
\prod_{k} \mathcal{E}^{k}}
\Big/ \sim.
$$
Here
$(x,v) \in g(U_{i}) \times
\prod_{k} \mathcal{E}^{k}$ and 
$(y,w) \in h(U_{j}) \times
\prod_{k} \mathcal{E}^{k}$
are equivalent if and only if
$x=y \in g(U_{i}) \cap h(U_{j}) \neq \emptyset$ and
$\Psi_{g(U_{i}),h(U_{j})}(v)=w$.
By the construction of $V$,
if $\map{p_n}{\prod_k B_{k}}{B_{n}}$ denote the projection
onto the $n$-th component,
$V \hattensor_{p_n} B_{n}$ is isomorphic to the original
$n$-th component $E^{n}$.

In order to verify the continuity of the induced connection,
let $\{ \mathbf{e}^{i} \}$ be an orthonormal local frame on
$U_{i}$ for an arbitrarily fixed $E^{k}$
obtained by the parallel transport along the minimal
geodesic from the center $x_i \in U_{i}$,
namely,
$\mathbf{e}^{i}(y) = \Phi_{x_i ; y}\mathbf{e}^{i}(x_i)$.
It is sufficient to verify that
$
\norm{\nabla^{k}\mathbf{e}_{i}} < C
$.
Let $ v \in T_{y}X $ be a unit tangent vector and 
$p(t) := \exp_{y}(tv)$ be the geodesic of unit speed
with direction $v$.
\begin{eqnarray*}
\nabla^{k}_{v}\mathbf{e}_{i}(y)
&=&
\lim_{t\to 0}\frac{1}{t}\parens{
\Phi_{p(t);p(0)}\mathbf{e}^{i}(p(t)) - \mathbf{e}^{i}(p(0))
}
\\
&=&
\lim_{t\to 0}\frac{1}{t}\parens{
\Phi_{p(t);p(0)}\Phi_{x_i ; p(t)}
 - \Phi_{x_i ; p(0)}
}
\mathbf{e}^{i}(x_i),
\\
\norm{\nabla^{k}_{v}\mathbf{e}_{i}(y)}
&\leq&
\lim_{t\to 0}\frac{1}{|t|}\norm{
\Phi_{p(t);p(0)}\Phi_{x_i ; p(t)}
 - \Phi_{x_i ; p(0)}
}
\\
&\leq&
\lim_{t\to 0}\frac{1}{|t|}
\ep\cdot \area(D(t)),
\end{eqnarray*}
where $D(t)$ is a $2$-dimensional disk in $U_{i}$
spanning the piece-wise geodesic
connecting $x_{i}$, $p(0)=y$, $p(t)$ and $x_{i}$.
As above, we can find a constant $C>0$ and
disks $D(t)$ satisfying
$$
\area(D(t)) \leq C {\cdot} \mathrm{dist} (p(0), p(t))
= C |t|
$$
 for $|t| \ll 1$.
Hence, we obtain 
$$
\norm{\nabla^{k}_{v}\mathbf{e}_{i}(y)}
\leq
C\ep.
$$
 \end{proof}

\begin{dfn}
\label{ConstructW}
Set $Q=\parens{\prod B_{k}}
\big/
\parens{\dirsum{ B_{k}}}$
and let
$
\map{\pi}{\prod_k B_{k}}
{Q}
$ denote the quotient map and define
$$
W:= V\hattensor_{\pi}Q
$$
as a Hilbert $Q$-module bundle.
\end{dfn}
The family of parallel transport
of $E^{k}$ induces the parallel transport
$\Phi^{W}$ of
$W$ which commutes with the action of $\Ga$.
\begin{prop}
\label{FlatnessOfW}
If the parallel transport of $E^{k}$ is $C_k$-close
to the identity with $C_k \searrow 0$,
then
$W$ constructed above is a flat bundle.
More precisely the parallel transport
$\Phi^{W}(p) \in \Hom(W_x, W_y)$ depends only on
the ends-fixing homotopy class of $p \in \path(X)[x,y]$.
\end{prop}
\begin{proof}
It is sufficient to prove that for any contractive loop
$p\in \path(X)[x,x]$,
it satisfies 
$\Phi^{W}(p) = \mathrm{id}_{W_x}$.
Fix a two dimensional disk $D\subset X$ spanning $p$.
For arbitrary $\ep >0$ there exists $n_0$ such that
every $k \geq n_0$ satisfies that
$\Phi^{E^{k}}$ is $\frac{\ep}{1+\area(D)}$-close to the identity.
Then,
\begin{eqnarray*}
\norm{\Phi^{W}(p) - \mathrm{id}_{W_x}}
=
\limsup_{k\to \infty} \norm{\Phi^{E^{k}}(p) - \mathrm{id}}
\leq
\sup_{k\geq n_0} \norm{\Phi^{E^{k}}(p) - \mathrm{id}}
\leq
\frac{\ep}{1+\area(D)}{\cdot}\area(D)
\leq
\ep.
\end{eqnarray*}
This implies $\Phi^{W}(p) = \mathrm{id}_{W_x}$.
 \end{proof}
Applying the proposition \ref{ConstructProductBundle}
to $E^{k}=L_{1/k}$ with all $B_{k}=\C$ and 
we obtain$V=\prod L_{1/k}$
and construct
a $Q$-bundle over $X$ denoted by $W$ as definition 
\ref{ConstructW},
here,
$Q = \parens{\prod \C} \big/ \parens{\dirsum{ \C}}$.
Since $L_{1/k}$ has the curvature $\frac{i}{k}\omega$,
whose norm converges to zero as $k \to \infty$,
we can apply the proposition \ref{FlatnessOfW} to
see that $W$ is flat.
As a summary,
\begin{itemize}
\item We have constructed the infinite product bundle
$\prod L_{1/k}$
as a Hilbert $(\prod \C)$-module bundle,
\item and  
$W=\parens{\prod L_{1/k}} \! \otimes_{\pi}Q$
as a Hilbert
$Q = \parens{\prod \C} \big/ \parens{\dirsum{ \C}}$-module bundle.
\item
The actions of $\Ga$ on $L_{1/n}$ induce actions on
$\prod L_{1/k}$ and $W$.
\item
The $\Ga$-invariant connections on $L_{1/n}$
induce $\Ga$-invariant connection on $\prod L_{1/k}$
and $W$. Furthermore $W$ is flat .
\end{itemize}
\subsection{Index of the Product Bundle}
\label{SectionIndexOfProductBundles}
Following the definition \ref{BCmap},
we define the index maps with coefficients;
\begin{dfn}
For $C^*$-algebras $B = \prod \C$ or $Q$,
define the index map
$$
\map{\mu_{\Gaa}}{K\!K^{\Gaa}(C_0(X), B)}
{K_{0}(C^{*}(\Ga)\otimes B)}
$$
as the composition of
\begin{itemize}
\item
$\map{j^{\Gaa}}{K\!K^{\Gaa}(C_0(X), B)}
{K\!K(C^*(\Ga;C_0(X)) ,\: C^{*}(\Ga;B))}$
with
\item
$\map{[c]\hattensor}{K\!K(C^*(\Ga;C_0(X)) ,\: C^{*}(\Ga;B))}
{K_{0}(C^{*}(\Ga;B))}$.
\end{itemize}
In other words, 
\begin{eqnarray*}
\mu_{\Gaa}(\mathchar`-) :=
[c]\hattensor_{C^*(\Gaa;C_0(X))} j^{\Gaa}(\mathchar`-)
\quad \in K_{0}(C^{*}(\Ga;B)).
\end{eqnarray*}
Note that $\Ga$ acts on $B$ trivially so 
$C^{*}_{\Max}(\Ga; B)$ and 
$C^{*}_{\red}(\Ga; B)$
will be naturally identified with
$C^{*}_{\Max}(\Ga)\!\otimes_{\Max}\! B$ and 
$C^{*}_{\red}(\Ga)\!\otimes_{\min}\! B$ respectively.
Moreover $B$ is commutative and hence nuclear,
the tensor product with $B$
is uniquely determined, so we will abbreviate $\otimes_{\Max} B$
and $\otimes_{\min} B$ as $\otimes B$.
\end{dfn}
\begin{dfn}
Define $C_0 \parens{\prod L_{1/k}}$ as a space consisting of
sections $\map {s} {X} {\prod L_{1/k}}$
vanishing at infinity.
$C_0 \parens{\prod L_{1/k}}$ is considered as a
$C_0\parens{X; \prod \C} \cong C_0(X) \otimes
\parens{ \prod \C} $-module
equipped with the scalar product
$$
\angles{ s, s'}_{C_0 \parens{\prod L_{1/k}}} (x) := \braces{
\angles{s_{k}(x) ,\ s'_{k}(x)}_{L_{1/k}}
 }_{k\in \N} \in C_0\parens{X; \prod \C},
$$
where $s_{k}$ denotes the $k$-th component of $s$.

Similarly, define $C_0 (W)$ as a space consisting of
sections $\map {s} {X} {W}$
vanishing at infinity.
$C_0 (W)$ is considered as a
$C_0\parens{X; Q} \cong C_0(X) \otimes
Q $-module
equipped with the scalar product
$$
\angles{ s, s'}_{C_0 (W)}(x) := \braces{
\angles{s_{k}(x) ,\ s'_{k}(x)}_{L_{1/k}}
 }_{k\in \N} \in C_0\parens{X; Q}.
$$
\end{dfn}
\begin{dfn}
$\prod L_{1/k}$ and $W$ define elements in $K\!K$-theory
\begin{eqnarray*}
\brackets{\prod L_{1/k}}
 = \parens{C_0 \parens{\prod L_{1/k}} ,\ 0}
&\in& K\!K^{\Gaa}\parens{C_0(X) ,\ C_0(X)\hattensor \prod \C}
\\
\text{and }\qquad
[W] = \big(C_0(W) ,\ 0\big)
&\in& K\!K^{\Gaa}\parens{C_0(X) ,\ C_0(X)\hattensor Q}
\end{eqnarray*}
The action of $C_0(X)$ on $C_0(\prod L_{1/k})$
and $C_0(W)$ are the
point-wise multiplications.
\end{dfn}
\begin{rem}
Actually,  
$
f\parens{\mathrm{id}_{C_0 \parens{\prod L_{1/k}} } - 0^{2} }
\in
\bb{L}\parens{ C_0 \parens{\prod L_{1/k}} }
$
is a $\Ga$-equivariant $\prod \C$-compact operator
for any $f\in C_0(X)$.
To see this, it is sufficient to check it when $f$ has a compact support.
Let ${\mathbf{1}} := \{ 1,1,\ldots \} \in \prod \C$ denote 
the identity element in $\prod \C$.
We may regard $\prod L_{1/k}$
as a trivial bundle $X \times (\prod \C)$ equipped with 
the connection $\nabla^{1/k}$ on each $k$-th component.

Let $s_1\in C_0 \parens{\prod L_{1/k}}$ be a section such that
$s_1(x) = \mathbf{1}$ for any $x\in \supp(f)$.
Then,
$
f {\cdot} \mathrm{id}_{C_0 \parens{\prod L_{1/k}} }
=
fs_1 \angles{ s_1, \cdot }_{C_0 \parens{\prod L_{1/k}} }
$
is of rank $1$ and hence a compact operator.
The similar verification works also in the case of $W$.
\end{rem}
\begin{prop}\label{MuW0}
Let $[A]$ be a $K$-homology element in $K\!K^{\Gaa}(C_0(X), \C)$
determined by a Dirac operator on $\bb{V}$,
and let $W$ be the flat $Q$-module bundle
constructed as above.

Then
$\mu_{\Gaa}\parens{[W]\hattensor [A]}=0
\in K_{0}(C^{*}(\Ga) \otimes Q)$
if $\mu_{G}([A]) = 0$.
\end{prop}
In order to prove this,
we introduce an element $[W]_{\mathrm{rpn}}\in K\!K^{\Gaa}(Q,Q)$
using the holonomy representation of $\Ga$.
\begin{lem}
The parallel transport of $W$ depends only on
the ends of the paths.
\end{lem}
\begin{proof}
Let $x$ and $y$ be any points in $X$,
and consider any two paths $p_1$ and $p_2 \in \path(X)[x,y]$ from $x$ to $y$.
Combining them,
we have a loop $p=(p_1)^{-1} p_2 \in \path(X)[x,x]$
starting and ending at $x\in X$.
Now we want to verify the parallel transport 
$\Phi(p_1)^{-1} \Phi(p_2) = \Phi(p)$
along the loop $p$ is the identity map $\mathrm{id}_{W_x}$.

Since $W$ is flat, the parallel transport on $W$ depends only on
the homotopy type of paths, so it provides the homomorphism
$\map{\Phi} {\pi_1(X)} {U(W_x)} $,
which is called the holonomy representation of $\pi_1(X)$.
The image of this homomorphism is contained in 
the structure group of $W$.

Remark that the sturucture group of $\prod L_{1/k}$ is of the diagonal form
$$
\prod U(1) \subset U\parens{\prod \C},
$$
and hence abelian. The structure group of $W$,
 which is a quotient of this group,
is also abelian.

Noting that $H_1(X;\Z)$ is the abelianization of $\pi_1(X)$, 
the above holonomy representation should factor through
$$\map{\Phi} {\pi_1(X)} {H_1(X;\Z)} \to U(W_x), $$
which must be the constant map to $\mathrm{id}_{W_x}$
by the assumption of $H_1(X;\Z)=0$.
\end{proof}
\begin{dfn}
From now on let $\Phi_{x;y}$ denote the parallel
transport of $W$ along an arbitrary path
from $x\in X$ to $y\in X$.
Let us fix a base point $x_0 \in X$
and let $W_{x_0}$ be the fiber on $x_0$.
$W_{x_0}$ is regarded as a Hilbert $Q$-module.
Define $[W]_{\mathrm{rpn}}$ as
$$
[W]_{\mathrm{rpn}} := (W_{x_0}, 0) \in K\!K^{\Gaa}(\C,Q).
$$
The action of $\Ga$ on $W_{x_0}$ is obtained by
the holonomy representation of $\Ga$,
which is denoted by $\map{\rho}{\Ga}{\End_{Q}(W_{x_0})}$
and defined by the formula;
$$
\rho [g](w) = (\Phi_{x_0;gx_0})^{-1} g(w) 
\qqfor g\in \Ga, \: w \in W_{x_0}.
$$
\end{dfn}
\begin{lem}
$$
[W]\hattensor_{C_0(X)} [A] = [A]
\hattensor_{\C} [W]_{\mathrm{rpn}}
\in K\!K^{\Gaa}(C_0(X), Q).
$$
In particular we obtain
\begin{eqnarray*}
\mu_{\Gaa}\parens{[W]\hattensor_{C_0(X)} [A]}
=
\mu_{\Gaa}\parens{ [A]
\hattensor_{\C} [W]_{\mathrm{rpn}} }
\in K_{0}(C^{*}(\Ga) \otimes Q).
\end{eqnarray*}
\end{lem}
\begin{proof}
Recall that $[A]\in K\!K^{G}(C_0(X), \C)$ is given by
$\left( L^2(\mathbb{V} ),\ A \right) $ and then,
$$
[W]\hattensor_{C_0(X)} [A] = \parens{
C_0(W) {\otimes}_{C_0 (X)} L^2 (\mathbb{V})
,\ A_{W} 
},
$$
where $A_{W}$ is the Dirac type operator $A$ twisted by $W$
acting on $L^{2}(W {\otimes} \bb{V})
\simeq C_0(W) {\otimes}_{C_0(X)}L^{2}(\bb{V})$.
The action of $C_0 (X)$ on $C_0(W)$ and
$L^2 (\mathbb{V})$ are 
the point-wise multiplications.

$L^{2}(W {\otimes} \bb{V})$ is regarded as a Hilbert $Q$-module
in the following way.
First, each fiber $W_{x} \otimes \bb{V}_{x}$ on $x\in X$
is of a structure of Hilbert $Q$-module
with the right $Q$ action given by
$
(w{\otimes}v){\cdot} q = (w{\cdot}q)\otimes v
$
 for $w\in W_{x}$, $v\in \bb{V}_{x}$ and $q\in Q$,
and the scalar product given by
$
\angles{w_1{\otimes}v_1 ,\ w_2{\otimes}v_2}_{W_{x} \otimes \bb{V}_{x}}
= \angles{w_1, w_2}_{W_{x}} \angles{v_1, v_2}_{\bb{V}_{x}}
 \in Q $
for $w_1, w_2 \in W_{x}$ and $v_1, v_2 \in \bb{V}_{x}$.

Then, $ L^2 (W{\otimes}\mathbb{V})$
is of a structure of Hilbert $Q$-module
with the right $Q$-action by the point-wise actions,
and the scalar product given by
$
\angles{u_1, u_2}_{L^2 (W{\otimes}\mathbb{V})}
=\int_{X} \angles{u_1(x), u_2(x)}_{W_{x} \otimes \bb{V}_{x}} 
\:\dif \mathrm{vol}(x)
 \in Q$
for $u_1, u_2 \in L^2 (W{\otimes}\mathbb{V})$.

On the other hand, 
$$
[A] \hattensor_{\C}[W]_{\mathrm{rpn}}
=
\left( L^2 (\mathbb{V}) {\otimes}_{\C}
W_{x_0},\ A {\otimes} 1 \right).
$$
The action of $C_0 (X)$ is the point-wise multiplications.
Note that the action of $\Ga$
on $W_{x_0}$ is given by the holonomy representation $\rho$.
It is sufficient to give an isomorphism
$$
\map{\varphi}
{L^2 (\mathbb{V}) {\otimes}_{\C} W_{x_0}}
{C_0(W) {\otimes}_{C_0 (X)} L^2 (\mathbb{V})}
$$
which is compatible 
with the action of $\Ga$ and the operators
$A_{W}$ and $A{\otimes}1$. 
Set 
$\overline{w} \colon x \mapsto \Phi_{x_0 ; x}w \in W_x$
and define 
$\varphi$ on a dense subspace $C_c(\bb{V}) {\odot} W_{x_0}$ as
$$
\varphi (s\otimes w) := \overline{w} {\cdot} \chi \otimes s
\qquad \text{
for $s\in C_c (\mathbb{V})$ and $w \in W_{x_0}$},
$$
where $\chi \in C_0(X)$ is an
arbitrary compactly supported function on $X$
with values in $[0,1]$
satisfying that $\chi(x)=1$ for all $x\in \supp(s)$.

$\varphi$ is independent of the choice of $\chi$ and hence
well-defined. In fact, let $\chi' \in C_c(X)$ be another such function,
and let $\rho \in C_c(X)$ be a compactly supported function on $X$
with values in $[0,1]$
satisfying that $\rho(x)=1$ for all
$x\in \supp(\chi) \cup \supp(\chi')$. Then
in $ C_0(W) {\otimes}_{C_0 (X)} C_c(\mathbb{V})$,
$$
\overline{w} {\cdot} \chi \otimes s
 -  \overline{w} {\cdot} \chi' \otimes s
=
\overline{w} {\cdot} \parens{\chi-\chi'} \otimes s
=
\overline{w} {\cdot} \rho {\cdot} \parens{\chi-\chi'} \otimes s
=
\overline{w} {\cdot} \rho \otimes \parens{\chi-\chi'} s
=
0.
$$
In other words, 
$\varphi (s \otimes w) = \overline{w} {\cdot} \chi \otimes s$
corresponds to the section
of point-wise tensor products;
$x \mapsto \overline{w}(x) \otimes s(x)
\in L^{2}(W \otimes \bb{V})$
under the identification
$C_0(W) {\otimes}_{C_0 (X)} L^2 (\mathbb{V})
 \cong L^{2}(W \otimes \bb{V})$.

Firstly, compatibility with the operators is verified as follows;
$$
A_{W} {\circ} \varphi (s\otimes w) 
=
A_{W} (\overline{w} \otimes s) 
=
\overline{w} \otimes A(s)
=
\varphi {\circ} (A{\otimes}1) (s\otimes w)
$$
for $s\in C_c (\mathbb{V})$ and $w \in W_{x_0}$.
This is because $\nabla^{W} \overline{w} = 0$ by its construction.

Compatibility with the action of $\Ga$ is verified as follows;
\begin{eqnarray*}
\varphi (g (s\otimes w))(x)
= \Phi_{x_0 ; x}(\rho[g](w)) \otimes g(s(g^{-1}x))
 &=& \Phi_{x_0 ; x} (\Phi_{x_0;gx_0})^{-1} g(w) \otimes g(s(g^{-1}x))
 \\
 &=& \Phi_{gx_0 ; x}(g(b)) \otimes g(s(g^{-1}x)),
\\
g (\varphi (s\otimes w))(x)
 = g((\Phi_{x_0;g^{-1}x})(w) \otimes s(g^{-1}x))
 &=& \Phi_{g x_0; x}(g(w)) \otimes g(s(g^{-1}x)).
\end{eqnarray*}

Let us check that $\varphi$ induces an isomorphism.
For $s_1\otimes w_1,\ s_2\otimes w_2 \in 
C_c (\mathbb{V}) {\odot}_{\C} W_{x_0}
$,
\begin{eqnarray*}
&&
\angles{
\varphi(s_1\otimes w_1),\;
\varphi(s_2\otimes w_2)
}_{C_0(W) {\otimes}_{C_0 (X)} L^2 (\mathbb{V})}
\\
&=&
\Big\langle
s_1 ,\
\angles{\overline{w_1} {\cdot} \chi,
 \overline{w_2} {\cdot} \chi}_{C_0(W)} s_2
\Big\rangle_{L^2 (\mathbb{V})}
\\
&=&
\int_{X} \Big\langle s_1(x),\ 
\angles{(\Phi_{x_0 ; x}w_1)\chi(x),
(\Phi_{x_0 ; x}w_2)\chi(x) }_{W_{x}} s_2(x)\Big\rangle_{\bb{V}_{x}}
\:\dif \mathrm{vol}(x)
\\
&=&
\int_{X} \angles{w_1,w_2}_{W_{0}} \chi(x)^2 \angles{s_1(x),
s_2(x)}_{\bb{V}_{x}}
\:\dif \mathrm{vol}(x)
\\
&=&
\angles{w_1,w_2}_{W_{0}} \angles{s_1,s_2}_{L^{2}(\bb{V})}
\\
&=&
\angles{
s_1\otimes w_1,\ 
s_2\otimes w_2
}_{L^{2}(\bb{V}){\otimes} W_{x_0}},
\end{eqnarray*}
where $\chi \in C_0(X)$ is a compactly supported function on $X$
satisfying that $\chi(x)=1$ for all $x\in \supp(s_1) \cup \supp(s_2)$.
This implies that $\varphi$ is continuous and injective.

Moreover, choose arbitrary $F \in C_c(W)$ and
$s \in C_c(\bb{V})$.
Let $\{w_1, w_2, \ldots \}$ be a basis over $\C$ for $W_{x_0}$.
Since for each $x\in X$,
$\braces{ \overline{w_1}(x),  \overline{w_2}(x), \ldots}$
provides a basis for $W_{x}$,
there exist $f_1, f_2, \ldots \in C_c(X)$
such that
$$F(x) = \sum_{j\in \N}  \overline{w_j}(x) f_j(x) $$
for any $x\in X$.
Let $\chi \in C_c(X)$ be a function satisfying $\chi(x)=1$
for any $x \in \supp(s) \cup \supp(F)$.
Then
\begin{eqnarray*}
\varphi\parens{
\sum_{j\in \N} f_j s \otimes w_j
}
=
\sum_{j\in \N} \parens{ \overline{w_j} {\cdot} \chi \otimes f_j s }
=
\sum_{j\in \N} \parens{ \overline{w_j} {\cdot} \chi f_j \otimes s }
=
\parens{ \sum_{j\in \N} \overline{w_j} f_j } {\cdot} \chi \otimes s
=
F \otimes s.
\end{eqnarray*}
This implies that the image of $\varphi$ is dense in
$C_0(W) {\otimes} L^{2}(\bb{V})$.
Therefore $\varphi$ induces an isomorphism.
 \end{proof}
\begin{proof}
[Proof of the Proposition \ref{MuW0}]
Due to the previous lemma it follows that
\begin{eqnarray*}
\mu_{\Gaa}\parens{[W]\hattensor [A]}
&=&
\mu_{\Gaa} ([A] \hattensor [W]_{\mathrm{rpn}})
\\
&=&
[c] \hattensor j^{\Gaa}\parens{
[A] \hattensor [W]_{\mathrm{rpn}}
}
\\
&=&
[c] \hattensor \left(j^{\Gaa}[A]\right) \hattensor
\parens{j^{\Gaa} [W]_{\mathrm{rpn}} }
\\
&=&
(\mu_{\Gaa}[A]) \hattensor \parens{j^{\Gaa} [W]_{\mathrm{rpn}} }.
\end{eqnarray*}
As in Remark \ref{RemMuGa},
$\mu_{G}[A] =0$ if and only if $\mu_{\Gaa}[A] =0$,
thereforet the assumption $\mu_{G} [A] =0$ implies
$\mu_{\Gaa}\parens{[W]\hattensor [A]}=0$.
 \end{proof}


Next, we consider $\prod L_{1/k}$,
which is a $\prod \C$
-module bundle of rank $1$.
Consider the index maps in the same way as above,
\begin{eqnarray*}
\mu_{\Gaa} \parens{ \brackets{\prod L_{1/k}} \hattensor [A]}
\in K_{0} \parens{C^*(\Ga) {\otimes} \parens{\prod \C}}.
\end{eqnarray*}
\begin{prop}
\label{PiMu}
\text{}
\begin{enumerate}
\item
Let $\map{p_n}{\prod \C}{\C}$ denote the projection
onto the $n$-th component and consider
$$
\map{(1 \otimes p_n)_{*}}
{K_{0} \parens{C^*(\Ga) {\otimes} \parens{\prod \C}}}
{K_{0} \parens{C^*(\Ga)}}.
$$
Then
$$
(1 \otimes p_n)_{*}
\mu_{\Gaa} \parens{ \brackets{\prod L_{1/k}} \hattensor [A]}
=
\mu_{\Gaa} \parens{ \brackets{L_{1/n}} \hattensor [A]}.
$$
\item
Let $\map{\pi}{\prod \C}{Q}$ denote the quotient map
and consider
$$
\map{(1 \otimes \pi)_{*}}
{K_{0} \parens{C^*(\Ga) {\otimes} \parens{\prod \C}}}
{K_{0} \parens{C^*(\Ga) {\otimes} Q}}.
$$
Then
$$
(1 \otimes \pi)_{*}
\mu_{\Gaa} \parens{ \brackets{\prod L_{1/k}} \hattensor [A]}
=
\mu_{\Gaa} ([Q] \hattensor [A]).
$$
\end{enumerate}
\end{prop}
\begin{proof}
As for the first part, 
$\brackets{L_{1/n}} = (p_n)_* \brackets{\prod L_{1/k}}
\in K\!K^{\Gaa}(C_0(X), C_0(X)\hattensor \C)$ by the construction of
$\prod L_{1/k}$.
Then it follows that
\begin{eqnarray*}
\mu_{\Gaa} \parens{ \brackets{L_{1/n}} \hattensor [A]}
&=&
\mu_{\Gaa} \parens{ (p_n)_* \brackets{\prod L_{1/k}} \hattensor [A]}
\\
&=&
[c]\hattensor j^{\Gaa}
\parens{\brackets{\prod L_{1/k}} \hattensor [A] \hattensor{\mathbf{p_n}}}
\\
&=&
[c]\hattensor j^{\Gaa} \parens{\brackets{\prod L_{1/k}} \hattensor [A]}
\hattensor j^{\Gaa}({\mathbf{p_n}})
\\
&=&
\mu_{\Gaa} \parens{\brackets{\prod L_{1/k}} \hattensor [A]}
\hattensor j^{\Gaa}({\mathbf{p_n}})
\\
&=&
(1 \otimes p_n)_* \mu_{\Gaa} 
\parens{ \brackets{\prod L_{1/k}} \hattensor [A]},
\end{eqnarray*}
where $\mathbf{p_n} = (\C,\ p_n,\ 0) \in K\!K\parens{\prod \C, \C}$. 
Then note that  
$
j^{\Gaa}(\mathbf{p_n})
=
(C^{*}(\Ga),\ 1 \otimes p_n,\ 0)
$
is equal to the element in $K\!K$-group
$
K\!K\parens{C^{*}(\Ga) {\otimes} \prod \C,\ C^{*}(\Ga)}
$
which is determined by the map
$\map{1 \otimes p_n} 
{C^{*}(\Ga) {\otimes} \prod \C} {C^{*}(\Ga)}$.

Since
$\pi_{*}\brackets{\prod L_{1/k}}
= [W] \in K\!K^{\Gaa}(C_0(X),\ C_0(X) {\otimes} Q)$
by the construction of $W$,
the second part can be proved in the similar way.
 \end{proof}

\section{Proof of the Main Theorem}
Let us restate the settings and the claim.

Let $X$ be a complete Riemannian manifold with $H_1(X;\Z)$
and let $G$ be a second countable locally compact Hausdorf
unimodular group
which is acting on $X$ properly and isometrically.
Assume $X$ is $G$-compact.
Let $\map{\mu_{G}}{K\!K^{G}(C_0(X), \C)}{K_0(C^{*}(G))}$
be the index map,
and $E = L^{1} \oplus \ldots \oplus L^{N}$ be a finite direct sum of
smooth Hermitian $G$-line bundles
each of which has vanishing first Chern class; $c_1(L^{k}) = 0
\in H^{2}(X;\R)$.
Let $A$ be a properly supported $G$-invariant elliptic operator
of order $0$.
\begin{thmalph}
If $[A] \in K\!K^{G}(C_0(X), \C)$ satisfies that
$$
\angles{[E], [A]}_{G} := \tr_{G} \mu_{G}([E] \hattensor [A])
\neq 0 \in \R,
$$
then $\mu_{G}([A]) \neq 0 \in K_{0}(C^{*}(G))$.
\end{thmalph}
\begin{coralph}
Let $X$ and $G$ be as above and additionally we assume that 
$X$ is a spin manifold and the scalar curvature of $X$ is positive.
Then
for any Hermitian $G$-vectorbundle $E$ satisfying the conditions above,
the following higher $G$-$\widehat{\mathcal{A}}$-genus vanishes.
$$
\widehat{\mathcal{A}}_{G}(X;E)
:=
\int_{X} c(x) \widehat{\mathcal{A}}(TX)\wedge \mathrm{ch}(E)
=0,
$$
where $c\in C_c(X)$ denotes an arbitrary cut-off function.
\end{coralph}
\newcommand{\ZZZZ}
{\mu_{\Gaa} \parens{\brackets{\prod L_{1/k}}\hattensor [A]}}
\newcommand{\WWWW}
{\mu_{\Gaa} \parens{[W]\hattensor[A]}}
\newcommand{\XXXX}[1]
{\mu_{\Gaa} \parens{ \brackets{L_{1/{#1}}} \hattensor [A]}}
\begin{proof}
[Proof of the Theorem \ref{MainThm}]
As subsection \ref{SubsectionReductionOnDirac},
we may assume that $A$ is a Dirac type operator.
In this case,
$$
\angles{[E], [A]}_{G} = 
\int_{TX} c(x)  \mathrm{Td}(T_{\C}X)
\wedge \mathrm{ch}(\sigma_A) \wedge \mathrm{ch}(E).
$$
Since $\mathrm{ch}(L^1 \oplus \ldots \oplus L^{N})
= \mathrm{ch}(L^1) \oplus \ldots \oplus \mathrm{ch} (L^{N})$,
it is sufficient to prove when $E$ is 
a single Hermitian $G$-line bundle $L$.
Fix a $G$-invariant connection and let $i\omega$
denote its curvature.
Assume the converse; $\mu_{G}[A] = 0$.
Then it implies $\mu_{\Gaa}[A] = 0$ by Proposition \ref{MuW0}.
Firstly we have an exact sequence;
$$
0 \rightarrow
\dirsum{ \C}
\xrightarrow{\;\;\; \iota\;\;\;}
\prod \C
\xrightarrow{\;\;\; \pi \;\;\;}
Q
\rightarrow 0,
$$
where $\iota$ and $\pi$ are natural inclusion and projection.
Since all of $\dirsum{ \C}$, $\prod \C$ and $Q$
are commutative and hence nuclear,
the tensor product with them are uniquely determined and
we also have the following exact sequence
according to \cite[Theorem T.6.26]{We93};
$$
0 \rightarrow
C^{*}(\Ga)\otimes \parens{\dirsum{ \C}}
\xrightarrow{\;\;\; 1 \otimes \iota\;\;\;}
C^{*}(\Ga)\otimes \parens{\prod \C}
\xrightarrow{\;\;\;1 \otimes \pi \;\;\;}
C^{*}(\Ga)\otimes Q
\rightarrow 0.
$$
Now we have the exact sequence of $K$-groups
$$
K_{0}\parens{C^{*}(\Ga)\otimes \parens{\dirsum{ \C}}}
\xrightarrow{\;\;(1 \otimes \iota)_{*}\;\;}
K_{0}\parens{C^{*}(\Ga)\otimes \parens{\prod \C}}
\xrightarrow{\;\;(1 \otimes \pi)_{*}\;\;}
K_{0}\parens{C^{*}(\Ga)\otimes Q}.
$$
Let us start with $\ZZZZ
\in K_{0}\parens{C^{*}(\Ga)\otimes \parens{\prod \C}}$.
\\
Due to
proposition \ref{PiMu},
$$(1 \otimes \pi)_{*}\ZZZZ = \WWWW = 0.$$
It follows that
there exists $\zeta \in
K_{0}\parens{C^{*}(\Ga)\otimes \parens{\dirsum{ \C}}}$
such that $$(1 \otimes \iota)_{*}(\zeta) = \ZZZZ.$$

Note that $C^{*}(\Ga) \otimes (\dirsum{\C})$
is naturally isomorphic to 
$\dirsum{C^{*}(\Ga)}
= {\displaystyle{\varinjlim_{k}
}} \bigoplus^{k} C^{*}(\Ga)$, which consists of all
$C^{*}(\Ga)$-valued sequences vanishing at infinity, by the map
\begin{eqnarray*}
\begin{array}{cccc}
&C^{*}(\Ga) \otimes (\dirsum{\C})
& \to
& \dirsum{C^{*}(\Ga)}
\\
&z \otimes \parens{ \{a_k\}_{k\in\N} }
& \mapsto
& \braces{a_{k} z}_{k\in \N}
\end{array}
\qqand
\begin{array}{cccc}
& \dirsum{C^{*}(\Ga)}
& \to
& C^{*}(\Ga) \otimes (\dirsum{\C})
\\
& \braces{z_{k}}_{k\in \N}
& \mapsto
& \underset{n \in \N}{\sum}
z_{n} \otimes \parens{ \{ \delta_{k,n} \}_{k \in \N} }.
\end{array}
\end{eqnarray*}
In addition to this, since there is a natural isomorphism
$K_{0}\parens{\dirsum{C^{*}(\Ga)} } \simeq
\bigoplus K_{0}\parens{C^{*}(\Ga)}$
\cite[4.1.15 Proposition, 4.2.3 Remark]{HiRo00},
with the last $\bigoplus$ meaning the algebraic direct sum,
we can consider the following diagram;
$$
\begin{CD}
K_{0}\parens{C^{*}(\Ga)\otimes \parens{\dirsum{ \C}}}
@>{\iota_{*}}>>
K_{0}\parens{C^{*}(\Ga)\otimes \parens{\prod \C}}
@>{(1 \otimes \pi)_{*}}>> 
K_{0}\parens{C^{*}(\Ga)\otimes Q}
\\
@V{\{(1 \otimes p_k)_{*} \} }V{\cong}V
@V{\{(1 \otimes p_k)_{*} \} }VV
@. 
\\
\bigoplus K_{0}(\Ga)
@>>{\text{inclusion}}> 
\prod K_{0}(\Ga).
\end{CD}
$$
Since $p_k = \iota p_k$, this diagram commutes.
Note that both $\bigoplus$ and $\prod$ in the bottom row are
in the algebraic sense.
Again due to the proposition \ref{PiMu},
\begin{eqnarray*}
\braces{\XXXX{k}}_{k \in \N}
&=&
\{(1 \otimes p_k)_{*} \}\parens{\ZZZZ}
\\
&=&
\{(1 \otimes p_k)_{*} \}\parens{(1 \otimes \iota)_{*}(\zeta)}
\\
&=&
\{(1 \otimes p_k)_{*} \}(\zeta)
\\
&\in& \bigoplus K_{0}(\Ga).
\end{eqnarray*}
This implies that all of
$\XXXX{n} \in K_{0}(\Ga)$ are equal to zero
except for finitely many $n \in \N$.

Then apply the trace $\mathrm{tr}_{\Gaa}$ to each $n$-th
component to obtain
$$
\tr_{\Gaa}(p_n)_{*} \parens{\XXXX{n}}
=
\int_{TX} c(x)  \mathrm{Td}(T_{\C}X)
\wedge \mathrm{ch}\wedge(\sigma_{A})
\exp\parens{\frac{-1}{2\pi n} \omega}.
$$
Scince all but finitely many these values are zero,
$$
\int_{TX} c(x)  \mathrm{Td}(T_{\C}X)
\wedge \mathrm{ch}(\sigma_A) \wedge \exp
\parens{\frac{-t}{2\pi}\omega} =0
$$
as a polynomial in $t$ in particular
$$
\int_{TX} c(x) \mathrm{Td}(T_{\C}X)
\wedge \mathrm{ch}(\sigma_A) \wedge \exp
\parens{\frac{-1}{2\pi}\omega}
=
\int_{TX} c(x) \mathrm{Td}(T_{\C}X)
\wedge \mathrm{ch}(\sigma_A) \wedge
\parens{\mathrm{ch}(L)}
=
0.
$$
This contradicts to the assumption.
\end{proof}
\begin{proof}[Proof of the Corollary \ref{MainCor}]
Let the scalar curvature on $X$ be denoted by $\mathrm{Sc}$,
which is assumed to be positive everywhere.
Note that by the bounded geometry of $X$
(corollary \ref{CorSliceThm}),
it is uniformly positive, namely,
$\displaystyle{\inf_{x\in X}} \mathrm{Sc}(x) > 0$.
Let $\mathcal{D}$ be the canonical Dirac operator 
on the spinor bundle $\bb{S}$ over $X$.
Then we have
$$
{\mathcal{D}}^{*} {\mathcal{D}}
=
{\nabla}^{*} {\nabla}
 + \frac{1}{4} {\mathrm{Sc}}
$$
and
since $\displaystyle{\inf_{x \in X}}
{\mathrm{Sc}}(x) > 0$, the operator 
$\displaystyle{
A:= 
\frac{  \mathcal{D}  }
{\sqrt{1+ \mathcal{D} ^2}}
}$
is invertible on $L^{2}(\bb{S})$.
In particular $\mu_{G}([A]) = 0$.
Then due to the theorem \ref{MainThm}, we to obtain
$$
0=
\angles{[E],[A]}_{G}
=
\int_{X} c(x) \widehat{\mathcal{A}}(TX)
\wedge \mathrm{ch}(E)
=
\widehat{\mathcal{A}}_{G}(X;E).
$$
\end{proof}
\section*{Acknowledgments}
The research for my doctoral thesis was supported by 
JSPS KAKENHI Grant Number 13J01329,
and supported by the Kyoto Top Global University Project (KTGU),
by which Professor Gennadi Kasparov 
from Vanderbilt University was
appointed as a Distinguished Visiting Professor in the year 2015.
I would especially like to thank the support from the KTGU Project
that made me possible to visit Vanderbilt University,
and I heartily thank Professor Gennadi Kasparov
for proposing this topic and for having discussions with me.

\end{document}